\pdfoutput=1 

\RequirePackage{fix-cm}
\documentclass[smallextended]{svjour3}
\usepackage[utf8]{inputenc}
\usepackage{booktabs}
\usepackage{empheq}
\usepackage{amssymb,amsmath,graphics,graphicx,subfig}
\usepackage[title]{appendix}
\usepackage{mathrsfs}
\usepackage{tikz}
\usepackage{comment}
\usepackage{eurosym,bm,amsmath}     
\usepackage{scalerel}
\usepackage{enumitem}
\usepackage{multirow}
\definecolor{lightblue}{rgb}{0.22,0.45,0.70}
\definecolor{lightgreen}{rgb}{0.22,0.50,0.25}
\usepackage[right = 2cm, left=2cm, top = 2.5cm, bottom =2.5cm]{geometry}
\usepackage[colorlinks=true,breaklinks=true,linkcolor=lightblue,citecolor=lightblue]{hyperref}

\definecolor{darkred}{rgb}{0.82,0.15,0.20}
\definecolor{darkblue}{rgb}{0.82,0.15,0.12}
\newcommand*{\Scale}[2][4]{\scalebox{#1}{$#2$}} 

\renewenvironment{proof}{\noindent{\it Proof.}}{\hfill$\square$}

\numberwithin{equation}{section}
\numberwithin{figure}{section}
\numberwithin{table}{section}
\numberwithin{lemma}{section}
\numberwithin{corollary}{section}
\numberwithin{theorem}{section}
\numberwithin{remark}{section}
\numberwithin{proposition}{section}

\newcommand{\vertiii}[1]{{\left\vert\kern-0.25ex\left\vert\kern-0.25ex\left\vert #1 
		\right\vert\kern-0.25ex\right\vert\kern-0.25ex\right\vert}}

\allowdisplaybreaks

\begin{document}
	\titlerunning{ VEM for  Boussinesq  equation}   
	\title{A stabilized virtual element framework for the steady state Boussinesq  equation with temperature-dependent parameters}    
	\authorrunning{ S. Mishra, S. Natarajan, Natarajan E.}
	\author{Sudheer Mishra \and Sundararajan Natarajan \and Natarajan E }
	\institute{
		Sudheer Mishra \at Department of Mathematics, Indian Institute of Space Science and Technology, Thiruvananthapuram, 695547, India\\
		\email{sudheermishra.20@res.iist.ac.in}.\\
		Sundararajan Natarajan \at Department of Mechanical Engineering, Indian Institute of Technology Madras, Chennai--600036, India\\
		\email{snatarajan@iitm.ac.in} \\
		Natarajan E \at
		Department of Mathematics, Indian Institute of Space Science and Technology, Thiruvananthapuram, 695547, India \\
		\email{thanndavam@iist.ac.in}.\\
		\and
	}
	\date{}
	\maketitle

\begin{abstract}
		This work presents a new conforming stabilized virtual element method for the generalized Boussinesq equation with temperature-dependent viscosity and thermal conductivity. A gradient-based local projection stabilization method is introduced in the discrete formulation to circumvent the violation of the discrete inf-sup condition. The well-posedness of the continuous problem is established under sufficiently small datum. We derive a stabilized virtual element problem for the Boussinesq equation using equal-order virtual element approximations. The proposed method has several advantages, such as being more straightforward to implement, free from higher-order derivative terms, providing separate stabilization terms without introducing coupling between solution components, and minimizing the number of globally coupled degrees of freedom. The existence of a discrete solution to the stabilized virtual element problem is demonstrated using the Brouwer fixed-point theorem. The error estimates are derived in the energy norm. Additionally, several numerical examples are presented to show the efficiency and robustness of the proposed method, confirming the theoretical results.
	
\end{abstract}


	\section{Introduction}
 \label{sec-1}
Free convection plays a crucial role in various fields, such as science, engineering, industrial and geophysical applications. For instance, it is employed in the cooling of electrical and electronic devices, temperature distribution in the Earth's mantle, petrochemical processes, solar ponds, food processing, heat exchanger design, thermo-protection systems, climate modeling, and the study of ocean currents, and so on. The mathematical description of the above processes is usually given by a non-isothermal thermo-fluid dynamic problem--- the Boussinesq equation, which involves the coupling of the incompressible Navier-Stokes equation with a convection-diffusion transport equation. The coupling occurs through two key mechanisms: the buoyancy force, which appears in the momentum equation and drives the flow due to temperature-induced density variations, and the convective heat transfer term, which captures the transport of thermal energy by the fluid motion. Due to its significance in various applications, several computational methods have been developed, including the Finite Element Method (FEM) for the Boussinesq problem with scalar or temperature-dependent (nonlinear) coefficients. In particular, some novel literature on addressing the scalar and nonlinear coefficients is listed as follows: stream function-vorticity based finite approximations \cite{stream1,stream2,stream3,stream4}, primal formulation based FEM \cite{hiroko1992non,bernardi1995,oyarzua2014exactly,oyarzua2017analysis,dallmann2016stabilized} and mixed-FEM \cite{colmenares2016,colmenares2020banach,almonacid2018,almonacid2020fully,almonacid2020new}. 

To the best of the authors’ knowledge, the work by Bernardi et al. \cite{bernardi1995} represents one of the earliest contributions to the analysis of finite element methods for the Boussinesq problem, where the existence of solutions is demonstrated by employing topological degree theory. Furthermore, optimal-order convergence estimates are derived using the equal-order finite element subspaces for both the velocity vector field and the temperature field. In \cite{ccibik2011projection}, the authors developed a local projection based stabilized FEM for the Boussinesq problem with scalar coefficients, where the stability and optimal-convergence estimates are established using the Taylor-Hood elements for discrete velocity and pressure. In \cite{dallmann2016stabilized}, a combination of streamline-upwind (SUPG)-type stabilizations and the grad-div stabilizations methods is employed to stabilize the Oberbeck-Boussinesq problem (unsteady Boussinesq equation with constant coefficients). Furthermore, a higher-order stabilized FEM derived from the Variational Multiscale Stabilization (VMS) method has been discussed for the transient natural convection problem by Rebollo et al. in \cite{chacon2018high}, employing equal-order approximations. However, the stability of \cite{chacon2018high} is shown for P2-finite elements, whereas stability for the lowest order P1-finite elements has not been investigated. In \cite{allendes2018divergence}, the authors proposed a lowest-order stabilized FEM for nonlinear variable coefficients using the lowest-order Raviart–Thomas space and pressure jump. Notably, almost all works are limited to simple-geometrical shaped elements, such as triangles and quadrilaterals in $2\text{D}$, and {tetrahedrons and hexahedron} in $3\text{D}$.

In the recent past, the development and extension of polygonal-based numerical methods have achieved significant attention from mathematicians, physicists, geologists, and engineers. Virtual Element Method (VEM) was first developed for the Poisson problem in \cite{vem1}, representing an extension of FEM to polygonal and polyhedral elements. In \cite{vem25}, the authors extended \cite{vem1} to a 3D framework, introducing the enhanced nodal virtual element space. In short, the virtual element spaces consist of polynomial and non-polynomial functions equipped with suitable choices of degrees of freedom. The remarkable property of VEM is that it does not require the explicit knowledge of basis functions to compute the stiffness matrix and load vector.  Due to its flexibility, VEM has been extended to address several problems, such as parabolic and hyperbolic equations \cite{vem40,vem32,vem8,vem33}, elasticity \cite{vem2,mvem19}, stabilized VEM \cite{vem4,vem04,vem08,mishra2025supg}, nonconforming VEM \cite{vem26,adak2024nonconforming}, and divergence-free VEM for the Stokes or Navier-Stokes problem \cite{mvem9,mvem11,mvem12,mvem3,adak2021virtual,mvem16}. However, the implementation of VEM can be found in  \cite{vem3,dassi2025vem++}. 

In this work, a novel extension of the classical VEM to $H^1$-conforming stabilized VEM is proposed and analyzed to address the Boussinesq problem with temperature-dependent viscosity and thermal conductivity using equal-order virtual element approximations. The stabilization technique is mainly derived from \cite{dallmann2016stabilized,vem28m}. Recently, in \cite{mvem3}, the authors investigated the divergence-free VEM for the coupled Navier-Stokes-heat equation. In \cite{da2023fully}, the authors discussed a stream-function based VEM formulation for the Boussinesq problem with constant coefficients. Gatica et. al \cite{gatica2024banach} developed a fully mixed-VEM for the Boussinesq problem with constant coefficients, and its extension to the temperature-dependent coefficients case is analyzed in \cite{gharibi2025mixed} using the five-field formulation in combination with mixed-VEM. Based on the aforementioned studies, it can be concluded that most of the VEM/stabilized-FEM literature is limited to the constant coefficient case, and its extension to stabilized VEM has not been explored/developed yet. The proposed method has several advantages: easier to implement, provides separate stabilization terms, reduces the coupling between virtual element triplets, is free from higher-order derivative terms, is efficient for addressing higher-order virtual elements, and minimizes the global degrees of freedom. {In addition, the proposed method can be extended to 3D case, which will be considered as future work.}

The contribution of the proposed method in the field of computational fluid dynamics can be drawn as follows:
\begin{itemize}
	\item the well-posedness of the weak formulation is demonstrated under sufficiently small datum. 
	\item an equivalence relation is established between the gradient-based pressure stabilizing term (introduced in the proposed work) and the local mass-based pressure stabilization technique discussed in \cite{vem028}. In contrast, only the numerical investigation has been carried out in \cite{mishra2024unified} for both stabilization methods.
	\item the existence of a stabilized discrete solution is derived employing the Brouwer fixed-point theorem, and its uniqueness is achieved by restricting the data.
	\item the error estimates are derived in the energy norm with optimal-order convergence rates.
	\item the efficiency and robustness of the proposed method are investigated through several examples, including the convergence studies on convex and non-convex domains, validating the theoretical estimates. 
\end{itemize}    

\subsection{Structure of the paper}
The rest structure of the present work is planned as follows. \ref{sec-2} presents the Boussinesq equation and the well-posedness of its weak formulation. In Section \ref{sec-3}, we propose a stabilized virtual element problem for the weak formulation. Section \ref{sec-4} focuses on the theoretical results for the stabilized virtual element problem, including well-posedness and error estimates in the energy norm. In Section \ref{sec-5}, we present numerical results validating the theoretical aspects. In Section \ref{sec-6}, we will briefly discuss the major conclusions and scope of the proposed work.
-
\section{Model problem } \label{sec-2}
Let {$\Omega \subset \mathbb R^2$} be an open and bounded polygonal domain with a Lipschitz boundary $\partial \Omega$. We consider the following generalized Boussinesq equation, given as follows: find a velocity field $\mathbf{u}$, a pressure field $p$, and a temperature field $\theta$, such that 
\begin{empheq}[left=(P) \empheqlbrace]{align} 
		\quad - \nabla \cdot (\mu(\theta) \nabla \mathbf{u}) + (\nabla \mathbf{u}) \mathbf{u} + \nabla p - \alpha \theta \mathbf{g}  &= \mathbf{f}\quad \text{in} \quad \Omega, \label{stoke-1}\\
		\nabla \cdot \mathbf{u} &= 0 \quad  \text{in} \quad \Omega, \label{stoke-2}\\
		-\nabla \cdot (\kappa(\theta)\nabla \theta) + \mathbf{u} \cdot \nabla \theta &= \mathcal{Q} \quad \text{in} \quad \Omega, \label{heat-1}\\
		\mathbf{u} &= \mathbf{0}  \quad \text{on} \quad  \partial \Omega, \label{stoke-3}\\
		\theta &= \theta_D \quad  \text{on} \quad \partial \Omega, \label{heat-2} 
\end{empheq}
where the functions $\mu : \mathbb{R} \rightarrow \mathbb{R}^{+}$ and $\kappa: \mathbb{R} \rightarrow \mathbb{R}^{+}$ denote temperature-dependent viscosity and thermal conductivity of the fluid, respectively.  Further,  $\mathbf{f}$ and $\mathcal{Q}$ are the external source terms. Additionally, the vector-valued function $\mathbf{g}$ is a given external force per unit mass, $\theta_D$
is a given non-vanishing boundary temperature, and $\alpha$ is a positive constant 
associated with the coefficient of volume expansion.

Throughout this work, we use the standard notations for the Sobolev spaces. {In short, we denote the  $L^2$-inner product and $L^2$-norm by $(\cdot, \cdot)_\Omega$ and $\|\cdot\|_{0,\Omega}$. On any $\omega \subset \Omega$, $W^k_p(\omega)$ denotes the Sobolev space of order $k \in \mathbb N \cup \{0\}$ with the Lebesgue regularity index $p \geq1$. The $W^k_p$-norm and its semi-norm are denoted by $\|\cdot\|_{k,p,\omega}$ and $|\cdot|_{k,p,\omega}$, respectively. We will relax the subscript $\Omega$ in $(\cdot,\cdot)_\Omega$ when no confusion arises.} We adopt bold fonts or bold symbols to represent vector-valued functions or tensors. The positive constant $C$ is independent of the mesh sizes and can differ for different occurrences. The Poincar\'e constant is given by $C_p$. Further, the constant $C_q>0$ depends on $\Omega$, which satisfies 
\begin{align*}
	\| w \|_{1,\Omega} \leq C_q \|\nabla w\|_{0,\Omega} \qquad \text{for all} \,\, w \in W^1_2(\Omega).
\end{align*}  
Note that for $p=2$, the Sobolev space $W^k_p({\Omega})$ reduces to $H^k({\Omega})$. Additionally, the compact embedding $H^1(\Omega) \hookrightarrow L^r(\Omega)$ for $r\geq 1$ is given by 
\begin{align*}
	\| w \|_{0,r,\Omega} \leq C_{1 \hookrightarrow r} \| w\|_{1,\Omega} \qquad \text{for all} \,\, w \in H^1(\Omega),
\end{align*}  
where $C_{1 \hookrightarrow r}$ denotes the compact embedding constant. 

We emphasize that, for simplicity, the homogeneous Dirichlet boundary condition is imposed for the velocity vector field $\mathbf{u}$. However, other boundary conditions can be applied as well. {In particular, when Neumann boundary conditions are imposed the weak formulation involves boundary integrals that can be addressed using the trace inequalities. The core arguments for stability and the convergence estimates are expected to remain valid with different constants, although the proofs become more technical. The extension to Neumann boundary conditions can be incorporated following \cite{colmenares2016dual}.} Hereafter, we suppose that the data of problem $(P)$ satisfy the following assumptions:

\noindent \textbf{(A0) Data assumptions:}
\begin{itemize}
	\item temperature-dependent viscosity $\mu$ is bounded and Lipschitz continuous. More precisely, there exist positive constants $\mu_\ast$, $\mu^\ast$ and $L_\mu$, such that
	\begin{align*}
		0<\mu_\ast \leq \mu(\xi) &\leq \mu^\ast \qquad \,\, \text{for all} \,\, \xi \in \mathbb R,\\
		|\mu(\xi_1)- \mu(\xi_2)| &\leq  L_\mu |\xi_1 - \xi_2|  \qquad \,\, \text{for all} \,\, \xi_1, \xi_2 \in \mathbb R;
	\end{align*}
	\item thermal conductivity $\kappa$ is bounded and Lipschitz continuous. In short, there exist positive constants $\kappa_\ast$, $\kappa^\ast$ and $L_\kappa$, such that
	\begin{align*}
		0<\kappa_\ast \leq \kappa(\xi) &\leq \kappa^\ast \qquad \,\, \text{for all} \,\, \xi \in \mathbb R,\\
		|\kappa(\xi_1)- \kappa(\xi_2)| &\leq  L_\kappa |\xi_1 - \xi_2|  \qquad \,\, \text{for all} \,\, \xi_1, \xi_2 \in \mathbb R;
	\end{align*}
	\item external loads satisfy $\mathbf{g} \in [L^2(\Omega)]^2$, $\mathbf{f} \in [L^2(\Omega)]^2$ and $\mathcal{Q} \in L^2(\Omega)$;
	\item the datum satisfies $\theta_D \in H^{1/2}(\partial \Omega)$.
\end{itemize}

\noindent We now introduce the following functional setting for the Boussinesq equation:  
	\begin{center}
		$\mathbf{V} := [H_0^1(\Omega)]^2$,  \qquad  $Q := L_0^2(\Omega) = \big\{ p \in L^2(\Omega): \int_{\Omega} p d \Omega = 0 \big\}$, \qquad \quad \,\,\,\,\hfill \hfill \\ $\Sigma:=H^1(\Omega)$, \qquad $\Sigma_0:=H_0^1(\Omega)$,  \qquad $\Sigma_D:=\big\{
		\phi \in H^1(\Omega): \phi_{|\partial\Omega} = \theta_D\}$,
	\end{center}
endowed with their natural norms. In addition, we define the following forms:
\begin{itemize}
	\item $a_V(\cdot; \cdot,\cdot): \Sigma \times \mathbf{V} \times \mathbf{V}  \rightarrow \mathbb R$, \quad \,\,\, $a_V(\phi;\mathbf{v},\mathbf{w}):= \int_{\Omega} \mu(\phi) \nabla \mathbf{v}: \nabla \mathbf{w}\, d\Omega$;
	\item $b(\cdot,\cdot):\mathbf{V} \times Q  \rightarrow \mathbb R$, \qquad \qquad \qquad \quad \,\, $b(\mathbf{v},q):= \int_{\Omega} (\nabla \cdot \mathbf{v}) q\, d\Omega$;
	\item $c_V(\cdot; \cdot,\cdot): \mathbf{V} \times \mathbf{V} \times \mathbf{V} \rightarrow \mathbb R$, \qquad  $c_V(\mathbf{v}; \mathbf{w},\mathbf{z}):= \int_{\Omega} (\nabla \mathbf{w})\mathbf{v} \cdot \mathbf{z} \, d\Omega$;  
	\item $a_T(\cdot; \cdot,\cdot):\Sigma \times \Sigma \times \Sigma \rightarrow \mathbb R$, \qquad  \,\, $a_T(\theta;\phi,\psi):= \int_{\Omega} \kappa(\theta) \nabla \phi \cdot \nabla \psi \, d\Omega$;
	\item $c_T(\cdot; \cdot,\cdot): \mathbf{V} \times \Sigma \times \Sigma  \rightarrow \mathbb R$, \qquad \, $c_T(\mathbf{v}; \phi,\psi):= \int_{\Omega} (\mathbf{v} \cdot \nabla \phi) \psi \, d\Omega$;
	\item For given $\phi \in \Sigma$, define $F_\phi: \mathbf{V} \rightarrow \mathbb R$, \, $F_\phi(\mathbf{v}):= \int_{\Omega} (\alpha \mathbf{g}\phi)\cdot \mathbf{v} \, d\Omega + \int_{\Omega} \mathbf{f} \cdot \mathbf{v} \, d\Omega$.
\end{itemize}

We remark that under the assumption \textbf{(A0)}, the above forms are well-defined. For any $\phi \in \Sigma$, the maps $a_V(\phi; \cdot,\cdot)$ and $a_T(\phi; \cdot,\cdot)$ are bilinear forms. In addition, the forms above satisfy the following stability properties: 
\begin{itemize}
	\item continuity and coercivity of $a_V(\cdot; \cdot, \cdot)$: for any given $\phi \in \Sigma$, we have the following
	\begin{align}
		a_V(\phi; \mathbf{v}, \mathbf{w}) \leq \mu^\ast \|\nabla \mathbf{v}\|_{0,\Omega} \|\nabla \mathbf{w}\|_{0,\Omega}, \qquad a_V(\phi; \mathbf{w}, \mathbf{w}) \geq \mu_\ast \|\nabla \mathbf{w}\|^2_{0,\Omega}, \quad \,\, \text{for all} \,\, \mathbf{v}, \mathbf{w} \in \mathbf{V}. 	\label{con-av}
	\end{align}
	\item continuity and coercivity of $a_T(\cdot; \cdot, \cdot)$: for any given $\theta \in \Sigma$, we obtain
	\begin{align}
		a_T(\theta; \phi, \psi) \leq \kappa^\ast \|\nabla \phi\|_{0,\Omega} \|\nabla \psi\|_{0,\Omega}, \qquad a_T(\theta; \phi, \phi) \geq \kappa_\ast \|\nabla \phi\|^2_{0,\Omega}, \quad \,\, \text{for all} \,\, \phi, \psi \in \Sigma. 	\label{con-at}
	\end{align}
	
	\item bilinear form \( b(\cdot, \cdot) \) satisfies the continuity property, ensuring boundedness,  
	\begin{align}
		|b(\mathbf{v}, q)| \leq  {\sqrt{2}} \|\nabla \mathbf{v}\|_{0,\Omega} \|q\|_{0,\Omega} \qquad \,\, \text{for all} \,\, \mathbf{v} \in \mathbf{V} \,\, \text{and} \,\, q \in Q. \label{con-b}
	\end{align}
	Additionally, it satisfies the inf-sup condition   
	\begin{align}
		\sup_{\mathbf{0} \neq \mathbf{v} \in \mathbf{V}} \dfrac{b(\mathbf{v}, q)}{\|\mathbf{v}\|_{1,\Omega}} \geq \beta \|q\|_{0,\Omega}, \label{con-binf}
	\end{align}
	where  $\beta>0$ represents the inf-sup constant, for more details see~\cite{bookgirault}.
	
	\item for any $\mathbf{v} \in \mathbf{V}$ with $\nabla \cdot \mathbf{v}=0$, the bilinear forms $c_V(\mathbf{v}; \cdot, \cdot)$ and $c_T(\mathbf{v}; \cdot, \cdot)$ satisfy the skew-symmetric property, 
	\begin{align}
		c_V(\mathbf{v}; \mathbf{w}, \mathbf{z}) &= -	c_V(\mathbf{v}; \mathbf{z}, \mathbf{w}) \qquad \,\, \text{for all} \,\, \mathbf{w}, \mathbf{z} \in \mathbf{V}, \label{skew-v}\\ 
		c_T(\mathbf{v}; \phi, \psi) &= - c_T(\mathbf{v}; \psi, \phi) \qquad \,\, \text{for all} \,\, \phi, \psi \in \Sigma. \label{skew-t}
	\end{align}
	Therefore, we introduce the following skew-symmetric forms corresponding to convective terms as follows:
	\begin{align}
		c^S_V(\mathbf{v}; \mathbf{w}, \mathbf{z}) &= \frac{1}{2} \Big[	c_V(\mathbf{v}; \mathbf{w}, \mathbf{z}) - 	c_V(\mathbf{v}; \mathbf{z}, \mathbf{w}) \Big]  \qquad \,\, \text{for all} \,\, \mathbf{w}, \mathbf{z} \in \mathbf{V}, \\ 
		c^S_T(\mathbf{v}; \phi,\psi)&= \frac{1}{2} \Big[c_T(\mathbf{v}; \phi,\psi) - c_T(\mathbf{v}; \psi, \phi) \Big] \qquad \,\, \text{for all} \,\, \phi, \psi \in \Sigma.
	\end{align}
	Additionally, they are continuous, i.e., 
	\begin{align}
		c^S_V(\mathbf{v}; \mathbf{w}, \mathbf{z}) & \leq N_0 \|\nabla \mathbf{v}\|_{0,\Omega} \| \nabla \mathbf{w}\|_{0,\Omega} \| \nabla \mathbf{z}\|_{0,\Omega}  \qquad \,\, \text{for all} \,\, \mathbf{w}, \mathbf{z} \in \mathbf{V}, \label{cnt-sv}\\ 
		c^S_T(\mathbf{v}; \phi,\psi)&\leq N_0 \|\nabla \mathbf{v}\|_{0,\Omega} \|\nabla \phi\|_{0,\Omega} \|\nabla \psi\|_{0,\Omega} \qquad \,\, \text{for all} \,\, \phi \in \Sigma, \psi \in \Sigma_0, \label{cnt-st}
	\end{align}
	where $N_0>0$ is the continuity constant. 
\end{itemize}  
Finally, we introduce the following forms:
\begin{align*}
	A_V[\theta; (\mathbf{v},q), (\mathbf{w}, r)]&:= a_V(\theta;\mathbf{v},\mathbf{w}) - b(\mathbf{w},q) + b(\mathbf{v},r) \qquad \, \text{for all} \,\, \theta \in \Sigma \,\, \text{and} \,\, (\mathbf{v},q), (\mathbf{w}, r) \in \mathbf{V} \times Q, \\
	A_T ( \theta, \mathbf{v}; \phi, \psi )&:= a_T(\theta; \phi,\psi) +  c^S_T(\mathbf{v}; \phi,\psi) \qquad  \qquad  \quad \text{ for all} \,\, \mathbf{v} \in \mathbf{V} \,\, \text{and}\,\, \theta, \phi, \psi \in \Sigma.
\end{align*}
The primal variational formulation of the Boussinesq equation $(P)$ is given by 
\begin{align}
	\begin{cases}
		\text{find } (\mathbf{u}, p, \theta) \in \mathbf{V} \times Q \times \Sigma_D, \text{ such that} \\[1ex]
		\begin{aligned}
			a_V(\theta;\mathbf{u},\mathbf{v}) - b(\mathbf{v},p) + c^S_V(\mathbf{u}; \mathbf{u}, \mathbf{v}) &= F_\theta(\mathbf{v}) 
			&& \quad \text{for all } \mathbf{v} \in \mathbf{V}, \\ 
			b(\mathbf{u}, q ) &= 0 
			&& \quad \text{for all } q \in Q,\\
			A_T(\theta, \mathbf{u}; \theta, \psi) &= (\mathcal{Q}, \psi) 
			&& \quad \text{for all } \psi \in \Sigma_0.
		\end{aligned}
	\end{cases}
	\label{variation-1}
\end{align}

\begin{remark}
	An equivalent primal formulation of the problem $(P)$ can be given as follows:
	\begin{align}
		\begin{cases}
			\text{find } (\mathbf{u}, p, \theta) \in \mathbf{V} \times Q \times \Sigma_D, \text{ such that} \\[1ex]
			\begin{aligned}
				A_V[\theta;(\mathbf{u},p),(\mathbf{v},q)] + c^S_V(\mathbf{u}; \mathbf{u}, \mathbf{v}) &= F_\theta(\mathbf{v}) 
				&& \quad \text{for all } \,\, (\mathbf{v},q) \in \mathbf{V} \times Q, \\ 
				A_T(\theta, \mathbf{u}; \theta, \psi) &= (\mathcal{Q}, \psi) 
				&& \quad \text{for all } \psi \in \Sigma_0.
			\end{aligned}
		\end{cases}
		\label{variation-01}
	\end{align}
\end{remark}
\textbf{Decoupled Problem:} The decoupled form of the problem \eqref{variation-01} can be given as follows: \newline
$\bullet$ For given $\widehat{\theta} \in \Sigma$ and $\widehat{\mathbf{u}} \in \mathbf{V}$, find $(\mathbf{u},p) \in \mathbf{V} \times Q$ such that
\begin{align}
	A_V[\widehat{\theta};(\mathbf{u},p),(\mathbf{v},q)] + c^S_V(\widehat{\mathbf{u}}; \mathbf{u}, \mathbf{v}) = F_{\widehat{\theta}}(\mathbf{v}) \qquad \text{for all } \,\, (\mathbf{v},q) \in \mathbf{V} \times Q. \label{dvar-1}
\end{align}
$\bullet$ For given $\widehat{\theta} \in \Sigma$ and $\widehat{\mathbf{u}} \in \mathbf{V}$, find $\theta \in \Sigma_D$ such that
\begin{align}
	A_T(\widehat{\theta}, \widehat{\mathbf{u}}; \theta, \psi) &= (\mathcal{Q}, \psi)  \qquad \text{for all } \psi \in \Sigma_0. \label{dvar-2}
\end{align}

\subsection{Existence of a solution}
Hereafter, we first introduce the following space
\begin{align}
	\mathbf{V}_{div} := \big\{\mathbf{v} \in [H^1_0(\Omega)]^2\,\, \text{such that}\,\, \nabla\cdot \mathbf{v}=0 \big\}.
\end{align}
On the space \( \mathbf{V}_{{div}} \), we can obtain a reduced problem, which is equivalent to the primal formulation \eqref{variation-1} or \eqref{variation-01}, and is given by:
\begin{align}
	\begin{cases}
		\text{find } (\mathbf{u}, \theta) \in \mathbf{V}_{div} \times \Sigma_D, \text{ such that} \\[1ex]
		\begin{aligned}
			a_V( \theta; \mathbf{u}, \mathbf{v}) + c^S_V(\mathbf{u}; \mathbf{u}, \mathbf{v}) &= F_\theta(\mathbf{v}) 
			&& \quad \text{for all } \mathbf{v} \in \mathbf{V}_{div}, \\ 
			A_T(\theta, \mathbf{u}; \theta, \psi) &= (\mathcal{Q}, \psi) 
			&& \quad \text{for all } \psi \in \Sigma_0.
		\end{aligned}
	\end{cases}
	\label{variation-2}
\end{align}
Furthermore, we introduce the following decomposition of the temperature $\theta \in \Sigma_D$:
\begin{align}
	\theta= \theta_0 + \theta_1 \quad \text{with}\quad \theta_0 \in \Sigma_0 \quad \text{and} \quad \theta_1 \in \Sigma \quad \text{such that} \,\, {\theta_1}_{| \partial \Omega}= \theta_D.
\end{align}

The existence of a continuous solution is given as follows:
\begin{proposition} \label{existenceP}
	Under the assumption \textbf{(A0)}, the primal formulation \eqref{variation-1} has at least a solution $(\mathbf{u},p,\theta) \in \mathbf{V} \times Q \times \Sigma_D$ such that it satisfies
	\begin{align}
		\|\nabla \mathbf{u}\|^2_{0,\Omega} + \|\nabla \theta\|^2_{0,\Omega} &\leq C_{exist}\Big[ (1+ c^2 C^{-1}_{exist}) \|\theta_D\|^2_{H^{1/2}(\partial\Omega)} + \|\mathbf{f}\|^2_{0,\Omega} + \|\mathcal{Q}\|^2_{0,\Omega} \,+ \nonumber \\ &  \qquad \qquad \quad \big(\|\theta_D\|^2_{H^{1/2}(\partial \Omega)} +    \|\mathcal{Q}\|^2_{0, \Omega}  + 1 \big)  \|\mathbf{g}\|^2_{0,\Omega} \Big], \label{exist-0}
	\end{align}
	where the constants $c$ depends on $\Omega$, and $C_{exist}$ is given by \eqref{cexist}.
\end{proposition}

\begin{proof}
	We will prove the existence of a continuous solution in the following steps: \newline
	\textbf{Step 1.} {Following \cite[Lemma 3.2]{colmenares2016dual}, for any $\epsilon \in (0,1)$}, there exists an extension $\theta_1 \in \Sigma$ of $\theta_D \in H^{1/2}(\partial \Omega)$ such that
	\begin{align}
		\|\nabla \theta_1\|_{0,\Omega} \leq c \epsilon^{-4} \| \theta_D\|_{H^{1/2}(\partial \Omega)}, \qquad \| \theta_1\|_{0,3,\Omega} \leq c\epsilon  \| \theta_D\|_{H^{1/2}(\partial \Omega)},  \label{hopf}
	\end{align}
	where the positive constant $c$ does not rely on $\epsilon$.\newline
	Further, using the Riesz representation theorem, we define a map $\chi:\mathbf{V}_{div} \times \Sigma_0 \rightarrow (\mathbf{V}_{div} \times \Sigma_0)'$ such that for any fixed $(\mathbf{u}, \theta_0) \in \mathbf{V}_{div} \times \Sigma_0$, we have
	\begin{align}
		\langle \chi(\mathbf{u}, \theta_0), (\mathbf{v}, \psi) \rangle &= 	a_V(\theta_0+ \theta_1; \mathbf{u}, \mathbf{v}) + c^S_V(\mathbf{u}; \mathbf{u}, \mathbf{v}) +	A_T(\theta_0+\theta_1, \mathbf{u}; \theta_0+ \theta_1, \psi) \nonumber \\ & \qquad - F_{\theta_0+\theta_1}(\mathbf{v}) - (\mathcal{Q}, \psi) \qquad \text{for all }\,\, (\mathbf{v}, \psi) \in \mathbf{V}_{div} \times \Sigma_0.
	\end{align}
	\textbf{Step 2.} Continuity of $\chi$:  Let $\big(\mathbf{u}^n, \theta^n_0\big)_{n \in \mathbb N }$ be any sequence of functions in $\mathbf{V}_{div} \times \Sigma_0$ such that $\big(\mathbf{u}^n, \theta^n_0\big) \longrightarrow (\mathbf{u}, \theta_0)$ strongly in $\mathbf{V} \times \Sigma$. We prove that $	\langle \chi(\mathbf{u}^n, \theta_0^n), (\mathbf{v}, \psi) \rangle \longrightarrow	\langle \chi(\mathbf{u}, \theta_0), (\mathbf{v}, \psi) \rangle$. \newline
	$\bullet$ Recalling the assumption \textbf{(A0)}, we can conclude that for any $\mathbf{v} \in \mathbf{V}_{div}$,
	\begin{align}
		\mu(\theta^n_0+\theta_1)\nabla \mathbf{v} \longrightarrow \mu(\theta_0+ \theta_1) \nabla \mathbf{v} \quad\text{a. e. in $\Omega$ \qquad  and } \qquad \|\mu(\theta^n_0+\theta_1) \nabla \mathbf{v}\|_{0,\Omega} \leq \mu^\ast \|\nabla \mathbf{v}\|_{0,\Omega}. 
	\end{align}
	Thus, employing the Lebesgue dominated convergence theorem, we obtain
	\begin{align}
		\lim\limits_{n \rightarrow \infty}\mu(\theta^n_0+\theta_1) \nabla \mathbf{v}=\mu(\theta_0+ \theta_1) \nabla \mathbf{v} \qquad \text{strongly in $\mathbf{V}$ for all $\mathbf{v} \in \mathbf{V}_{div}$}.
	\end{align}
	Therefore,
	\begin{align}
		\lim\limits_{n \rightarrow \infty} a_V(\theta^n_0+\theta_1; \mathbf{u}^n, \mathbf{v}) =a_V(\theta_0+\theta_1; \mathbf{u}, \mathbf{v}).
	\end{align}
	Similarly, using \eqref{con-at}, we can easily obtain 
	\begin{align}
		\lim\limits_{n \rightarrow \infty} a_T(\theta^n_0+\theta_1; \theta^n_0 + \theta_1, \psi) =a_T(\theta_0+\theta_1; \theta_0 + \theta_1, \psi).
	\end{align}
	$\bullet$ Due to strong convergence of $\big(\mathbf{u}^n, \theta^n_0\big)_{n \in \mathbb N }$ in $\mathbf{V} \times \Sigma$ allows us to pass the limit in trilinear forms $c_V^S(\cdot; \cdot, \cdot)$ and $c_T^S(\cdot; \cdot, \cdot)$, see \cite[Chapter IV, Theorem 2.1]{bookgirault}. Furthermore, the convergence in the linear part of $F_{\theta^n_0+\theta_1}(\mathbf{v})$ is direct. Hence, $\chi$ is continuous.
	
	\noindent \textbf{Step 3.} A compact and convex ball: We now aim to find a convex and compact ball where the hypothesis of the Brouwer fixed-point theorem can be fulfilled. In this sequel, we first replace $(\mathbf{v}, \psi)$ by $(\mathbf{u}, \theta_0)$, and then employ the stability properties \eqref{con-av}, \eqref{con-at}, \eqref{skew-v} and \eqref{skew-t}, it holds that 
	\begin{align}
		\langle \chi(\mathbf{u}, \theta_0), (\mathbf{u}, \theta_0) \rangle &\geq \mu_\ast \|\nabla \mathbf{u}\|^2_{0,\Omega} + \kappa_\ast \|\nabla \theta_0\|^2_{0,\Omega}+ a_T( \theta_0+ \theta_1; \theta_1, \theta_0) + c^S_T( \mathbf{u}; \theta_1, \theta_0) - F_{\theta_0+ \theta_1}(\mathbf{u}) - (\mathcal{Q}, \theta_0) \nonumber \\
		&\geq \mu_\ast \|\nabla \mathbf{u}\|^2_{0,\Omega} + \kappa_\ast \|\nabla \theta_0\|^2_{0,\Omega}- \kappa^\ast \|\nabla \theta_1\|_{0,\Omega} \|\nabla\theta_0\|_{0,\Omega} - c_T( \mathbf{u}; \theta_0, \theta_1) \nonumber \\ & \qquad- F_{\theta_0+ \theta_1}(\mathbf{u}) - (\mathcal{Q}, \theta_0) \nonumber \\
		&\geq \mu_\ast \|\nabla \mathbf{u}\|^2_{0,\Omega} + \kappa_\ast \|\nabla \theta_0\|^2_{0,\Omega}- \kappa^\ast \|\nabla \theta_1\|_{0,\Omega} \|\nabla\theta_0\|_{0,\Omega} - {\|\mathbf{u}\|_{0,6, \Omega} \| \nabla \theta_0\|_{0,\Omega} \|\theta_1\|_{0,3,\Omega}} \nonumber \\ & \qquad  - \alpha \|\mathbf{g}\|_{0,\Omega} \|\theta_0\|_{0,4,\Omega} \|\mathbf{u}\|_{0,4,\Omega} - \alpha \|\mathbf{g}\|_{0,\Omega} \|\theta_1\|_{0,3,\Omega} \|\mathbf{u}\|_{0,6,\Omega} \nonumber \\ & \qquad -  \|\mathbf{f}\|_{0,\Omega}\| \mathbf{u}\|_{0,\Omega}- \|\mathcal{Q}\|_{0,\Omega}\|\theta_0\|_{0,\Omega}. \nonumber
	\end{align}
	We use the Sobolev embedding theorem, the Poincar$\acute{\text{e}}$ inequality and the bound \eqref{hopf}:
	\begin{align}
		\langle \chi(\mathbf{u}, \theta_0), (\mathbf{u}, \theta_0) \rangle 	&\geq \min \big\{\mu_\ast, \kappa_\ast  \big\} \big(\|\nabla \mathbf{u}\|^2_{0,\Omega} + \|\nabla \theta_0\|^2_{0,\Omega} \big)- {c \kappa^\ast \epsilon^{-4} \|\nabla\theta_0\|_{0,\Omega} \|\theta_D\|_{H^{1/2}(\partial\Omega)} }\nonumber \\ & \,\, - c\epsilon C_q C_{1 \hookrightarrow 6} \|\nabla \mathbf{u}\|_{0, \Omega} \| \nabla \theta_0\|_{0,\Omega} \|\theta_D\|_{H^{1/2}(\partial\Omega)}   - \alpha C^2_q C^2_{1 \hookrightarrow 4} \|\mathbf{g}\|_{0,\Omega} \|\nabla \theta_0\|_{0,\Omega} \|\nabla \mathbf{u}\|_{0,\Omega} \nonumber \\ & \,\, - c \epsilon\alpha C_q C_{1 \hookrightarrow 6} \|\mathbf{g}\|_{0,\Omega} \|\nabla \mathbf{u}\|_{0,\Omega} \|\theta_D\|_{H^{1/2}(\partial\Omega)} - C_p \|\mathbf{f}\|_{0,\Omega}\| \nabla \mathbf{u}\|_{0,\Omega}- C_p \|\mathcal{Q}\|_{0,\Omega}\| \nabla \theta_0\|_{0,\Omega}. \nonumber
	\end{align}
	For sufficiently small $\epsilon$, we choose 
	\begin{align}
		c\epsilon C_q C_{1 \hookrightarrow 6} \|\theta_D\|_{H^{1/2}(\partial\Omega)} \leq   \min \big\{\mu_\ast, \kappa_\ast  \big\}:=C_o. \label{exist-co}
	\end{align}
	{Employing the Young inequality, we arrive at 
		\begin{align}
			\langle \chi(\mathbf{u}, \theta_0), (\mathbf{u}, \theta_0) \rangle 	&\geq \|\nabla \mathbf{u}\|_{0,\Omega} \Big[\frac{C_o}{2} \|\nabla \mathbf{u}\|_{0,\Omega} -  \alpha C^2_q C^2_{1 \hookrightarrow 4}\|\mathbf{g}\|_{0,\Omega} \|\nabla \theta_0\|_{0,\Omega} - C_o\alpha  \|\mathbf{g}\|_{0,\Omega}  - C_p \|\mathbf{f}\|_{0,\Omega} \Big]\, + \nonumber \\ & \qquad \|\nabla \theta_0\|_{0,\Omega} \Big[ \frac{C_o}{2} \|\nabla \theta_0\|_{0,\Omega} -c \kappa^\ast \epsilon^{-4} \|\theta_D\|_{H^{1/2}(\partial\Omega)} -   C_p \|\mathcal{Q}\|_{0,\Omega}\Big]. 
		\end{align}
		To obtain the non-negativity condition, we proceed as follows
		\begin{align}
			\|\nabla \mathbf{u}\|_{0,\Omega} &\geq \frac{2}{C_o} \big[\alpha C^2_q C^2_{1 \hookrightarrow 4}\|\mathbf{g}\|_{0,\Omega} \|\nabla \theta_0\|_{0,\Omega} + C_o\alpha  \|\mathbf{g}\|_{0,\Omega} + C_p \|\mathbf{f}\|_{0,\Omega} \Big], \label{exn1}\\ 
			\|\nabla \theta_0\|_{0,\Omega} &\geq \frac{2}{ C_o} \Big[c \kappa^\ast \epsilon^{-4} \|\theta_D\|_{H^{1/2}(\partial\Omega)} + C_p \|\mathcal{Q}\|_{0,\Omega} \Big] \equiv F_\theta(\theta_D, \mathcal{Q}). \label{exn2}
		\end{align}
		Using \eqref{exn2} in \eqref{exn1}, we obtain
		\begin{align}
			\|\nabla \mathbf{u}\|_{0,\Omega} &\geq \frac{2}{C_o} \Big[  \big(2c \kappa^\ast \epsilon^{-4}\alpha C_o^{-1} C^2_q C^2_{1 \hookrightarrow 4} \|\theta_D\|_{H^{1/2}(\partial \Omega)} + 2\alpha C_p C_o^{-1} C^2_q C^2_{1 \hookrightarrow 4} \|\mathcal{Q}\|_{0, \Omega}  + C_o\alpha \big)  \|\mathbf{g}\|_{0,\Omega} \nonumber \\
			&\qquad + C_p \|\mathbf{f}\|_{0,\Omega} \Big] \nonumber \\
			& \equiv F_{\mathbf{u}}(\theta_D, \mathcal{Q}, \mathbf{g}, \mathbf{f}). \label{exnf1}
	\end{align}}
	We now choose $\rho >0$ such that $ \rho^2 \geq F^2_\theta(\theta_D, \mathcal{Q}) + F^2_{\mathbf{u}}(\theta_D, \mathcal{Q}, \mathbf{g}, \mathbf{f})$, and we will determine its value later. Then, we conclude that
	\begin{align}
		\langle \chi(\mathbf{u}, \theta_0), (\mathbf{u}, \theta_0) \rangle  \geq 0 \qquad \text{for all} \,\, (\mathbf{u}, \theta_0) \in \mathbf{V}_{div} \times \Sigma_0 \quad \text{such that} \quad \|\nabla \mathbf{u}\|^2_{0,\Omega} + \|\nabla \theta_0\|^2_{0,\Omega} = \rho^2. \label{eee1}
	\end{align}
	
	\noindent \textbf{Step 4.} Finite approximation of $\mathbf{V}_{div} \times \Sigma_0$: Since the spaces $\mathbf{V}_{div}$ and  $\Sigma_0$ are separable, we can construct a monotonic sequence $\{\mathbf{V}^r_{div}\}_{r \in \mathbb N}$ of finite-dimensional subspaces of $\mathbf{V}_{div}$, and an increasing sequence $\{ \Sigma^r_0\}_{r \in \mathbb N }$ of finite-dimensional subspaces of $\Sigma_0$ such that $\bigcup \limits_{r\in \mathbb{N}} (\mathbf{V}^r_{div}\times \Sigma^r_0)$ is dense in $\mathbf{V}_{div} \times \Sigma_0$. Furthermore, we emphasize that for each $r \in \mathbb N$, the  map $\chi:(\mathbf{V}^r_{div}\times \Sigma^r_0) \rightarrow (\mathbf{V}^r_{div}\times \Sigma^r_0)^\prime$ is continuous and satisfies the following: 
	\begin{align}
		\Scale[0.97]{	\langle \chi(\mathbf{u}_r, \theta_{0,r}), (\mathbf{u}_r, \theta_{0,r}) \rangle  \geq 0 \quad \text{for all} \,\, (\mathbf{u}_r, \theta_{0,r}) \in \mathbf{V}^r_{div} \times \Sigma^r_0 \,\,\, \text{such that} \,\,\, \|\nabla \mathbf{u}_r\|^2_{0,\Omega} + \|\nabla \theta_{0,r}\|^2_{0,\Omega} = \rho^2.}
	\end{align}
	Thus, employing the Brouwer fixed-point theorem, for each $r \in \mathbb N$, there exists $(\mathbf{u}_r,\theta_{0,r}) \in \mathbf{V}^r_{div}\times \Sigma^r_0$ such that 
	\begin{align}
		\langle \chi(\mathbf{u}_r, \theta_{0,r}), (\mathbf{v}_r, \psi_{r}) \rangle = 0 \qquad \text{for all} \,\, (\mathbf{v}_r,\psi_r) \in \mathbf{V}^r_{div}\times \Sigma^r_0 \,\,\,\,\text{and} \,\,\,\,  \|\nabla \mathbf{u}_r\|^2_{0,\Omega} + \|\nabla \theta_{0,r}\|^2_{0,\Omega} \leq \rho^2. \label{ext-a}
	\end{align}
	\textbf{Step 5.} Limiting case: From \eqref{ext-a}, the sequence $\{ \mathbf{u}_r, \theta_{0,r}\}_{r\in \mathbb N}$ is uniformly bounded in $\mathbf{V} \times \Sigma_0$. Employing the compact embedding theorem, there exists a subsequence, which we denote by $\{ \mathbf{u}_r, \theta_{0,r}\}_{r\in \mathbb N}$ such that  $\{ \mathbf{u}_r, \theta_{0,r}\}_{r\in \mathbb N} \xlongrightarrow{\text{w}}  \{ \mathbf{u}, \theta_0\}$ weakly in $\mathbf{V} \times \Sigma_0$ and  $\{ \mathbf{u}_r, \theta_{0,r}\}_{r\in \mathbb N} \xlongrightarrow{\text{s}}  \{ \mathbf{u}, \theta_0\}$ strongly in $[L^q(\Omega)]^2 \times L^q(\Omega)$ for $q \geq 1$. Furthermore, for $m\leq r$, we obtain
	\begin{align}
		\langle \chi(\mathbf{u}_r, \theta_{0,r}), (\mathbf{v}_m, \psi_{m}) \rangle = 0 \qquad \text{for all} \,\,  (\mathbf{v}_m,\psi_m) \in \mathbf{V}^m_{div}\times \Sigma^m_0. \label{ext-b}
	\end{align}
	As $r \longrightarrow \infty$, we observe the following: \newline
	$\bullet$ $F_{\theta_{0,r} + \theta_1}(\mathbf{v}_m)- F_{\theta_{0} + \theta_1}(\mathbf{v}_m)= \big(\alpha \mathbf{g}(\theta_{0,r}-\theta_{0}), \mathbf{v}_m\big) \leq  \alpha \|\mathbf{g}\|_{0,\Omega} \|\theta_{0,r}-\theta_{0}\|_{0,4,\Omega} \|\mathbf{v}_m\|_{0,4,\Omega}$. \newline Thus, $F_{\theta_{0,r} + \theta_1}(\mathbf{v}_m) \longrightarrow F_{\theta_{0} + \theta_1}(\mathbf{v}_m)$ as $r \longrightarrow \infty$. \\
	$\bullet$ the trilinear form $c^S_V(\mathbf{u}_r; \mathbf{u}_r, \mathbf{v}_m) \longrightarrow c^S_V(\mathbf{u}; \mathbf{u}, \mathbf{v}_m) $ as $r \longrightarrow \infty$ derived in \cite[Chapter IV, Theorem 2.1]{bookgirault}. Following the same, we obtain $c^S_T(\mathbf{u}_r; \theta_{0,r}+\theta_1, \psi_m) \longrightarrow c^S_T(\mathbf{u}; \theta_0+\theta_1, \mathbf{\psi}_m) $ as $r \longrightarrow \infty$. \newline
	$\bullet$ applying $\theta_{0,r} \xlongrightarrow{\text{s}} \theta_{0}$ strongly in $L^q(\Omega)$ and the continuity of $\mu$, gives that 
	\begin{align}
		\mu(\theta_{0,r}+ \theta_1)\nabla \mathbf{v}_m \xrightarrow{r\rightarrow \infty} \mu(\theta_0+ \theta_1)\nabla \mathbf{v}_m \quad \text{a.e. in} \, \Omega, \qquad 	\mu(\theta_{0,r}+ \theta_1)\nabla \mathbf{v}_m \leq \mu^\ast \|\nabla \mathbf{v}_m\|_{0,\Omega}. \nonumber
	\end{align} 
	Therefore, employing the Lebesgue dominated convergence theorem, we conclude that $a_V(\theta_{0,r}+\theta_1; \mathbf{u}_r, \mathbf{v}_m)$ converges to $a_V(\theta_{0}+\theta_1; \mathbf{u}, \mathbf{v}_m)$ as $r \longrightarrow \infty$. Following the above, we obtain $a_T(\theta_{0,r}+\theta_1; \theta_{0,r}+\theta_1, \psi_m)$ converges to $a_T(\theta_{0}+\theta_1; \theta_{0}+\theta_1, \psi_m)$ as $r \longrightarrow \infty$.
	Finally, combining all the above, we obtain as $r \longrightarrow \infty$
	\begin{align}
		\langle \chi(\mathbf{u}, \theta_{0}), (\mathbf{v}_m, \psi_{m}) \rangle = 0 \qquad \text{for all} \,\,  (\mathbf{v}_m,\psi_m) \in \mathbf{V}^m_{div}\times \Sigma^m_0.  \nonumber
	\end{align}
	Further, using the density argument, we get
	\begin{align}
		\langle \chi(\mathbf{u}, \theta_{0}), (\mathbf{v}, \psi) \rangle = 0 \qquad \text{for all} \,\,  (\mathbf{v},\psi) \in \mathbf{V}_{div}\times \Sigma_0.  \nonumber
	\end{align}
	Thus, the pair $(\mathbf{u}, \theta=\theta_0 + \theta_1)$ satisfies the second and third equation of \eqref{variation-1}. Additionally, it also satisfies
	\begin{align}
		\varPhi(\mathbf{v}):= a_V(\theta; \mathbf{u}, \mathbf{v}) + c^S_V(\mathbf{u}; \mathbf{u}, \mathbf{v}) - F_\theta(\mathbf{v})=0 \qquad \text{for all }\,\,\mathbf{v} \in \mathbf{V}_{div}.
	\end{align}
	\textbf{Step 6.} Existence of a pressure: the linear functional $\varPhi: [H^1_0(\Omega)]^2 \rightarrow \mathbb{R}$ is continuous and $\varPhi(\mathbf{v})=0$  for all $\mathbf{v} \in \mathbf{V}_{div}$. Following \cite[Chapter 4, Proposition 4.6]{ern}, there exists a pressure $p\in Q$ satisfying
	\begin{align}
		\varPhi(\mathbf{v})= (\nabla \cdot \mathbf{v}, p) \qquad \text{for all}\,\, \mathbf{v} \in \mathbf{V}.
	\end{align}
	Thus, the existence of a continuous solution $(\mathbf{u},p,\theta)\in \mathbf{V} \times Q \times \Sigma_D$ is guaranteed by the above analysis. \newline 
	\textbf{Step 7.} Determination of the constant $\rho$: Using the generalized AM-GM inequality, we have
	\begin{align}
		F^2_\theta(\theta_D, \mathcal{Q})& + F^2_{\mathbf{u}}(\theta_D, \mathcal{Q}, \mathbf{g}, \mathbf{f}) \nonumber \\
		&\leq \frac{8}{ C^2_o} \Big[c^2 {\kappa^\ast}^2 \epsilon^{-8}  \|\theta_D\|^2_{H^{1/2}(\partial\Omega)} + C^2_p \|\mathcal{Q}\|^2_{0,\Omega}\Big]+ \frac{16}{C^2_o} \Big[  \big(4c^2 {\kappa^\ast}^2 \epsilon^{-8}\alpha^2 C_o^{-2} C^4_q C^4_{1 \hookrightarrow 4} \|\theta_D\|^2_{H^{1/2}(\partial \Omega)}  \nonumber  \\ & \qquad  + 4\alpha^2 C^2_p C_o^{-2} C^4_q C^4_{1 \hookrightarrow 4} \|\mathcal{Q}\|^2_{0, \Omega} + C^2_o\alpha^2 \big)  \|\mathbf{g}\|^2_{0,\Omega} + C^2_p \|\mathbf{f}\|^2_{0,\Omega} \Big] \nonumber \\
		&\leq C_{exist} \Big[ \|\theta_D\|^2_{H^{1/2}(\partial\Omega)} + \|\mathcal{Q}\|^2_{0,\Omega} + \big(\|\theta_D\|^2_{H^{1/2}(\partial \Omega)} +  \|\mathcal{Q}\|^2_{0, \Omega}  + 1 \big)  \|\mathbf{g}\|^2_{0,\Omega} +  \|\mathbf{f}\|^2_{0,\Omega} \Big] \nonumber \\
		&=: \rho^2, \label{eqnr}
	\end{align}
	where the constant $C_{exist}$ is given by
	\begin{align}
		C_{exist}:=\max \Big\{ \frac{8 c^2 {\kappa^\ast}^2 \epsilon^{-8} }{C^2_o}; \frac{16C^2_p}{C^2_o}; \frac{64c^2 {\kappa^\ast}^2 \epsilon^{-8}\alpha^2  C^4_q C^4_{1 \hookrightarrow 4}}{C_o^4}; \frac{64\alpha^2 C^2_p C^4_q C^4_{1 \hookrightarrow 4}}{C^4_o}; 16 \alpha^2 \Big\}. \label{cexist}
	\end{align}
	Furthermore, the estimate \eqref{exist-0} can be obtained from \eqref{hopf}, \eqref{ext-a}, and \eqref{eqnr}.
	
\end{proof}

{We remark that the existence and uniqueness of the weak solution to the problem $(P)$ has been demonstrated in \cite{lorca1996stationary} with $\mathbf{f}=\mathbf{0}$ and $\mathcal{Q}=0$. For the sake of completeness of this work, we have discussed the well-posedness of the weak solution to problem $(P)$. In contrast to \cite{lorca1996stationary}, we introduce a map $\chi$ and show its continuity (cf. \textbf{Step 2}). To employ the Brouwer fixed-point theorem for the finite approximation of $\mathbf{V}_{div} \times \Sigma_0$, the search for a compact and convex ball is significantly novel due to the map $\chi$ and \eqref{hopf} (cf. \textbf{Step 3.}). Moreover, similarity arises in the finite approximations and limiting case (\textbf{Step 4} to \textbf{6}). Additionally, the present derivation provides an elegant dependence of the physical parameters and the embedding and lifting constants. }

\subsection{Uniqueness} 
We will now investigate the uniqueness of the continuous solution.

\begin{proposition} \label{cunique}
	Under the data assumption \textbf{(A0)}, let $(\mathbf{u},p,\theta) \in \mathbf{V} \times Q \times \Sigma_D $ be any solution to the problem \eqref{variation-01}. Further, we assume the following: \newline
	(i)\,  the data of the problem (P) satisfies
	\begin{align}
		\Scale[0.9]{\sqrt{2} \mu^{-1}_\ast N_0 C^{1/2}_{exist} \Big[ (1+c^2C^{-1}_{exist})\|\theta_D\|^2_{H^{1/2}(\partial\Omega)} + \|\mathcal{Q}\|^2_{0,\Omega} + \|\mathbf{f}\|^2_{0,\Omega} + \big(\|\theta_D\|^2_{H^{1/2}(\partial \Omega)}+  \|\mathcal{Q}\|^2_{0, \Omega}  + 1 \big)  \|\mathbf{g}\|^2_{0,\Omega} \Big]^{1/2} <1}, \label{unique-a}
	\end{align}
	(ii)\,  let $(\mathbf{u}, \theta) \in \mathbf{V} \cap [W^{1,3}(\Omega)]^2 \times \Sigma_D \cap W^{1,3}(\Omega)$ and $\mathbf{g} \in [L^2(\Omega)]^2$ satisfy the following bound
	\begin{align}
		\kappa^{-1}_\ast \big( L_\kappa C_q C_{1\hookrightarrow 6} \|\nabla {\theta}\|_{0,3,\Omega} + \alpha C_q^2 C^2_{1 \hookrightarrow 4} \|\mathbf{g}\|_{0,\Omega} +  L_\mu C_q C_{1 \hookrightarrow 6}  \|\nabla {\mathbf{u}}\|_{0,3,\Omega} \big) < 1, \label{unique-b}
	\end{align}
	then, the problem \eqref{variation-01} has a unique solution $ (\mathbf{u},p,\theta)$.
\end{proposition}

\begin{proof}
	Let us assume $(\mathbf{u},p,\theta)$ and $(\widehat{\mathbf{u}},\widehat{{p}}, \widehat{\theta})$ be the solutions to the problem \eqref{variation-01}. We further set $\mathbf{u}_\ast := \mathbf{u}- \widehat{\mathbf{u}}$, $p_\ast := p- \widehat{p}$ and $\theta_\ast := \theta- \widehat{\theta}$. Notably, $\theta_\ast \in \Sigma_0$. From the second equation of \eqref{variation-01}, we infer
	\begin{align}
		a_T(\theta; \theta, \psi)- a_T(\widehat{\theta}; \widehat{\theta}, \psi)=c^S_T(\widehat{\mathbf{u}}; \widehat{\theta}, \psi) -	c^S_T(\mathbf{u}; \theta, \psi). \nonumber 
	\end{align}
	Replacing $\psi$ with $\theta_\ast \in \Sigma_0$, and adding and subtracting suitable terms, it yields
	\begin{align}
		a_T(\theta; \theta_\ast, \theta_\ast) = -c^S_T(\mathbf{u}_\ast; \widehat{\theta}, \theta_\ast) + 	a_T(\widehat{\theta}; \widehat{\theta}, \theta_\ast) -a_T(\theta; \widehat{\theta}, \theta_\ast). \nonumber 
	\end{align}
	We use the Lipschitz continuity of $\kappa$, the properties \eqref{con-at} and \eqref{cnt-st} and the Sobolev embedding theorem:
	\begin{align}
		\kappa_\ast \|\nabla \theta_\ast\|^2_{0,\Omega} &\leq N_0 \|\nabla \mathbf{u}_\ast\|_{0,\Omega} \|\nabla \widehat{\theta}\|_{0,\Omega} \|\nabla \theta_\ast\|_{0,\Omega} + L_\kappa \|\theta_\ast\|_{0,6,\Omega} \|\nabla \widehat{\theta}\|_{0,3,\Omega} \|\nabla \theta_\ast\|_{0,\Omega} \nonumber \\
		& \leq N_0 \|\nabla \mathbf{u}_\ast\|_{0,\Omega} \|\nabla \widehat{\theta}\|_{0,\Omega} \|\nabla \theta_\ast\|_{0,\Omega} + L_\kappa C_q C_{1\hookrightarrow 6} \|\nabla \widehat{\theta}\|_{0,3,\Omega} \|\nabla \theta_\ast\|^2_{0,\Omega}. \nonumber 
	\end{align}
	Therefore, the following holds:
	\begin{align}
		\Big(\kappa_\ast - L_\kappa C_q C_{1\hookrightarrow 6} \|\nabla \widehat{\theta}\|_{0,3,\Omega} \Big)\|\nabla \theta_\ast\|_{0,\Omega} \leq N_0  \|\nabla \widehat{\theta}\|_{0,\Omega} \|\nabla \mathbf{u}_\ast\|_{0,\Omega}. \label{unique-1}
	\end{align}
	Recalling the first equation in \eqref{variation-01} and replacing $(\mathbf{v}, q)$ by $(\mathbf{u}_\ast, p_\ast)$, we infer 
	\begin{align}
		a_V(\theta; \mathbf{u}, \mathbf{u}_\ast) - a_V(\widehat{\theta}; \widehat{\mathbf{u}}, \mathbf{u}_\ast) + c^S_V(\mathbf{u}; \mathbf{u},\mathbf{u}_\ast) - c^S_V(\widehat{\mathbf{u}}; \widehat{\mathbf{u}},\mathbf{u}_\ast) = F_{\theta}(\mathbf{u}_\ast)- F_{\widehat{\theta}}(\mathbf{u}_\ast). \nonumber
	\end{align}
	Adding and subtracting suitable terms, it holds that
	\begin{align}
		a_V(\theta;\mathbf{u}_\ast, \mathbf{u}_\ast)&=F_{\theta}(\mathbf{u}_\ast)- F_{\widehat{\theta}}(\mathbf{u}_\ast) + a_V(\widehat{\theta}; \widehat{\mathbf{u}}, \mathbf{u}_\ast) -a_V(\theta; \widehat{\mathbf{u}}, \mathbf{u}_\ast) - c^S_V(\mathbf{u}_\ast; \widehat{\mathbf{u}}, \mathbf{u}_\ast) \nonumber \\
		&=(\alpha \mathbf{g}\theta_\ast, \mathbf{u}_\ast) + ((\mu(\widehat{\theta})- \mu(\theta))\nabla \widehat{\mathbf{u}}, \nabla \mathbf{u}_\ast) - c^S_V(\mathbf{u}_\ast; \widehat{\mathbf{u}}, \mathbf{u}_\ast) \nonumber \\
		& \leq \alpha C_q^2 C^2_{1 \hookrightarrow 4} \|\mathbf{g}\|_{0,\Omega} \|\nabla \theta_\ast\|_{0,\Omega} \|\nabla \mathbf{u}_\ast\|_{0,\Omega} + L_\mu C_q C_{1 \hookrightarrow 6} \|\nabla \theta_\ast\|_{0,\Omega} \|\nabla \widehat{\mathbf{u}}\|_{0,3,\Omega} \|\nabla \mathbf{u}_\ast\|_{0,\Omega} + \nonumber \\
		& \qquad N_0  \|\nabla \widehat{\mathbf{u}}\|_{0,\Omega} \|\nabla \mathbf{u}_\ast\|_{0,\Omega}^2. \nonumber 
	\end{align}
	Using the property \eqref{con-av}, we obtain
	\begin{align}
		\Big(\mu_\ast -  N_0  \|\nabla \widehat{\mathbf{u}}\|_{0,\Omega} \Big) \|\nabla \mathbf{u}_\ast\|_{0,\Omega} \leq \Big(\alpha C_q^2 C^2_{1 \hookrightarrow 4} \|\mathbf{g}\|_{0,\Omega}   + L_\mu C_q C_{1 \hookrightarrow 6}  \|\nabla \widehat{\mathbf{u}}\|_{0,3,\Omega}  \Big)\|\nabla \theta_\ast\|_{0,\Omega}. \label{unique-2}
	\end{align}
	Combining \eqref{unique-1} and \eqref{unique-2}, it yields that
	\begin{align}
		\Big(\mu_\ast -  N_0  \big(\|\nabla \widehat{\mathbf{u}}\|_{0,\Omega} + \|\nabla \widehat{\theta}\|_{0,\Omega} \big) \Big) \|\nabla \mathbf{u}_\ast\|_{0,\Omega} &+ \Big( \kappa_\ast - \big( L_\kappa C_q C_{1\hookrightarrow 6} \|\nabla \widehat{\theta}\|_{0,3,\Omega} + \alpha C_q^2 C^2_{1 \hookrightarrow 4} \|\mathbf{g}\|_{0,\Omega} \nonumber \\ & \qquad  +  L_\mu C_q C_{1 \hookrightarrow 6}  \|\nabla \widehat{\mathbf{u}}\|_{0,3,\Omega} \big) \Big)\|\nabla \theta_\ast\|_{0,\Omega} \leq 0. \label{unique-3}
	\end{align}
	Recalling \eqref{exist-0}, it holds that
	\begin{align}
		\|\nabla \mathbf{\widehat{\mathbf{u}}}\|_{0,\Omega} + \|\nabla \widehat{\theta}\|_{0,\Omega} &\leq \sqrt{2} C^{1/2}_{exist} \Big[ (1+c^2C^{-1}_{exist})\|\theta_D\|^2_{H^{1/2}(\partial\Omega)} + \|\mathbf{f}\|^2_{0,\Omega} + \|\mathcal{Q}\|^2_{0,\Omega} \, + \nonumber \\ & \qquad \qquad \qquad \big(\|\theta_D\|^2_{H^{1/2}(\partial \Omega)} +  \|\mathcal{Q}\|^2_{0, \Omega}  + 1 \big)  \|\mathbf{g}\|^2_{0,\Omega} \Big]^{1/2}. \label{unique-4}
	\end{align}
	Thus, employing \eqref{unique-4}, \eqref{unique-a} and \eqref{unique-b} in \eqref{unique-3}, the uniqueness of velocity field and temperature field is evident. Furthermore, the uniqueness of the pressure field can be easily obtained using the inf-sup condition \eqref{con-binf}.  
\end{proof}

\section{Stabilized virtual element framework} \label{sec-3}
In this section, we propose a local projection-based stabilized virtual element formulation for the problem $(P)$. We begin by stating the mesh regularity assumptions and defining the global virtual spaces that approximate the spaces for the velocity vector field, pressure field and temperature field.

\noindent \textbf{(A1)} \textbf{Mesh assumption.} \label{mesh_reg}
Hereafter, let $E$ be any polygonal element such that its area and diameter are given by $|E|$ and $h_E$, respectively. The edge of {$E$} is denoted by $e$.  Moreover, $\mathbf{x}_E \in \mathbb R^2$  and $\mathbf{n}^E$ denote the barycenter and outward unit normal to the boundary $\partial E$ of $E$. Furthermore, we consider a sequence of decompositions of $\Omega$ into non-overlapping general polygonal elements $E$, denoted by $\{\Omega_h\}_{h>0}$. The maximum diameter is defined by $h:=\max_{E \in \Omega_h} h_E$. 
Further, we assume that there exists a constant $\delta_0 >0$, independent of $h$, such that any  $E\in\Omega_h$ satisfies the following  assumptions:
\begin{itemize}
	\item each $E$ is star-shaped with respect to a ball $B_{E}$ of radius  $\geq \delta_0 h_E$;
	\item any edge $e \in \partial E$ has a  finite length $|e| \geq \delta_0  h_E$.
\end{itemize}

Let $\mathbb P_k(E)$ represent the space of polynomials of degree $\leq k$ on $E \in \Omega_h$ with $k \in \mathbb N$. Additionally, we have $\mathbb P_{-1}(E)=\{0\}$. Further, we introduce a set of normalized monomial basis $\mathbb M_{k}(E)$ for the polynomial space $\mathbb P_k(E)$, given by
\begin{align*}
	\mathbb M_k(E)= \Big\{ \Big(\dfrac{ \mathbf{x}-\mathbf{x}_E}{h_E} \Big)^{\mathbf{d}}, |\mathbf{d}| \leq k \Big\},
\end{align*}
where $\mathbf{d}=(d_1,d_2) \in \mathbb{N}^2$ is a multi-index such that $|\mathbf{d}|=d_1+d_2$, and $\mathbf{x}^{\mathbf{d}}= x_1^{d_1}x_2^{d_2}$. {Furthermore, the broken Sobolev and polynomial spaces are defined as follows
	\begin{align}
		W^k_p(\Omega_h)&= \big\{ \psi \in L^2(\Omega) \,\, \text{such that} \,\, \psi_{|E} \in W^k_p(E) \,\, \text{for all}\,\, E \in \Omega_h  \}, \\
		\mathbb{P}_k(\Omega_h)&= \big\{ q \in L^2(\Omega) \,\, \text{such that} \,\, q_{|E} \in \mathbb P_k(E) \,\, \text{for all}\,\, E \in \Omega_h  \}.
\end{align}}

From now on, the local contribution of the continuous forms  $a_V(\cdot; \cdot, \cdot)$, $b(\cdot, \cdot)$,  $c^S_V(\cdot; \cdot, \cdot)$, $a_T(\cdot; \cdot, \cdot)$ and $c^S_T(\cdot; \cdot, \cdot)$ is given by $a^E_V(\cdot; \cdot, \cdot)$, $b^E(\cdot, \cdot)$,  $c^{S,E}_V(\cdot; \cdot, \cdot)$, $a^E_T(\cdot; \cdot, \cdot)$ and $c^{S,E}_T(\cdot; \cdot, \cdot)$, respectively. We will adopt the same decompositions for the load terms. In addition, we have the following
\begin{align*}
	a_{V}(\phi; \mathbf{v},\mathbf{w}) &= \sum_{E \in \Omega_{h}} a^E_{V}(\phi; \mathbf{v},\mathbf{w}),  & b(\mathbf{v},q) &= \sum_{E \in \Omega_{h}} b^{E}(\mathbf{v},q), &
	a_{T}(\theta; \phi, \psi) &= \sum_{E \in \Omega_{h}} a^E_{T}(\theta; \phi, \psi), \\
	c^S_V(\mathbf{v}; \mathbf{w},\mathbf{z}) &= \sum_{E \in \Omega_{h}} c^{S,E}_V(\mathbf{v}; \mathbf{w}, \mathbf{z}), & 
	c^S_T(\mathbf{v}; \phi,\psi) &= \sum_{E \in \Omega_{h}} c^{S,E}_T(\mathbf{v}; \phi, \psi), & F_\phi(\mathbf{v})&= \sum_{E \in \Omega_h} F^E_\phi(\mathbf{v}).
\end{align*}

For any $E \in \Omega_h$, we define the following polynomial projections: 
\begin{itemize}
	\item the $H^1$-energy projection operator  $\Pi^{\nabla,E}_k: H^1(E) \rightarrow \mathbb{P}_k(E)$ can be defined as follows
	\begin{align}
		\begin{cases}
			\int_E \nabla (\psi - \Pi^{\nabla,E}_k \psi) \cdot \nabla m_k \, dE = 0 \qquad \text{for all} \, \psi \in H^1(E)\,\, \text{and} \,\, m_k \in \mathbb{P}_k(E),\\
			\int_{\partial E} (\psi - \Pi^{\nabla,E}_k \psi) \, ds=0,
		\end{cases}
		\label{H1proj}
	\end{align}
	with its extension for the vector-valued function $\boldsymbol{\Pi}^{\nabla,E}_k: [H^1(E)]^2 \rightarrow [\mathbb{P}_k(E)]^2$.
	\item  the $L^2$-projection operator $\Pi^{0,E}_k: L^2(E) \rightarrow \mathbb{P}_k(E)$ can be given as follows
	\begin{align}
		\int_E (\psi- \Pi^{0,E}_k \psi) m_k \, dE = 0 \qquad \text{for all} \,\, \psi \in L^2(E)\,\, \text{and} \,\, m_k \in \mathbb{P}_k(E), \label{l2proj}
	\end{align}
	with its extension for vector-valued function $\boldsymbol{\Pi}^{0,E}_k: [L^2(E)]^2 \rightarrow [\mathbb{P}_k(E)]^2$.
\end{itemize} 
\begin{remark}
	From \cite[Remark 3.1]{mishra2025equal}, the $L^2$-projection operator satisfies the stability property with respect to $L^q$-norm for any $q\geq 2$, i.e., for any $v\in L^2(E)$
	\begin{align}
		\|\Pi^{0,E}_{k} v\|_{L^q(E)} \leq C \|v\|_{L^q(E)} \qquad \,\, \text{for all }\,\, q\geq 2. \label{lqstab}
	\end{align}
\end{remark}

\subsection{Virtual element spaces } 
Following the VEM framework \cite{vem25}, we now outline the $H^1$-conforming virtual space for the velocity vector field, pressure field and temperature field. For any $E \in \Omega_h$, we define the local virtual space as follows:
\begin{align}
	V_h(E):= \Big\{ \psi_h \in H^1(E) \cap C^0(\partial E),& \,\, \text{such that} \nonumber \\ &(i) \,\,\Delta \psi_h \in \mathbb P_k(E), \nonumber \\ &(ii)\, \,  \psi_{h|e} \in \mathbb P_k(e) \,\, \forall \,\, e \in \partial E,  \nonumber \\  &(iii) \,\, \big(\psi_h-\Pi_k^{\nabla,E}\psi_h, m_\gamma \big)=0, \, \forall \, m_\gamma \in \mathbb{M}^\ast_{k-1}(E) \cup \mathbb{M}^\ast_{k}(E) \Big\}, \nonumber
\end{align}
where $\mathbb{M}^\ast_{k}(E)$ represents the set of monomials of degree equal to $k$.

Hereafter, we summarize the following degrees of freedom for the local space $V_h(E)$:
\begin{itemize}
	\item $\mathbf{DoF_1}(\psi_h)$: the values of $\psi_h$ evaluated at each vertex of the polygonal element $E$,
	\item $\mathbf{DoF_2}(\psi_h)$: the values of $\psi_h$ evaluated at the $(k-1)$ distinct points of each edge $e \in \partial E$,
	\item $\mathbf{DoF_3}(\psi_h)$: the internal moments of $\psi_h$ 
	\begin{align*}
		\frac{1}{|E|} \, \int_E \psi_h m_\gamma \, dE \quad \text{for all}\, \, m_\gamma \in \mathbb{M}_{k-2}(E).
	\end{align*}	
\end{itemize}
We thus obtain the global virtual element space by assembling the local spaces $V_h(E)$ as follows
\begin{align}
	V_h:=\{ \psi_h \in H^1(\Omega) \, \, \text{such that}\, \, \psi_{h|E} \in V_h(E) \, \,  \text{for all}\,\, E \in \Omega_h \}. \nonumber
\end{align}
Therefore, the $H^1$-conforming global velocities space $\mathbf{V}_{h}$ and pressure space $Q_{h}$ can be defined by
\begin{align}
	\mathbf{V}_{h} &:= \{ \mathbf{z}_h \in \mathbf{V} \,\, \,\, \text{such that}\,\,\,\, \mathbf{z}_{h|E} \in [V_h(E)]^2 \,\,\,\, \text{for all} \,\, E \in \Omega_h\}, \nonumber \\
	Q_{h} &:= \{ q_h \in Q \,\, \,\, \text{such that}\,\,\,\, q_{h|E} \in V_h(E) \,\, \,\,\text{for all} \,\, E \in \Omega_h\}. \nonumber
\end{align} 
Furthermore, the global temperature space $\Sigma_h$ is defined by
\begin{align}
	\Sigma_{h} &:= \{ \psi_h \in \Sigma \,\,\,\, \text{such that}\,\, \,\, \psi_{h|E} \in V_h(E) \,\,\,\, \text{for all} \,\, E \in \Omega_h\}. \nonumber
\end{align}
Its homogeneous counterpart is defined as
\begin{align}
	\Sigma_{0,h} &:= \{ \psi_h \in \Sigma_h \,\,\,\, \text{such that}\,\, \,\, \psi_{h|\partial \Omega}=0 \}. \nonumber
\end{align}
To address nonhomogeneous Dirichlet data under assumption (A0), we define the ‘quasi’-interpolant $\theta_{D,h}$ of $\theta_{D}$ as  
\begin{itemize}
	\item \(\theta_{D,h} \in C^0(\partial \Omega)\),
	\item piecewise polynomial on each edge: $\theta_{D,h}|_e \in \mathbb{P}_k(e)$ for all $e \subset \partial \Omega$.
\end{itemize}
If $\theta_D \in C^0(\partial \Omega)$, the usual nodal interpolant provides a natural choice for $\theta_{D,h}$.
We are now in a position to define the discrete counterpart of $\Sigma_D$, given by
\begin{align}
	\Sigma_{D,h} &:= \{ \psi_h \in \Sigma_h \,\,\,\, \text{such that}\,\, \,\, \psi_{h|\partial \Omega}=\theta_{D,h} \}. \nonumber
\end{align}

\subsection{Virtual element forms} 
We now focus on the virtual element approximations of the continuous forms and external load terms, which are easily computable through the degrees of freedom \textbf{DoF}s. Let $\mathbf{v}_h, \mathbf{w}_h, \mathbf{z}_h \in \mathbf{V}_h, q_h \in Q_h$, and $\theta_h, \phi_h, \psi_h \in \Sigma_h$, then the local discrete forms corresponding to continuous forms $a_V$, $b$, $c^S_V$, $a_T$ and $c^S_T$ can be given as follows:
\begin{align*}
a^E_{V,h}( \phi_h; \mathbf{v}_h, \mathbf{w}_h) :&=  \int_E \mu(\Pi^{0,E}_k \phi_h) \boldsymbol{\Pi}^{0,E}_{k-1} \nabla \mathbf{v}_h: \boldsymbol{\Pi}^{0,E}_{k-1} \nabla \mathbf{w}_h\, dE + \mu(\Pi^{0,E}_0 \phi_h) \textbf{S}^E_ {\nabla,k} \big(\mathbf{v}_h, \mathbf{w}_h \big),  \\  
b^E_{h}(\mathbf{v}_h, q_h) :&= \int_E \Pi^{0,E}_{k-1} \nabla \cdot \mathbf{v}_h \Pi^{0,E}_{k} q_h\, dE, \\
c^{E}_{V,h}(\mathbf{v}_h; \mathbf{w}_h,\mathbf{z}_h):&= \int_E  \big(\boldsymbol{\Pi}^{0,E}_{k-1} \nabla \mathbf{w}_h \big) \boldsymbol{\Pi}^{0,E}_{k} \mathbf{v}_h \cdot \boldsymbol{\Pi}^{0,E}_k \mathbf{z}_h\, dE,\\
c^{S,E}_{V,h}(\mathbf{v}_h; \mathbf{w}_h,\mathbf{z}_h):&= \dfrac{1}{2} \big[c^{E}_{V,h}(\mathbf{v}_h; \mathbf{w}_h,\mathbf{z}_h) - c^{E}_{V,h}(\mathbf{v}_h; \mathbf{z}_h,\mathbf{w}_h) \big], \\ 
a^E_{T,h}(\theta_h; \phi_h, \psi_h):&= \int_{E} \kappa(\Pi^{0,E}_k \theta_h) \boldsymbol{\Pi}^{0,E}_{k-1} \nabla \phi_h \cdot \boldsymbol{\Pi}^{0,E}_{k-1} \nabla \psi_h \, dE + \kappa(\Pi^{0,E}_0 \theta_h) S^E_{\nabla,k} \big( \phi_h, \psi_h \big), \\
c^{E}_{T,h}(\mathbf{v}_h; \phi_h,\psi_h):&= \int_E  \big[\boldsymbol{\Pi}^{0,E}_{k} \mathbf{v}_h \cdot	\boldsymbol{\Pi}^{0,E}_{k-1} \nabla \phi_h \big] \Pi^{0,E}_k \psi_h\, dE,  \\
c^{S,E}_{T,h}(\mathbf{v}_h; \phi_h,\psi_h):&= \dfrac{1}{2} \big[c^{E}_{T,h}(\mathbf{v}_h; \phi_h,\psi_h) - c^{E}_{T,h}(\mathbf{v}_h; \psi_h,\phi_h) \big],  \nonumber 
\end{align*}
where the local VEM stabilizing terms $\mathbf{S}^E_{\nabla,k}(\cdot, \cdot)$  and ${S}^E_{\nabla,k}(\cdot, \cdot)$ are defined as follows:
\begin{align}
\textbf{S}^E_ {\nabla,k} \big(\mathbf{v}_h, \mathbf{w}_h \big) &:=  {S^E_1} \big((\mathbf{I}-\boldsymbol{\Pi}^{\nabla,E}_k) \mathbf{v}_h,(\mathbf{I}-\boldsymbol{\Pi}^{\nabla,E}_k) \mathbf{w}_h \big), \nonumber\\
S^E_{\nabla,k} \big( \phi_h, \psi_h \big) &:= {S^E_2} \big(( I- {\Pi}^{\nabla,E}_k) \phi_h,({I}- {\Pi}^{\nabla,E}_k) \psi_h \big), \nonumber
\end{align}
and where ${S^E_1}(\cdot,\cdot): \mathbf{V}_{h}(E) \times \mathbf{V}_{h}(E) \rightarrow \mathbb{R}$ and ${S^E_2}(\cdot, \cdot): \Sigma_h(E) \times \Sigma_h(E) \rightarrow \mathbb R$ are the VEM computable symmetric positive definite bilinear forms such that they satisfy 
\begin{align}
\lambda_{1\ast} \ \|\nabla \mathbf{w}_h\|_{0,E}^2 &\leq S^E_\nabla(\mathbf{w}_h, \mathbf{w}_h)	 \leq
\lambda_1^\ast  \|\nabla \mathbf{w}_h\|_{0,E}^2 \qquad  \text{for all} \,\, \mathbf{w}_h \in \mathbf{V}_{h}(E) \cap ker(\boldsymbol{\Pi}^{\nabla,E}_{k}), \label{vem-a} \\
\lambda_{2\ast} \ \| \nabla \phi_h\|_{0,E}^2 &\leq S^E(\phi_h, \phi_h) \leq
\lambda_2^\ast  \| \nabla \phi_h\|_{0,E}^2 \qquad  \text{for all}\, \phi_h \in \Sigma_{h}(E) \cap ker(\Pi^{\nabla,E}_{k}), \label{vem-b}
\end{align}
with the positive constants $\lambda_{1\ast} \leq \lambda_1^\ast$ and $\lambda_{2\ast} \leq \lambda_2^\ast$, which are independent of $h$.
\begin{remark}
Following the standard VEM literature \cite{vem1,mvem9,mvem3}, the local VEM stabilization terms $S^E_1(\cdot,\cdot)$ and $S^E_2(\cdot,\cdot)$ are defined by dofi-dofi stabilization. Let $\vec{\mathbf{v}}_h$, $\vec{\mathbf{w}}_h$, $\vec{\phi}_h$ and $\vec{\psi}_h$ denote the real valued vectors containing the values of the local degrees of freedom associated to ${\mathbf{v}}_h, {\mathbf{w}}_h \in \mathbf{V}_h$  and to ${\phi}_h, {\psi}_h \in \Sigma_h$, then
\begin{align}
	S^E_1(\mathbf{v}_h, \mathbf{w}_h) = \vec{\mathbf{v}}_h \cdot \vec{\mathbf{w}}_h, \qquad \qquad S^E_2(\phi_h, \psi_h)= \vec{\phi}_h \cdot \vec{\psi}_h. \label{dofi}
\end{align}
\end{remark}

\noindent Furthermore, the global discrete forms corresponding to the above local forms are given by
\begin{align*}
a_{V,h}( \phi_h; \mathbf{v}_h,\mathbf{w}_h) &= \sum_{E \in \Omega_{h}} a^E_{V,h}(\phi_h; \mathbf{v}_h,\mathbf{w}_h),  &  b_{h}(\mathbf{v}_h,q_h) & = \sum_{E \in \Omega_{h}} b^{E}_{h}(\mathbf{v}_h,q_h), \\ c^S_{V,h}(\mathbf{v}_h; \mathbf{w}_h,\mathbf{z}_h) &= \sum_{E \in \Omega_{h}} c^{S,E}_{V,h}(\mathbf{v}_h; \mathbf{w}_h,\mathbf{z}_h), &
a_{T,h}(\theta_h; \phi_h, \psi_h) &= \sum_{E \in \Omega_{h}} a^E_{T,h}( \theta_h; \phi_h, \psi_h),  \\
c^S_{T,h}(\mathbf{v}_h; \phi_h,\psi_h) &= \sum_{E \in \Omega_{h}} c^{S,E}_{T,h}(\mathbf{v}_h; \phi_h, \psi_h).
\end{align*}
For given $\psi_h \in \Sigma_h$, the discrete load term corresponding to $F_{\psi_h}(\cdot)$ is defined as follows
\begin{align}
{F}_{h,\psi_h}(\mathbf{v}_h):= \sum_{E \in \Omega_h} \big( \alpha \mathbf{g} \Pi^{0,E}_k \psi_h, \boldsymbol{\Pi}^{0,E}_k \mathbf{v}_h\big) + \sum_{E \in \Omega_h} \big(\mathbf{f}, \boldsymbol{\Pi}^{0,E}_k \mathbf{v}_h\big); \qquad (\mathcal{Q}_h, \psi_h):= \sum_{E \in \Omega_h} \big( \mathcal{Q}, \Pi^{0,E}_k \psi_h\big). \nonumber 
\end{align}
\noindent  \textbf{Stability properties of the discrete forms:} the global discrete form defined above satisfies the following properties: \newline
\noindent $\bullet$ continuity and coercivity of $a_{V,h}(\cdot; \cdot, \cdot)$: for any given $\phi_h \in \Sigma_h$,  there exist two positive constant $\varrho^\ast$ and $\varrho_\ast$ independent of the mesh size such that
\begin{align}
a_{V,h}(\phi_h; \mathbf{v}_h, \mathbf{w}_h) \leq \varrho^\ast \mu^\ast \|\nabla \mathbf{v}_h\|_{0,\Omega} \|\nabla \mathbf{w}_h\|_{0,\Omega}, \quad  	a_{V,h}(\phi_h; \mathbf{w}_h, \mathbf{w}_h) \geq \varrho_\ast \mu_\ast \|\nabla \mathbf{w}_h\|^2_{0,\Omega}, \label{dcon-av}
\end{align}
for all $\mathbf{v}_h, \mathbf{w}_h \in \mathbf{V}_h$. \newline 
$\bullet$ continuity of the bilinear form $b_h(\cdot; \cdot)$: for any $\mathbf{v}_h \in \mathbf{V}_h$ and $q_h \in Q_h$, we obtain
\begin{align}
b_h(\mathbf{v}_h, q_h) \leq {\sqrt{2}} \|\nabla \mathbf{v}_h\|_{0,\Omega} \|q\|_{0,\Omega}. \nonumber
\end{align}
$\bullet$ continuity and coercivity of $a_{T,h}(\cdot; \cdot, \cdot)$: for any given $\theta_h \in \Sigma_h$, there exist positive constants $\wp^\ast$ and $\wp_\ast$ independent of the mesh size such that
\begin{align}
a_{T,h}(\theta_h; \phi_h, \psi_h) \leq \wp^\ast \kappa^\ast \|\nabla \phi_h\|_{0,\Omega} \|\nabla \psi_h\|_{0,\Omega}, \quad  a_{T,h}(\theta_h; \psi_h, \psi_h) \geq \wp_\ast \kappa_\ast \|\nabla \psi_h\|^2_{0,\Omega}, \label{dcon-at}
\end{align}
for all $\phi_h, \psi_h \in \Sigma_h$. \newline 
$\bullet$ continuity of $c^S_{V,h}(\cdot; \cdot, \cdot)$ and $c^S_{T,h}(\cdot; \cdot, \cdot)$: 
\begin{align}
c^S_{V,h}(\mathbf{v}_h; \mathbf{w}_h, \mathbf{z}_h) & \leq \widehat{N_0} \|\nabla \mathbf{v}_h\|_{0,\Omega} \| \nabla \mathbf{w}_h\|_{0,\Omega} \| \nabla \mathbf{z}_h\|_{0,\Omega}  \qquad \,\, \text{for all} \,\, \mathbf{v}_h, \mathbf{w}_h, \mathbf{z}_h \in \mathbf{V}_h, \label{dcnt-sv}\\ 
c^S_{T,h}(\mathbf{v}_h; \phi_h,\psi_h)&\leq \widehat{N_0} \|\nabla \mathbf{v}_h\|_{0,\Omega} \|\nabla \phi_h\|_{0,\Omega} \|\nabla \psi_h\|_{0,\Omega} \qquad \,\, \text{for all} \,\, \mathbf{v}_h \in \mathbf{V}_h, \phi_h \in \Sigma_h, \psi \in \Sigma_{0,h}, \label{dcnt-st}
\end{align}
where $\widehat{N_0}>0$ is the continuity constant independent of the mesh size. Additionally, they satisfy
\begin{align}
c^S_{V,h}(\mathbf{v}_h; \mathbf{w}_h, \mathbf{w}_h) =0,  \qquad 
c^S_{T,h}(\mathbf{v}_h; \psi_h,\psi_h)= 0 \qquad \,\, \text{for all} \,\, \mathbf{v}_h,\mathbf{w}_h \in \mathbf{V}_h, \psi_h \in \Sigma_{0,h}. \label{dcon-cv}
\end{align} 

\subsection{ Stabilized virtual element problem }
\label{sec-3.4}
Since the equal-order VEM approximation is employed, the discrete solution is strongly polluted by non-physical oscillations or moving fronts. To address this issue, we define the following stabilization terms:

\noindent $\bullet$ the velocity field stabilizing term:
\begin{align}
\mathcal{L}_{1,h}(\mathbf{w}_h,\mathbf{z}_h) &:= \sum_{E \in \Omega_h}\tau_{1,E} \big[\big( \widehat{\mathbf{r}}_h(\nabla \mathbf{w}_h), \widehat{\mathbf{r}}_h(\nabla \mathbf{z}_h) \big)  + \mathbf{S}^E_{\nabla,k}\big(\mathbf{w}_h,  \mathbf{z}_h \big)\big]\qquad \text{for all}\,\, \mathbf{w}_h, \mathbf{z}_h \in \mathbf{V}_h,  \label{vel-1}
\end{align}
where $\widehat{\mathbf{r}}_h := \boldsymbol{\Pi}^{0,E}_k - \boldsymbol{\Pi}^{0,E}_{k-1}$, for more details see \cite{vem28m}.\newline
\noindent $\bullet$ the pressure stabilizing term:
\begin{align}
\mathcal{L}_{2,h}(p_h,q_h) &:= \sum_{E \in \Omega_h}\tau_{2,E} \big[\big( \widehat{\mathbf{r}}_h(\nabla p_h), \widehat{\mathbf{r}}_h(\nabla q_h) \big)  + {S^E_{p}\big( p_h - \Pi^{\nabla,E}_{k-1}p_h,  q_h - \Pi^{\nabla,E}_{k-1} q_h \big)}\big] \qquad \text{for all}\,\, p_h, q_h \in Q_h,  \label{press-1}
\end{align}
{where the local VEM stabilization term $S^E_p(\cdot,\cdot) : Q_h(E) \times Q_h(E)$ is defined by
\begin{align}
	\lambda_{3\ast} \ \|\nabla q_h\|_{0,E}^2 \leq S^E_p(q_h, q_h) \leq
	\lambda_3^\ast  \|\nabla q_h\|_{0,E}^2 \qquad  \forall\, q_h \in Q_{h}(E) \cap ker({\Pi}^{\nabla,E}_{k-1}), \label{vem-c}
\end{align}
with the positive constants $\lambda_{3\ast} \leq \lambda_3^\ast$ which are independent of the mesh size.} \newline
\noindent $\bullet$ the transport temperature equation is stabilized by
\begin{align}
\mathcal{L}_{T,h}(\phi_h,\psi_h) &:= \sum_{E \in \Omega_h}\tau_{E} \big[\big( \widehat{\mathbf{r}}_h(\nabla \phi_h), \widehat{\mathbf{r}}_h(\nabla \psi_h) \big)  + S^E_{\nabla,k}\big(\phi_h, \psi_h \big)\big] \qquad \text{for all}\,\, \phi_h, \psi_h \in \Sigma_{h},  \label{trans_1}
\end{align}
where $\tau_{i,E}$ and $\tau_E$ are stabilization parameters.

Let $\mathcal{L}_{h}[(\mathbf{w}_h, p_h), (\mathbf{z}_h,q_h)]:= \mathcal{L}_{1,h}(\mathbf{w}_h,\mathbf{z}_h) + \mathcal{L}_{2,h}(p_h, q_h)$. In light of the preceding discussion, we define the following settings: 
\begin{align}
A_{V,h}[\phi_h;(\mathbf{w}_h, p_h), (\mathbf{z}_h,q_h) ] &:= a_{V,h}(\phi_h; \mathbf{w}_h, \mathbf{z}_h)- b_h(\mathbf{z}_h, p_h) + b_h(\mathbf{w}_h, q_h) +\mathcal{L}_{h}[(\mathbf{w}_h, p_h), (\mathbf{z}_h,q_h)], \nonumber \\
A_{T,h}(\theta_h, \mathbf{w}_h; \phi_h, \psi_h) &:= a_{T,h}(\theta_h; \phi_h, \psi_h)+ c^S_{T,h}(\mathbf{w}_h; \phi_h, \psi_h)+\mathcal{L}_{T,h}(\phi_h, \psi_h). \nonumber
\end{align}
Therefore, the stabilized virtual element problem for the primal formulation \eqref{variation-1} or \eqref{variation-01} is given as follows:
\begin{align}
\begin{cases}
	\text{Find}\,\, (\mathbf{u}_h, p_h,\theta_h) \in \mathbf{V}_h \times Q_h \times \Sigma_{D,h},\,\, \text{such that} \\
	A_{V,h}[ \theta_h; (\mathbf{u}_h, p_h), (\mathbf{v}_h, q_h)]+ c^S_{V,h}(\mathbf{u}_h; \mathbf{u}_h, \mathbf{v}_h)= F_{h,\theta_h}(\mathbf{v}_h) \quad \text{for all} \, (\mathbf{v}_h, q_h) \in \mathbf{V}_h \times Q_h, \\
	A_{T,h}(\theta_h, \mathbf{u}_h; \theta_h, \psi_h)=(\mathcal{Q}_h, \psi_h) \quad \text{for all} \,\, \psi_h \in \Sigma_{0,h}.
\end{cases}
\label{nvem}
\end{align}

\begin{remark}
In VEM literature, a mass-based local projection stabilization technique has been developed by J. Guo et al. in \cite{vem028} for the Stokes equation, where the stabilization term is given by
\begin{align}
	\mathcal{L}^\ast_{2,h}(p_h,q_h) &:= \sum_{E \in \Omega_h} \big[\big( \Pi^{0,E}_k p_h - \Pi^{0,E}_{k-1} p_h, \Pi^{0,E}_k q_h - \Pi^{0,E}_{k-1} q_h \big)   + S^E_{0,k-1}\big( p_h,  q_h \big)\big],   \label{press-01}
\end{align}
where $S^E_{0,k-1}\big( p_h,  q_h \big):=S^E((I- \Pi^{0,E}_{k-1})p_h, (I-\Pi^{0,E}_{k-1}) q_h)$ for all $p_h,q_h \in Q_h$. {Following \eqref{vem-c}, there exist positive constants $\lambda_{4\ast} \leq \lambda_4^\ast$ independent of the mesh size such that the following holds
	\begin{align}
		\lambda_{4\ast} \ \| q_h\|_{0,E}^2 \leq S^E_{0,k-1}(q_h, q_h) \leq
		\lambda_4^\ast  \| q_h\|_{0,E}^2 \qquad  \forall\, q_h \in Q_{h}(E) \cap ker({\Pi}^{0,E}_{k-1}). \label{vem-d}
\end{align}}
\end{remark}

We now aim to investigate the relationship between gradient-based local projection stabilization techniques \eqref{press-1} and mass-based local projection stabilization methods \eqref{press-01}. Under the suitable choice of $\tau_{2,E}$, both stabilization methods are equivalent.
\begin{lemma} \label{pressure}
The pressure stabilization techniques \eqref{press-1} and \eqref{press-01} are equivalent when $\tau_{2,E} \sim h^2_E$ for all $E \in \Omega_h$, i.e., there exist positive constants $\alpha_1$ and $\alpha_2$ independent of the mesh size, satisfying 
\begin{align}
	\alpha_1 \mathcal{L}_{2,h}(q_h,q_h) \leq \mathcal{L}^\ast_{2,h}(q_h, q_h) \leq \alpha_2 \mathcal{L}_{2,h}(q_h,q_h) \qquad \text{for all}\,\, q_h \in Q_h. 
\end{align}
\end{lemma}
\begin{proof}
See the Appendix.
\end{proof}


\section{Theoretical analysis} \label{sec-4}
This section {focuses} on demonstrating the well-posedness of the stabilized virtual element problem \eqref{nvem} and the error estimates in the energy norm. 
\subsection{Preliminary results}
Let us define mesh-dependent energy norms over $\mathbf{V}_h \times Q_h$ and $\Sigma_{h}$ as follows:  
\begin{align}
\vertiii{(\mathbf{w}_h, q_h)}^2 &=  \|\nabla \mathbf{w}_h\|^2_{0, \Omega} + \|q_h\|^2_{0, \Omega} + \mathcal{L}_{1,h}(\mathbf{w}_h,\mathbf{w}_h) + \mathcal{L}_{2,h}(q_h,q_h), \nonumber \\
\vertiii{\psi}^2 & =  \|\nabla \psi_h\|^2_{0, \Omega} + \mathcal{L}_{T,h}(\psi_h,\psi_h). \nonumber
\end{align}

\begin{lemma}(\cite{scott})
\label{lemmaproj1}
Under the assumption \textbf{(A1)}, for any functions $\psi \in H^s(E)$ with $E\in \Omega_h$, satisfying the following
\begin{align}
	\|\psi -\Pi^{\nabla,E}_k \psi\|_{l,E} &\leq C h^{s-l}_E |\psi|_{s,E} \quad s, l \in \mathbb{N}, \,\,\, s \geq 1,\,\, \, l \leq s \leq k+1,\label{eqproj1}	 \\
	\|\psi - \Pi^{0,E}_k \psi \|_{l,E} &\leq C h^{s-l}_E |\psi|_{s,E} \quad s, l \in \mathbb{N},\,\,\, l \leq s \leq k+1.
\end{align}
\end{lemma}

\begin{lemma}\label{inverse}
(Inverse inequality \cite{vem28}) 
Under the assumption \textbf{(A1)}, for any virtual element function $\psi_h \in V_h(E)$ defined on $E \in \Omega_{h}$, it holds that
\begin{align}
	|\psi_h|_{1,E} \leq C_{inv} h^{-1}_E \|{\psi}_h\|_{0,E},
\end{align}
where the positive constant $C_{inv}$ is independent of the mesh size.
\end{lemma}

\begin{lemma} \label{estimate2}
For all  $(\mathbf{v}_h,q_h) \in \mathbf{V}_h \times Q_h$, $\theta_h, \phi_h \in \Sigma_h$ and $\psi_h \in \Sigma_{0,h}$, there holds the following:
\begin{align}
	|F_{h,\theta_h}(\mathbf{v}_h)| &\leq C_\mathbf{\alpha}\big(\|\mathbf{f}\|_{0,\Omega} + \|\mathbf{g}\|_{0,\Omega} \vertiii{\theta_h} \big) \vertiii{(\mathbf{v}_h, q_h)}, \label{est1} \\
	|a_{T,h}(\theta_h; \phi_h, \psi_h) + \mathcal{L}_{T,h}(\phi_h, \psi_h)| &\leq C_{\kappa,\lambda} \|\nabla \phi_h\|_{0,\Omega}\|\nabla \psi_h\|_{0,\Omega}, \label{est2}\\
	|A_{T,h}(\theta_h, \mathbf{v}_h; \psi_h, \psi_h) | &\geq C_{o,T} \vertiii{ \psi_h}^2, \label{est3}\\
	\vertiii{\psi_h}&\leq C_{\lambda} \|\nabla \psi_h\|_{0,\Omega}, \label{est4}
\end{align}
where $C_\alpha:=\max\{C_p, \alpha C_q^2 C_{1\hookrightarrow 4}^2 \}$, $C_{\kappa, \lambda}:= \big(\wp^\ast \kappa^\ast + (1+ \lambda^\ast_2) \max_{E \in \Omega_h} \tau_{E} \big)$, $C_{o,T}:=\min\{ 1, \wp_\ast \kappa_\ast \}$ and $C_\lambda:= \sqrt{1+(1+ \lambda^\ast_2) \max_{E \in \Omega_h} \tau_{E}}$.
\end{lemma}

\begin{proof}
Using the stability of the projectors, the Poincar$\acute{\text{e}}$ inequality and the Sobolev embedding theorem, we have
\begin{align}
	|F_{h,\theta_h}(\mathbf{v}_h)| &\leq C_p \|\mathbf{f}\|_{0,\Omega} \|\nabla \mathbf{v}_h\|_{0,\Omega} + \alpha C^2_{1\hookrightarrow 4} \|\mathbf{g}\|_{0,\Omega} \|\mathbf{v}_h\|_{1,\Omega} \|\theta_h\|_{1,\Omega} \nonumber \\
	& \leq \max\{C_p, \alpha C_q^2 C_{1\hookrightarrow 4}^2 \} \big( \|\mathbf{f}\|_{0,\Omega}  +  \|\mathbf{g}\|_{0,\Omega} \vertiii{\theta_h}\big) \vertiii{(\mathbf{v}_h,q_h)},	\nonumber
\end{align}
which proves the result \eqref{est1}. Concerning \eqref{est2}, we use \eqref{vem-b}, \eqref{dcon-at} and the properties of the projectors:
\begin{align}
	|a_{T,h}(\theta_h; \phi_h, \psi_h) + \mathcal{L}_{T,h}(\phi_h, \psi_h)| &\leq \wp^\ast \kappa^\ast \|\nabla \phi_h\|_{0,\Omega} \|\nabla \psi_h\|_{0,\Omega} + (1+ \lambda^\ast_2) \max_{E \in \Omega_h} \tau_{E} \|\nabla \phi_h\|_{0,\Omega} \|\nabla \psi_h\|_{0,\Omega} \nonumber \\
	&\leq \big(\wp^\ast \kappa^\ast + (1+ \lambda^\ast_2) \max_{E \in \Omega_h} \tau_{E} \big) \|\nabla \phi_h\|_{0,\Omega} \|\nabla \psi_h\|_{0,\Omega} \nonumber,
\end{align}
which completes the proof of \eqref{est2}. The bound \eqref{est3}  readily follows from \eqref{dcon-at} and the definition of $\vertiii{\cdot}$. Furthermore, the bound \eqref{est4} can be easily obtained by following the estimation of \eqref{est2}. 
\end{proof}

\begin{lemma}\label{vacca}
Let $\widehat{\theta}_h, \theta_h \in \Sigma_{D,h}$. For sufficiently smooth $(\mathbf{v}_h, \phi_h) \in \mathbf{V}_h\times \Sigma_h$, there exist constants $C_v$ and $C_t$ depending on $\Omega$ and $k$ only and the mesh regularity constant $\delta_0$ in assumption \textbf{(A1)} such that there holds that for all $\mathbf{w}_h \in \mathbf{V}_h$ and $\psi_h \in \Sigma_{0,h}$
\begin{align}
	|a_{V,h}(\widehat{\theta}_h; \mathbf{v}_h, \mathbf{w}_h)- a_{V,h}(\theta_h; \mathbf{v}_h, \mathbf{w}_h)| & \leq C_v C_q L_\mu C_{1 \hookrightarrow 6} \|\nabla\mathbf{v}_h\|_{0,3, \Omega} \|\nabla(\widehat{\theta}_h - \theta_h)\|_{0,\Omega} \|\nabla \mathbf{w}_h\|_{0,\Omega}, \\ 
	|a_{T,h}(\widehat{\theta}_h; \phi_h, \psi_h)- a_{T,h}(\theta_h; \phi_h, \psi_h)| & \leq C_t C_q L_\kappa C_{1 \hookrightarrow 6} \|\nabla\phi_h\|_{0,3, \Omega} \|\nabla(\widehat{\theta}_h - \theta_h)\|_{0,\Omega} \|\nabla \psi_h\|_{0,\Omega}. 
\end{align}
\end{lemma}

\begin{proof}
The proof of Lemma \ref{vacca} follows from \cite[Lemma 4.3]{mvem3}. 
\end{proof}
\subsection{Well-posedness of the decoupled problems}
The stabilized virtual element problem \eqref{nvem} can be decoupled as follows:
\begin{itemize}
\item For a given $\widehat{\theta}_h \in \Sigma_h$ and $\widehat{\mathbf{u}}_h \in \mathbf{V}_h$, find $(\mathbf{u}_h, p_h) \in \mathbf{V}_h \times Q_h$, such that it satisfies 
\begin{align}
	A_{V,h}[ \widehat{\theta}_h; (\mathbf{u}_h, p_h), (\mathbf{v}_h, q_h)]+ c^S_{V,h}(\widehat{\mathbf{u}}_h; \mathbf{u}_h, \mathbf{v}_h)= F_{h,\widehat{\theta}_h}(\mathbf{v}_h) \quad \text{for all} \, \, (\mathbf{v}_h, q_h) \in \mathbf{V}_h \times Q_h.	\label{dnvem1}
\end{align}
Notably, the linearized problem \eqref{dnvem1} is the Oseen problem.
\item For given $\widehat{\mathbf{u}}_h \in \mathbf{V}_h$ and $\widehat{\theta}_h \in \Sigma_h$, find $\theta_h \in \Sigma_{D,h}$, such that 
\begin{align}
	A_{T,h}(\widehat{\theta}_h, \widehat{\mathbf{u}}_h; \theta_h, \psi_h)=(\mathcal{Q}_h, \psi_h) \quad \text{for all} \,\, \psi_h \in \Sigma_{0,h}.
	\label{dnvem2}
\end{align} 
\end{itemize}

We now introduce the weak inf-sup condition to establish the well-posedness of the
decoupled problem \eqref{dnvem1}.

\begin{lemma}(weak inf-sup condition) \label{infsup}
Under the assumptions \textbf{(A1)}, there exist two positive constants $\beta_1$ and $\beta_2$ independent of $h$, such that there holds that
\begin{align}
	\sup \limits_{\mathbf{0} \neq \mathbf{w}_h \in \mathbf{V}_h} \dfrac{b_{h}(\mathbf{w}_h, q_h)}{\|\mathbf{w}_h\|_{1,\Omega}} \geq {\beta}_1 \|q_h\|_{0,\Omega} - {\beta}_2 [\mathcal{L}_{2,h}(q_h,q_h)]^{\frac{1}{2}} \qquad \text{for all} \,\, q_h \in Q_h.  \label{inf-sup1}
\end{align}
\end{lemma}

\begin{proof}
The proof follows from \cite[Lemma 4.3]{mishra2025equal}.
\end{proof}

\begin{lemma} \label{wellposed-1}
Under the assumption \textbf{(A0)} and \textbf{(A1)}, and let $\widehat{\mathbf{u}}_h \in \mathbf{V}_h$ and $\widehat{\theta}_h \in \Sigma_h$ be given. Then for any $(\mathbf{u}_h, p_h) \in \mathbf{V}_h \times Q_h$, there exists a positive constant $C_{o,V}$, independent of $h$, satisfying the following:
\begin{align}
	\sup \limits_{ (\mathbf{0},0) \neq (\mathbf{v}_h, q_h) \in \mathbf{V}_h \times Q_h} \dfrac{A_{V,h} [ \widehat{\theta}_h; (\mathbf{u}_h, p_h), (\mathbf{v}_h, q_h)] + c^S_{V,h}(\widehat{\mathbf{u}}_h; \mathbf{u}_h, \mathbf{v}_h) }{\vertiii{(\mathbf{v}_h, q_h)} }  &\geq C_{o,V} \vertiii{(\mathbf{u}_h, p_h)}, \label{d1well1} \\
	|A_{V,h} [ \widehat{\theta}_h; (\mathbf{u}_h, p_h), (\mathbf{v}_h, q_h)]| &\leq C \vertiii{(\mathbf{u}_h, p_h)}\, \vertiii{(\mathbf{v}_h, q_h)},   \label{d1well2}
\end{align}
where the constant $C_{o,V}$ depends on $\mu_\ast$, $\mu^\ast$, $\lambda_{1\ast}$, $\beta_1$ and $\beta_2$.
\end{lemma}

\begin{proof}
The proof follows from  \cite[Theorem 4.2]{mishra2025unified}. 
\end{proof}

We now present the well-posedness of the problem \eqref{dnvem1} as follows:
\begin{lemma}  \label{lm-ddV}
Under the assumption \textbf{(A0)}, for any given $\widehat{\theta}_h \in \Sigma_h$ 
and  $\widehat{\mathbf{u}}_h \in \mathbf{V}_h$, the problem \eqref{dnvem1} has a unique solution $(\mathbf{u}_h,p_h) \in \mathbf{V}_h \times Q_h$ such that it satisfies
\begin{align}
	\vertiii{(\mathbf{u}_h, p_h)} \leq C^{-1}_{o,V}C_\alpha \big[ \|\mathbf{f}\|_{0,\Omega} +  \|\mathbf{g}\|_{0,\Omega} \vertiii{\widehat{\theta}_h} \big].  \label{d1b0}
\end{align}
\end{lemma}

\begin{proof}
The existence and uniqueness of the discrete solution to problem \eqref{dnvem1} can be easily derived from Lemma \ref{wellposed-1}, for more details see \cite{bookgirault}. Further, employing the bound \eqref{d1well1}, we obtain
\begin{align}
	C_{o,V} \vertiii{(\mathbf{u}_h, p_h)} \vertiii{(\mathbf{v}_h, q_h)} \leq A_{V,h} [ \widehat{\theta}_h; (\mathbf{u}_h, p_h), (\mathbf{v}_h, q_h)] + c^S_{V,h}(\widehat{\mathbf{u}}_h; \mathbf{u}_h, \mathbf{v}_h) = F_{h,\widehat{\theta}_h}(\mathbf{v}_h). \nonumber
\end{align} 
Thus, the result \eqref{d1b0} readily follows from \eqref{est1}. 
\end{proof}

Hereafter, we will show the well-posedness of \eqref{dnvem2}. To do this, we first introduce $\theta_h= \theta^0_h+ \theta^\partial_h \in \Sigma_{D,h}$ such that $\theta_h^0 \in \Sigma_{0,h}$ and $\theta^\partial_h \in \Sigma_{D,h}$.
\begin{lemma} \label{lm-ddT}
Under the assumption \textbf{(A0)} and \textbf{(A1)}, for any given $\widehat{\mathbf{u}}_h \in \mathbf{V}_h$ and $\widehat{\theta}_h \in \Sigma_h$, the decoupled problem \eqref{dnvem2} has a unique solution $\theta_h \in \Sigma_{D,h}$ satisfying
\begin{align}
	\vertiii{\theta_h} \leq C^{-1}_{o,T} \big[ C_p \|\mathcal{Q}\|_{0,\Omega} + \widetilde{C}_\Omega \big(C_{k,\lambda} + C_\lambda C_{o,T} + \widehat{N_0} \vertiii{(\widehat{\mathbf{u}}_h, \widehat{{p}}_h)}\big)  \|\theta_{D}\|_{H^{1/2}(\partial \Omega)}\big]. \label{dd2-f0}
\end{align}
\end{lemma}

\begin{proof}
We will derive the inequality \eqref{dd2-f0} in the following two steps: \newline
\noindent \textbf{Step 1.} Let $\theta^\partial_h \in \Sigma_{D,h}$ be the solution of problem 
\begin{align}
	\begin{cases}
		\text{find}\,\, \theta^\partial_h \in \Sigma_{D,h},\,\, \text{such that} \\
		\int_\Omega \nabla \theta^\partial_h \cdot \nabla \psi_h =0 \quad \text{for all} \, \psi_h \in \Sigma_{0,h}. \nonumber
	\end{cases}
\end{align}
Let $\theta^\partial_h:= \theta^\partial_{h,0} + \theta^\partial_{h,1}$ such that $\theta^\partial_{h,1}$ be a discrete lifting of $\theta^\partial_{h} \in \Sigma_{D,h}$, i.e., $\theta^\partial_{{h,1}{|\partial \Omega}}=\theta_{D,h}$. Moreover, it holds
\begin{align}
	\|\theta^\partial_{h,1}\|_{1,\Omega} \leq C_{0} \|\theta_{D,h}\|_{H^{1/2}(\partial \Omega)}, \label{dd2-f2}		
\end{align}	
where $C_0>0$ depends on $\Omega$.	\newline
Employing this decomposition of $\theta^\partial_{h}$, and applying the Cauchy-Schwarz inequality, we infer
\begin{align}
	\|\nabla \theta^\partial_{h,0}\|_\Omega \leq \|\nabla \theta^\partial_{h,1} \|_\Omega. \label{dd2-f1}
\end{align}
Recalling the Poincar$\acute{\text{e}}$ inequality on $\Sigma_{0,h}$, yields $\|\theta^\partial_{h,0}\|_{1,\Omega} \leq (1+C_p) \| \nabla\theta^\partial_{h,0}\|_{0,\Omega}$. Therefore, combining the above, it follows that
\begin{align}
	\|\theta^\partial_{h}\|_{1,\Omega} & \leq (1+C_p) \| \nabla\theta^\partial_{h,0}\|_{0,\Omega} +  \|\theta^\partial_{h,1}\|_{1,\Omega}  \leq \big( 2 + C_p\big) \|\theta^\partial_{h,1}\|_{1,\Omega} \leq  C_{\Omega} \|\theta_{D,h}\|_{H^{1/2}(\Omega)}, \label{dbound}
\end{align}	
where $C_{\Omega}:= C_0 \big( 2 + C_p\big)$. Note that the existence and uniqueness of $\theta^\partial_h \in \Sigma_{D,h}$ is guaranteed by Lax-Milgram lemma, which satisfies \eqref{dbound}. \newline
	\noindent \textbf{Step 2.} Recalling the decomposition of $\theta_h= \theta^0_h+ \theta^\partial_h$, the decoupled problem \eqref{dnvem2} can be given as follows
\begin{align}
	\begin{cases}
		\text{find}\,\, \theta^0_h \in \Sigma_{0,h},\,\, \text{such that} \\
		A_{T,h}(\widehat{\theta}_h, \widehat{\mathbf{u}}_h; \theta^0_h, \psi_h)= l(\psi_h) \quad \text{for all} \,\, \psi_h \in \Sigma_{0,h}, \label{td2}
	\end{cases}
\end{align}
where $l(\psi_h):=(\mathcal{Q}_h, \psi_h) - A_{T,h}(\widehat{\theta}_h, \widehat{\mathbf{u}}_h; \theta^\partial_h, \psi_h)$.
Now replacing $\psi_h$ by $\theta^0_h$ and employing \eqref{est3}, gives
\begin{align}
	A_{T,h}(\widehat{\theta}_h, \widehat{\mathbf{u}}_h; \theta^0_h, \theta^0_h) \geq C_{o,T} \vertiii{\theta_0}^2. \nonumber
\end{align}
Note that $l(\psi_h)$ is a bounded linear functional on $\Sigma_{0,h}$. Then, using the Lax-Milgram Lemma, the problem \eqref{td2} has a unique solution $\theta^0_h \in \Sigma_{0,h}$. Consequently, problem \eqref{dnvem2} has a unique solution $\theta_h \in \Sigma_{D,h}$. Furthermore, replacing $\psi_h$ by $\theta^0_h$, we obtain
\begin{align}
	C_{o,T} \vertiii{\theta^0_h}^2 &\leq A_{T,h}(\widehat{\theta}_h, \widehat{\mathbf{u}}_h; \theta^0_h, \theta^0_h) \nonumber \\
	&\leq C_p\|\mathcal{Q}\|_{0,\Omega} \|\nabla \theta^0_h\|_{0,\Omega} + |a_{T,h}(\widehat{\theta}_h; \theta^\partial_h, \theta^0_h) + \mathcal{L}_{T,h}( \theta^\partial_h, \theta^0_h)|+ |c^S_{T,h}(\widehat{\mathbf{u}}_h; \theta^\partial_h, \theta^0_h)| \nonumber\\
	&\leq \big[C_p\|\mathcal{Q}\|_{0,\Omega} + \big(C_{\kappa,\lambda} + \widehat{N_0} \vertiii{(\widehat{\mathbf{u}}_h, p_h)} \big) C_{\Omega} \| \theta_{D,h}\|_{H^{1/2}(\partial \Omega)} \big] \|\nabla \theta^0_h\|_{0,\Omega}, \label{d2well3}
\end{align}
where the last line is obtained using \eqref{dcnt-st}, \eqref{est2} and \eqref{dbound}. From \eqref{est4}, we can write
\begin{align}
	\vertiii{\theta^\partial_h} \leq C_\lambda \|\nabla \theta^\partial_h\|_{0,\Omega} \leq C_\lambda C_\Omega \|\theta_{D,h}\|_{H^{1/2}(\partial\Omega)}. \label{d2well4}
\end{align}

Thus, the bound \eqref{dd2-f0} is evident from \eqref{d2well3}, \eqref{d2well4}
and using the classical polynomial interpolation estimates.  
\end{proof}

\subsection{Well-posedness of the coupled problem}

Hereafter, we aim to find a convex and compact ball to satisfy the hypothesis of the Brouwer fixed-point theorem.
\begin{lemma} \label{fball}
Let $(\mathbf{u}_h, p_h, \theta_h) \in \mathbf{V}_h \times Q_h \times \Sigma_{D,h}$ be any solution of \eqref{nvem} with $\theta_h=\theta^0_h + \theta_h^\partial$, where $\theta^0_h \in \Sigma_{0,h}$ and $\theta^\partial_h \in \Sigma_{D,h}$. Furthermore, we assume the following assumption on the data:
\begin{align}
	\widetilde{C}_\Omega C_\alpha C^{-1}_{o,V}C^{-1}_{o,T} \widehat{N_0} \|\mathbf{g}\|_{0,\Omega} \|\theta_{D}\|_{H^{1/2}(\partial \Omega)} \leq \frac{1}{2}. \label{cond1}
\end{align}
Then, there exist constants $\widehat{C}_{\mathbf{f}, \mathbf{g},\mathcal{Q}}$, $\widehat{C}_\mathbf{u}$ and $\widehat{C}_\theta$ such that
\begin{align}
	\vertiii{(\mathbf{u}_h,p_h)} \leq \widehat{C}_{\mathbf{f}, \mathbf{g},\mathcal{Q}} + \widehat{C}_\mathbf{u} \|\theta_{D}\|_{H^{1/2}(\partial \Omega)}, \qquad \vertiii{\theta_h} \leq C_p C^{-1}_{0,T} \|\mathcal{Q}\|_{0,\Omega} + \widehat{C}_\theta \|\theta_{D}\|_{H^{1/2}(\partial \Omega)},  
\end{align}
where the constants are given by
\begin{align}
	\widehat{C}_{\mathbf{f}, \mathbf{g},\mathcal{Q}}&:= 2C^{-1}_{o,V}C_\alpha \big[ \|\mathbf{f}\|_{0,\Omega} +  C_p C^{-1}_{o,T} \|\mathbf{g}\|_{0,\Omega}  \|\mathcal{Q}\|_{0,\Omega} \big], \nonumber \\
	\widehat{C}_\mathbf{u}&:= 2  \widetilde{C}_\Omega C_\alpha C^{-1}_{o,V}C^{-1}_{o,T} \|\mathbf{g}\|_{0,\Omega} \big(C_{k,\lambda} + C_\lambda C_{o,T}\big), \nonumber\\
	\widehat{C}_\theta&:= \widetilde{C}_\Omega C^{-1}_{o,T}\big(2(C_{k,\lambda} + C_\lambda C_{o,T}) + \widehat{N_0} \widehat{C}_{\mathbf{f},\mathbf{g}, \mathcal{Q}} \big). \nonumber
\end{align}
\end{lemma}

\begin{proof}
Let $(\mathbf{u}_h,p_h,\theta_h) \in \mathbf{V}_h \times Q_h \times \Sigma_{D,h}$ be any solution of the problem \eqref{nvem}. 
Employing Lemma \ref{lm-ddV}, we obtain
\begin{align}
	\vertiii{(\mathbf{u}_h, p_h)} \leq C^{-1}_{o,V}C_\alpha \big[ \|\mathbf{f}\|_{0,\Omega} +  \|\mathbf{g}\|_{0,\Omega} \vertiii{\theta_h} \big]. \label{ball1}
\end{align}
Again, replacing $\psi_h \in \Sigma_{0,h}$ by $\theta^0_h \in \Sigma_{0,h}$, there holds that
\begin{align}
	A_{T,h}(\theta_h, \mathbf{u}_h; \theta^0_h,\theta^0_h) = (\mathcal{Q}, \theta^0_h) -A_{T,h}(\theta_h, \mathbf{u}_h; \theta^\partial_h,\theta^0_h).\nonumber
\end{align}
Applying Lemma \ref{lm-ddT}, we arrive at
\begin{align}
	\vertiii{\theta_h} \leq C^{-1}_{o,T} \big[ C_p \|\mathcal{Q}\|_{0,\Omega} + \widetilde{C}_\Omega \big(C_{k,\lambda} + C_\lambda C_{o,T} + \widehat{N_0} \vertiii{(\mathbf{u}_h, p_h)}\big)  \|\theta_{D}\|_{H^{1/2}( \partial \Omega)}\big].  \label{balll2}
\end{align}
We now use \eqref{balll2} in \eqref{ball1}, we infer
\begin{align}
	\vertiii{(\mathbf{u}_h, p_h)} &\leq C^{-1}_{o,V}C_\alpha \big[ \|\mathbf{f}\|_{0,\Omega} +  C^{-1}_{o,T} \|\mathbf{g}\|_{0,\Omega}  \big( C_p \|\mathcal{Q}\|_{0,\Omega} + \widetilde{C}_\Omega \big(C_{k,\lambda} + C_\lambda C_{o,T}\, + \nonumber \\
	& \qquad  \widehat{N_0} \vertiii{(\mathbf{u}_h, p_h)}\big)  \|\theta_{D}\|_{H^{1/2}(\partial \Omega)}\big) \big]. \nonumber 
	\intertext{Employing \eqref{cond1}, we arrive at}
	\vertiii{(\mathbf{u}_h, p_h)} &\leq 2C^{-1}_{o,V}C_\alpha \big[ \|\mathbf{f}\|_{0,\Omega} + C^{-1}_{o,T} \|\mathbf{g}\|_{0,\Omega}  \big( C_p \|\mathcal{Q}\|_{0,\Omega} + \widetilde{C}_\Omega \big(C_{k,\lambda} + C_\lambda C_{o,T} \big)\big)  \|\theta_{D}\|_{H^{1/2}(\partial \Omega)}\big) \big] \nonumber \\
	&\leq 2C^{-1}_{o,V}C_\alpha \big[ \|\mathbf{f}\|_{0,\Omega} +  C_p C^{-1}_{o,T} \|\mathbf{g}\|_{0,\Omega}  \|\mathcal{Q}\|_{0,\Omega} \big] \nonumber \\
	& \qquad + 2 \widetilde{C}_\Omega C_\alpha C^{-1}_{o,V}C^{-1}_{o,T} \|\mathbf{g}\|_{0,\Omega} \big(C_{k,\lambda} + C_\lambda C_{o,T}\big)  \|\theta_{D}\|_{H^{1/2}(\partial \Omega)} \nonumber \\
	& \leq \widehat{C}_{\mathbf{f},\mathbf{g}, \mathcal{Q}} + \widehat{C}_\mathbf{u} \|\theta_{D}\|_{H^{1/2}(\partial \Omega)}. \label{ball3}
\end{align}
Substituting \eqref{ball3} in \eqref{balll2}, it holds that 
\begin{align}
	\vertiii{\theta_h} &\leq C^{-1}_{o,T} \big[ C_p \|\mathcal{Q}\|_{0,\Omega} + \widetilde{C}_\Omega \big(C_{k,\lambda} + C_\lambda C_{o,T} + \widehat{N_0} \widehat{C}_{\mathbf{f},\mathbf{g}, \mathcal{Q}} \big)  \|\theta_{D}\|_{H^{1/2}(\partial \Omega)} \nonumber \\ & \qquad +  \widetilde{C}_\Omega \widehat{N_0} \widehat{C}_\mathbf{u} \|\theta_{D}\|_{H^{1/2}(\partial \Omega)} \|\theta_{D}\|_{H^{1/2}(\partial \Omega)}\big]. \nonumber
\end{align}
Using \eqref{cond1} and the definition of $\widehat{C}_\mathbf{u}$, we get
\begin{align}
	\vertiii{\theta_h} &\leq C^{-1}_{o,T} \big[ C_p \|\mathcal{Q}\|_{0,\Omega} + \widetilde{C}_\Omega \big(2(C_{k,\lambda} + C_\lambda C_{o,T}) + \widehat{N_0} \widehat{C}_{\mathbf{f},\mathbf{g}, \mathcal{Q}} \big)  \|\theta_{D}\|_{H^{1/2}(\partial \Omega)}\big] . \nonumber 
\end{align}
which completes the proof. 
\end{proof}

Hereafter, we introduce the following compact and convex ball, defined by
\begin{align}
\mathcal{M} :=\Big\{(\widehat{\mathbf{u}}_h, \widehat{p}_h, \widehat{\theta}_h) \in \mathbf{V}_h \times Q_h \times \Sigma_{D,h}& \,\, \text{such that} \,\, \vertiii{(\widehat{\mathbf{u}}_h, \widehat{p}_h)} \leq \widehat{C}_{\mathbf{f}, \mathbf{g},\mathcal{Q}} + \widehat{C}_\mathbf{u} \|\theta_{D}\|_{H^{1/2}(\partial \Omega)}, \nonumber \\ &\vertiii{\widehat{\theta}_h} \leq C_p C^{-1}_{0,T} \|\mathcal{Q}\|_{0,\Omega} + \widehat{C}_\theta \|\theta_{D}\|_{H^{1/2}(\partial \Omega)} \Big\}. \label{defball}
\end{align}
We now define a mapping $\mathbb K : \mathcal{M} \rightarrow \mathcal{M} $ such that
\begin{align}
(\widehat{\mathbf{u}}_h,\widehat{{p}}_h, \widehat{\theta}_h) \rightarrow \mathbb K((\widehat{\mathbf{u}}_h,\widehat{{p}}_h, \widehat{\theta}_h)) = ({\mathbf{u}}_h,{{p}}_h, {\theta}_h), \nonumber
\end{align}
where $({\mathbf{u}}_h,{{p}}_h, {\theta}_h)$ denotes the solution of the decoupled problems \eqref{dnvem1} and \eqref{dnvem2}.

\begin{lemma} \label{selfcont}
Under the assumptions of Lemma \ref{fball}, for any $(\widehat{\mathbf{u}}_h,\widehat{{p}}_h, \widehat{\theta}_h) \in \mathcal{M}$, the mapping $\mathbb K$ is well-defined and $ \mathbb{K}(\mathcal{M}) \subset \mathcal{M}$.
\end{lemma}

\begin{proof}
From Lemmas \eqref{lm-ddV} and \eqref{lm-ddT}, map $\mathbb K$ is well-defined. Let $(\widehat{\mathbf{u}}_h,\widehat{{p}}_h, \widehat{\theta}_h) \in \mathcal{M}$  such that $K((\widehat{\mathbf{u}}_h,\widehat{{p}}_h, \widehat{\theta}_h)) = ({\mathbf{u}}_h,{{p}}_h, {\theta}_h)$ is a solution of the decoupled problems \eqref{dnvem1} and \eqref{dnvem2}. We derive the second part in two steps: \newline
\textbf{Step 1.} Since for given $(\widehat{\mathbf{u}}_h,\widehat{{p}}_h, \widehat{\theta}_h) \in \mathcal{M}$, $ ({\mathbf{u}}_h,{{p}}_h)$ is a solution of \eqref{dnvem1}, Lemma \ref{lm-ddV} gives that
\begin{align}
	\vertiii{(\mathbf{u}_h, p_h)} &\leq C^{-1}_{o,V}C_\alpha \big[ \|\mathbf{f}\|_{0,\Omega} +  \|\mathbf{g}\|_{0,\Omega} \big( C_p C^{-1}_{0,T} \|\mathcal{Q}\|_{0,\Omega} + \widehat{C}_\theta \|\theta_{D}\|_{H^{1/2}(\partial \Omega)}   \big) \big] \nonumber \\
	&\leq \frac{1}{2} \widehat{C}_{\mathbf{f}, \mathbf{g}, \mathcal{Q}} +  C^{-1}_{o,V} C_\alpha  \widehat{C}_\theta \|\mathbf{g}\|_{0,\Omega} \|\theta_{D}\|_{H^{1/2}(\partial \Omega)} \nonumber \\
	&\leq \frac{1}{2} \widehat{C}_{\mathbf{f}, \mathbf{g}, \mathcal{Q}} + C_\alpha C^{-1}_{o,V} \|\mathbf{g}\|_{0,\Omega} \big[\widetilde{C}_\Omega C^{-1}_{o,T}\big(2(C_{k,\lambda} + C_\lambda C_{o,T}) + \widehat{N_0} \widehat{C}_{\mathbf{f},\mathbf{g}, \mathcal{Q}} \big)\big] \|\theta_{D}\|_{H^{1/2}(\partial \Omega)} \nonumber \\
	&\leq \frac{1}{2} \widehat{C}_{\mathbf{f}, \mathbf{g}, \mathcal{Q}} + 2 \widetilde{C}_\Omega C_\alpha C^{-1}_{o,V} C^{-1}_{o,T} \|\mathbf{g}\|_{0,\Omega} (C_{k,\lambda} + C_\lambda C_{o,T}) \|\theta_{D}\|_{H^{1/2}(\partial \Omega)} \nonumber \\ 
	& \qquad + \big( \widetilde{C}_\Omega C_\alpha C^{-1}_{o,V} C^{-1}_{o,T} \|\mathbf{g}\|_{0,\Omega} \widehat{N_0}  \|\theta_{D}\|_{H^{1/2}(\partial \Omega)}) \widehat{C}_{\mathbf{f},\mathbf{g}, \mathcal{Q}}. \nonumber
	\intertext{Using \eqref{cond1} and  Lemma \ref{fball}, we can conclude that } 
	\vertiii{(\mathbf{u}_h, p_h)} & \leq \frac{1}{2} \widehat{C}_{\mathbf{f}, \mathbf{g}, \mathcal{Q}} + \widehat{C}_\mathbf{u} \|\theta_{D}\|_{H^{1/2}(\partial \Omega)} + \frac{1}{2} \widehat{C}_{\mathbf{f}, \mathbf{g}, \mathcal{Q}}. \label{cond2}
\end{align}
\textbf{Step 2.}  Since for given $(\widehat{\mathbf{u}}_h,\widehat{{p}}_h, \widehat{\theta}_h) \in \mathcal{M}$, $ \theta_h$ is a solution of \eqref{dnvem2}, Lemma \ref{lm-ddT} and Lemma \ref{fball} claim that
\begin{align}
	\vertiii{\theta_h} &\leq C^{-1}_{o,T} \big[ C_p \|\mathcal{Q}\|_{0,\Omega} + \widetilde{C}_\Omega \big(C_{k,\lambda} + C_\lambda C_{o,T} \big)\|\theta_{D}\|_{H^{1/2}(\Omega)}\big] + \widetilde{C}_\Omega C^{-1}_{o,T} \widehat{N_0} \vertiii{(\widehat{\mathbf{u}}_h, \widehat{{p}}_h)}  \|\theta_{D}\|_{H^{1/2}(\Omega)} \nonumber \\
	&\leq C^{-1}_{o,T} \big[ C_p \|\mathcal{Q}\|_{0,\Omega} + \widetilde{C}_\Omega \big(C_{k,\lambda} + C_\lambda C_{o,T} +  \widehat{N_0} C_{\mathbf{f}, \mathbf{g},\mathcal{Q}}\big) \|\theta_{D}\|_{H^{1/2}(\Omega)}\big] \, + \nonumber \\ & \qquad \widetilde{C}_\Omega C^{-1}_{o,T} \widehat{N_0} \widehat{C}_{\mathbf{u}}  \|\theta_{D}\|^2_{H^{1/2}(\Omega)}. \nonumber 
	\intertext{Using the definition of $\widehat{C}_\mathbf{u}$ and \eqref{cond1}, we have}
	\vertiii{\theta_h} &\leq C^{-1}_{o,T} \big[ C_p \|\mathcal{Q}\|_{0,\Omega} + \widetilde{C}_\Omega \big(C_{k,\lambda} + C_\lambda C_{o,T} +  \widehat{N_0} C_{\mathbf{f}, \mathbf{g},\mathcal{Q}}\big) \|\theta_{D}\|_{H^{1/2}(\Omega)}\big] \, + \nonumber \\ & \qquad \widetilde{C}_\Omega C^{-1}_{o,T} \big(C_{k,\lambda} + C_\lambda C_{o,T} \big)  \|\theta_{D}\|_{H^{1/2}(\Omega)} \nonumber \\
	&\leq C^{-1}_{o,T} \big[ C_p \|\mathcal{Q}\|_{0,\Omega} + \widetilde{C}_\Omega \big(2(C_{k,\lambda} + C_\lambda C_{o,T}) +  \widehat{N_0} C_{\mathbf{f}, \mathbf{g},\mathcal{Q}}\big) \|\theta_{D}\|_{H^{1/2}(\Omega)}\big]. \label{cond3}
\end{align}
Thus, the bounds \eqref{cond2} and \eqref{cond3} guarantee that $\mathbb K(\mathcal{M}) \subset \mathcal{M}$. 
\end{proof}
\begin{lemma} \label{continuity}
The mapping $\mathbb K$ is a continuous map on $\mathcal{M}$.
\end{lemma}

\begin{proof} The proof consists of the following steps: \newline
\textbf{Step 1.}	Let $\big\{(\widehat{\mathbf{u}}_{h,m}, \widehat{p}_{h,m}, \widehat{\theta}_{h,m}) \big\}_{m \in \mathbb N} \subset \mathcal{M}$ be a sequence converging to $(\widehat{\mathbf{u}}_{h}, \widehat{p}_{h}, \widehat{\theta}_{h})  \in \mathcal{M}$, i.e.
\begin{align}
	\vertiii{(\widehat{\mathbf{u}}_{h,m}, \widehat{p}_{h,m}) - (\widehat{\mathbf{u}}_{h}, \widehat{p}_{h})} \xrightarrow{m\rightarrow \infty} \mathbf{0} \qquad \text{and} \,\, \quad \vertiii{\widehat{\theta}_{h,m} - \widehat{\theta}_h} \xrightarrow{ m \rightarrow \infty} 0. \label{cnts1}
\end{align}
Further, we assume that $(\mathbf{u}_h, p_h, \theta_h) \in \mathcal{M}$ and $\big\{(\mathbf{u}_{h,m}, p_{h,m}, \theta_{h,m}) \big\}_{m \in \mathbb N} \subset \mathcal{M}$ such that for each $m \in \mathbb N$
\begin{align}
	\mathbb K((\widehat{\mathbf{u}}_{h,m}, \widehat{p}_{h,m}, \widehat{\theta}_{h,m}) )=(\mathbf{u}_{h,m}, p_{h,m}, \theta_{h,m})\quad  \text{and} \quad \mathbb K((\widehat{\mathbf{u}}_{h}, \widehat{p}_{h}, \widehat{\theta}_h)) = (\mathbf{u}_h, p_h, \theta_h). \label{cnts2}
\end{align}
Then, we will show that
\begin{align}
	\vertiii{({\mathbf{u}}_{h,m},{p}_{h,m}) - ({\mathbf{u}}_{h}, {p}_{h})} \xrightarrow{m\rightarrow \infty} \mathbf{0} \qquad \text{and} \,\, \quad \vertiii{{\theta}_{h,m} - {\theta}_h} \xrightarrow{ m \rightarrow \infty} 0. \label{cnts3}
\end{align}
From \eqref{cnts2}, we obtain the following systems
\begin{align}
	\begin{cases}
		A_{V,h}[ \widehat{\theta}_{h,m}; (\mathbf{u}_{h,m}, p_{h,m}), (\mathbf{v}_h, q_h)]+ c^S_{V,h}(\widehat{\mathbf{u}}_{h,m}; \mathbf{u}_{h,m}, \mathbf{v}_h)= F_{h,\widehat{\theta}_{h,m}}(\mathbf{v}_h), \\
		A_{T,h}(\widehat{\theta}_{h,m}, \widehat{\mathbf{u}}_{h,m}; \theta_{h,m}, \psi_h)=(\mathcal{Q}_h, \psi_h), 
	\end{cases}
	\label{dnvem3}
\end{align}
and
\begin{align}
	\begin{cases}
		A_{V,h}[ \widehat{\theta}_{h}; (\mathbf{u}_{h}, p_{h}), (\mathbf{v}_h, q_h)]+ c^S_{V,h}(\widehat{\mathbf{u}}_{h}; \mathbf{u}_h, \mathbf{v}_h)= F_{h,\widehat{\theta}_{h}}(\mathbf{v}_h),  \\
		A_{T,h}(\widehat{\theta}_{h}, \widehat{\mathbf{u}}_{h}; \theta_h, \psi_h)=(\mathcal{Q}_h, \psi_h).
	\end{cases}
	\label{dnvem4}
\end{align}	
\textbf{Step 2.} Subtracting the first equation of \eqref{dnvem4} from the first equation of \eqref{dnvem3}, we arrive at
\begin{align}
	A_{V,h}[\widehat{\theta}_{h,m}; &(\mathbf{u}_{h,m} - \mathbf{u}_h, {p}_{h,m} -p_h),(\mathbf{v}_h, q_h) ] + c^S_{V,h}(\widehat{\mathbf{u}}_{h,m}; \mathbf{u}_{h,m}-\mathbf{u}_h, \mathbf{v}_h)\nonumber \\
	&= F_{h,\widehat{\theta}_{h,m}}(\mathbf{v}_h) - F_{h,\widehat{\theta}_{h}}(\mathbf{v}_h) + 	A_{V,h}[ \widehat{\theta}_{h}; (\mathbf{u}_{h}, p_{h}), (\mathbf{v}_h, q_h)] -   A_{V,h}[\widehat{\theta}_{h,m}; (\mathbf{u}_h, p_h),(\mathbf{v}_h, q_h) ] \nonumber \\ & \qquad + c^S_{V,h}(\widehat{\mathbf{u}}_{h}; \mathbf{u}_h, \mathbf{v}_h) -c^S_{V,h}(\widehat{\mathbf{u}}_{h,m}; \mathbf{u}_{h}, \mathbf{v}_h) \nonumber \\
	&= F_{h,\widehat{\theta}_{h,m}}( \mathbf{v}_h) - F_{h,\widehat{\theta}_{h}}(\mathbf{v}_h) + 	a_{V,h} (\widehat{\theta}_{h}; \mathbf{u}_{h}, \mathbf{v}_h) - 
	a_{V,h}(\widehat{\theta}_{h,m}; \mathbf{u}_h, \mathbf{v}_h) \nonumber \\ & \qquad + c^S_{V,h}(\widehat{\mathbf{u}}_{h} - \widehat{\mathbf{u}}_{h,m}; \mathbf{u}_h, \mathbf{v}_h). \nonumber
\end{align} 
Using Lemmas \ref{wellposed-1} and \ref{vacca}, the Sobolev embedding theorem and \eqref{dcnt-sv}, we infer
\begin{align}
	C_{o,V} \vertiii{\mathbf{u}_{h,m} - \mathbf{u}_h, {p}_{h,m} -p_h} &\leq \big[\alpha C_q^2 C^2_{1 \hookrightarrow 4} \|\mathbf{g}\|_{0,\Omega} \vertiii{\widehat{\theta}_{h,m} - \widehat{\theta}_h} +  C \|\nabla \mathbf{u}_{h}\|_{0,3, \Omega} \vertiii{\widehat{\theta}_{h,m} - \widehat{\theta}_h}  \nonumber \\& + \widehat{N_0} \|\nabla \mathbf{u}_{h}\|_{0, \Omega} \vertiii{(\widehat{\mathbf{u}}_{h} - \widehat{\mathbf{u}}_{h,m}, \widehat{p}_{h} - \widehat{p}_{h,m}} \big]. \nonumber 
\end{align}
Applying \eqref{cnts1}, we obtain $\vertiii{(\mathbf{u}_{h,m} - \mathbf{u}_h, {p}_{h,m} -p_h)} \xrightarrow{m \rightarrow \infty } \mathbf 0$. \newline
\textbf{Step 3.} The remaining task is to prove that $\vertiii{\theta_{h,m}- \theta_h}  \xrightarrow{m \rightarrow \infty} 0$. Concerning this, we subtract the second equation of \eqref{dnvem4} from the second equation of \eqref{dnvem3}:
\begin{align}
	A_{T,h}(\widehat{\theta}_{h,m}, \widehat{\mathbf{u}}_{h,m}; \theta_{h,m}-\theta_h, \psi_h ) &= A_{T,h}(\widehat{\theta}_{h}, \widehat{\mathbf{u}}_{h}; \theta_h, \psi_h ) - A_{T,h}(\widehat{\theta}_{h,m}, \widehat{\mathbf{u}}_{h,m}; \theta_h, \psi_h ) \nonumber \\
	&= a_{T,h}(\widehat{\theta}_{h}; \theta_h, \psi_h ) - a_{T,h}(\widehat{\theta}_{h,m}; \theta_h, \psi_h )  +  c^S_{T,h}(\widehat{\mathbf{u}}_h-\widehat{\mathbf{u}}_{h,m}; \theta_h, \psi_h). \nonumber   
\end{align}
Replacing $\psi_h=\theta_{h,m}-\theta_h \in \Sigma_{0,h}$, and using Lemma \ref{estimate2}, Lemma \ref{vacca} and \eqref{dcnt-st}, we obtain
\begin{align}
	C_{o,T} \vertiii{\theta_{h,m} - \theta_h}^2 \leq  \big[ C \|\nabla \theta_{h}\|_{0,3, \Omega} \vertiii{\widehat{\theta}_{h,m} - \widehat{\theta}_h} + \widehat{N_0} \|\nabla \theta_h\|_{0,\Omega} \|\nabla(\widehat{\mathbf{u}}_h-\widehat{\mathbf{u}}_{h,m}) \|_{0,\Omega}  \big] \| \nabla(\theta_{h,m} - \theta_h)\|_{0,\Omega}. \nonumber 
\end{align}
Employing \eqref{cnts1}, we get $\vertiii{\theta_{h,m}- \theta_h}  \xrightarrow{m \rightarrow \infty} 0$. Thus, the mapping $\mathbb K$ is continuous on $\mathcal{M}$. 
\end{proof}
\begin{theorem}
Let $\mathcal{M}$ be defined by \eqref{defball}. Under the assumption \eqref{cond1}, there exists a solution $(\mathbf{u}_h, p_h, \theta_h) \in \mathbf{V}_h \times Q_h \times \Sigma_{D,h}$ such that it satisfies
\begin{align}
	\vertiii{({\mathbf{u}}_h, {p}_h)} \leq \widehat{C}_{\mathbf{f}, \mathbf{g},\mathcal{Q}} + \widehat{C}_\mathbf{u} \|\theta_{D}\|_{H^{1/2}(\partial \Omega)}, \qquad \vertiii{{\theta}_h} \leq C_p C^{-1}_{0,T} \|\mathcal{Q}\|_{0,\Omega} + \widehat{C}_\theta \|\theta_{D}\|_{H^{1/2}(\partial \Omega)}. \label{dexistence}
\end{align}
\end{theorem}
\begin{proof}
The proof readily follows from the Brouwer fixed-point theorem and Lemmas \ref{selfcont} and \ref{continuity}. 
\end{proof}


\begin{theorem}\label{WELLP}
Under the data assumptions \textbf{(A0)} and \textbf{(A1)}, let $(\mathbf{u}_h,p_h,\theta_h) \in \mathbf{V}_h \times Q_h \times \Sigma_{D,h} $ be the solution of problem \eqref{nvem}. Further, we assume the following: \newline
(i)\,  the data satisfies
\begin{align}
	C_{0,V}^{-1}  \widehat{N_0}\big[ \widehat{C}_{\mathbf{f}, \mathbf{g},Q} + C_p C_{0,T}^{-1} \|\mathcal{Q}\|_{0,\Omega} + (\widehat{C}_{\mathbf{u}} + \widehat{C}_\theta ) \|\theta_{D}\|_{H^{1/2}(\partial\Omega)} \big] <1, \label{disc1}
\end{align}
(ii)\,  let {$(\mathbf{u}_h, \theta_h) \in \mathbf{V}_h \times \Sigma_{D,h} \cap [W^{1,3}(\Omega_h)]^2 \times W^{1,3}(\Omega_h)$} and $\mathbf{g} \in [L^2(\Omega)]^2$, satisfy the following bound
\begin{align}
	C_{0,T}^{-1} \big( 
	\alpha C_q^2 C^2_{1\hookrightarrow4} \|\mathbf{g}\|_{0,\Omega} \| + C_v L_\mu C_q C_{1\hookrightarrow 6} \|\nabla \mathbf{u}_h\|_{0,3,\Omega} + C_t L_\kappa C_q C_{1 \hookrightarrow 6} \|\nabla \theta_h\|_{0,3,\Omega}  \big) < 1, \label{disc2}
\end{align}
then, the problem \eqref{nvem} has a unique solution $ (\mathbf{u}_h,p_h,\theta_h)$.
\end{theorem}

\begin{proof}
Let $(\mathbf{u}_h, p_h,\theta_h) $ and $(\mathbf{u}^\ast_h, p^\ast_h,\theta^\ast_h) $ be two solutions of the coupled problem \eqref{nvem}. We now set $\widetilde{\mathbf{u}}_h := \mathbf{u}_h - \mathbf{u}_h^\ast$,  $\widetilde{p}_h:= p_h - p_h^\ast$ and $\widetilde{\theta}_h:=\theta_h - \theta_h^\ast$. Replacing $(\mathbf{v}_h,q_h)$ by $(\widetilde{\mathbf{u}}_h, \widetilde{p}_h)$ in the first equation of \eqref{nvem}, we infer
\begin{align}
	C_{o,V} \vertiii{(\widetilde{\mathbf{u}}_h,\widetilde{p}_h)}\, \vertiii{(\mathbf{v}_h, q_h)} &\leq  	A_{V,h}[\theta_h; (\widetilde{\mathbf{u}}_h, \widetilde{p}_h), (\mathbf{v}_h, q_h)] + c_{V,h}^S(\mathbf{u}_h; \widetilde{\mathbf{u}}_h, \mathbf{v}_h) \nonumber \\
	&= F_{h,\theta_h}(\mathbf{v}_h) - F_{h,\theta^\ast_h}(\mathbf{v}_h) + 	A_{V,h}[\theta^\ast_h; (\mathbf{u}^\ast_h, p^\ast_h), (\mathbf{v}_h, q_h)] \,- \nonumber \\ & \quad  A_{V,h}[\theta_h; (\mathbf{u}^\ast_h, p^\ast_h), (\mathbf{v}_h, q_h)] + c^S_{V,h}(\mathbf{u}^\ast_h; \mathbf{u}^\ast_h, \mathbf{v}_h) - c^S_{V,h}(\mathbf{u}_h; \mathbf{u}^\ast_h, \mathbf{v}_h) \nonumber \\
	&= F_{h,\theta_h}(\mathbf{v}_h) - F_{h,\theta^\ast_h}({\mathbf{v}}_h) + 	a_{V,h}(\theta^\ast_h; \mathbf{u}^\ast_h, \mathbf{v}_h) - a_{V,h}(\theta_h; \mathbf{u}^\ast_h, \mathbf{v}_h) \nonumber \\ & \qquad + c^S_{V,h}(\mathbf{u}^\ast_h -\mathbf{u}_h; \mathbf{u}^\ast_h, \mathbf{v}_h).  \nonumber 
\end{align}
Using Lemmas \ref{vacca}, the embedding theorems and the property \eqref{dcnt-sv}, it holds that
\begin{align}
	C_{o,V} \vertiii{(\widetilde{\mathbf{u}}_h,\widetilde{p}_h)} &\leq \alpha C_q^2 C^2_{1\hookrightarrow4} \|\mathbf{g}\|_{0,\Omega} \| \nabla\widetilde{\theta}_h\|_{0,\Omega}  + C_v L_\mu C_q C_{1\hookrightarrow 6} \|\nabla \mathbf{u}^\ast_h\|_{0,3,\Omega} \|\nabla \widetilde{\theta}_h\|_{0,\Omega}  \nonumber \\ & \qquad + \widehat{N_0} \|\nabla \mathbf{u}^\ast_h\|_{0,\Omega} \|\nabla \widetilde{\mathbf{u}}_h\|_{0,\Omega}. \nonumber
\end{align}
Rearranging the above terms using the definition of energy norms, we obtain
\begin{align}
	\Big(C_{o,V} - \widehat{N_0} \vertiii{(\mathbf{u}^\ast_h,p^\ast_h)}\Big) \vertiii{(\widetilde{\mathbf{u}}_h,\widetilde{p}_h)} & \leq \Big(\alpha C_q^2 C^2_{1\hookrightarrow4} \|\mathbf{g}\|_{0,\Omega}  + C_v L_\mu C_q C_{1\hookrightarrow 6} \|\nabla \mathbf{u}^\ast_h\|_{0,3,\Omega} \Big) \vertiii{\widetilde{\theta}_h}. \label{dwella}
\end{align}
Again, replacing $\psi_h$ by $\widetilde{\theta}_h \in \Sigma_{0,h}$ in the second equation of \eqref{nvem}, it gives
\begin{align}
	A_{T,h}(\theta_h, \mathbf{u}_h; \widetilde{\theta}_h, \widetilde{\theta}_h)&= A_{T,h}(\theta_h^\ast, \mathbf{u}^\ast_h; \theta^\ast_h, \widetilde{\theta}_h) - A_{T,h}(\theta_h, \mathbf{u}_h; \theta^\ast_h, \widetilde{\theta}_h) \nonumber \\
	&=a_{T,h}(\theta_h^\ast; \theta^\ast_h, \widetilde{\theta}_h) - a_{T,h}(\theta_h; \theta^\ast_h, \widetilde{\theta}_h) + c^S_{T,h}(\mathbf{u}^\ast_h- \mathbf{u}_h; \theta^\ast_h,\widetilde{\theta}_h) \nonumber \\
	&\leq C_t L_\kappa C_q C_{1 \hookrightarrow 6} \|\nabla \theta^\ast_h\|_{0,3,\Omega} \|\nabla \widetilde{\theta}_h\|^2_{0,\Omega} + \widehat{N_0}\|\nabla \theta^\ast_h\|_{0,\Omega} \|\nabla \widetilde{\mathbf{u}}_h\|_{0,\Omega} \|\nabla \widetilde{\theta}_h\|_{0,\Omega}, \nonumber
\end{align}
where the last line is obtained employing \eqref{dcnt-st} and Lemma \ref{vacca}. Further, applying Lemma \ref{estimate2}, we arrive at
\begin{align}
	\Big(C_{o,T} - C_t L_\kappa C_q C_{1 \hookrightarrow 6} \|\nabla \theta^\ast_h\|_{0,3,\Omega} \Big) \vertiii{\widetilde{\theta}_h} \leq \widehat{N_0} \vertiii{\theta^\ast_h} \,\vertiii{ (\widetilde{\mathbf{u}}_h, \widetilde{p}_h)}. \label{dwellb}
\end{align}
Adding \eqref{dwella} and \eqref{dwellb}, and then applying \eqref{disc1} and \eqref{disc2}, we get  $\vertiii{ (\widetilde{\mathbf{u}}_h, \widetilde{p}_h)}=0$  and  $\vertiii{\widetilde{\theta}_h}=0$. 
\end{proof}

\subsection{Error estimates in the energy norm}
This section focuses on deriving the error estimates in the energy norm. We begin with the following approximations. 

\begin{lemma}	\label{lemmaproj2}
(\cite{vem31}) Under the assumption \textbf{(A1)}, for any $\psi \in H^{s+1}(E)$ there exists $\psi_I \in V_h(E)$ for all $E\in \Omega_h$ such that there holds the following
\begin{equation}
	\| \psi - \psi_I\|_{0,E} + h_E |\psi - \psi_I|_{1,E} \leq C h^{1+s}_E \, \|\psi \|_{s+1,E},\quad 0 \leq s\leq k, \label{projlll2}
\end{equation}
where the constant $C>0$ depends only on $k$ and $\delta_0$.
\end{lemma}

We now denote the virtual interpolant of $(\mathbf{u}, p, \theta) \in \mathbf{V} \times Q \times \Sigma_D$ by $(\mathbf{u}_I, p_I,\theta_I) \in \mathbf{V}_h \times Q_h \times \Sigma_{D,h}$. Let us introduce the following notations
\begin{align*}
e^\mathbf{u}&:=\mathbf{u}- \mathbf{u}_h, & e^p&:=p-p_h, & e^\theta&:=\theta- \theta_h, &
e^\mathbf{u}_I &:= \mathbf{u}-\mathbf{u}_I, &  e^p_I &:= p-p_I,  &
e^\theta_I&:=\theta - \theta_I,\\
e^\mathbf{u}_h &:= \mathbf{u}_h-\mathbf{u}_I,&  e^p_h &:= p_h - p_I,  & e^\theta_h&:= \theta_h - \theta_I.
\end{align*}

\noindent \textbf{(A3)} We impose the following regularity assumptions:
\begin{align*}
&\mathbf{u} \in \mathbf{V} \cap [H^{k+1}(\Omega_h)]^2,   && p \in Q \cap H^{k}(\Omega_h),  
&& \theta \in \Sigma_{D} \cap H^{k+1}(\Omega_h),  
&& \mathbf{f} \in [H^{k}(\Omega_h)]^2, \\ 
&\mathbf{g} \in [H^{k}(\Omega_h)]^2, 
&& \mathcal{Q} \in H^{k}(\Omega_h),  
&& \mu \in W^{k,\infty}(\Omega_h),  
&& \kappa \in W^{k,\infty}(\Omega_h).
\end{align*}

In addition, under the assumption \textbf{(A3)}, we have 
\begin{align}
\|\nabla (I - \Pi^{\nabla,E}_k) \theta_I\|_{0,E} &\leq \|\nabla \theta_I - \boldsymbol{\Pi}^{0,E}_k \nabla \theta_I\|_{0,E} \nonumber \\
& \leq \|\nabla \theta-\nabla \theta_I\|_{0,E} + \|\nabla \theta - \boldsymbol{\Pi}^{0,E}_{k-1}  \nabla \theta\|_{0,E} + \|\boldsymbol{\Pi}^{0,E}_{k-1}  (\nabla \theta - \nabla \theta_I)\|_{0,E} \nonumber \\
& \leq h^k_E \|\theta\|_{k+1,E}. \label{err0}
\end{align}
Hereafter, we will first derive the following three lemmas to establish the convergence estimates in the energy norm, assuming that the assumptions \textbf{(A0)} and \textbf{(A1)} are valid.
\begin{lemma} \label{velocity}
Under the assumption \textbf{(A2)}, let $(\mathbf{u},p)$ and $(\mathbf{u}_h,p_h)$ be the solution of \eqref{dvar-1} and \eqref{dnvem1} with $\theta \in \Sigma_D$ and $\theta_h \in \Sigma_{D,h}$, respectively. Further we assume that the stabilization parameters are $\tau_{1,E} \sim h_E$ and $\tau_{2,E} \sim h^2_E$ for all $E \in \Omega_h$. Then for any $(\mathbf{v}_h,q_h) \in \mathbf{V}_h \times Q_h$, the following holds
\begin{align}
	A_{V}[\theta; (\mathbf{u},p), (\mathbf{v}_h, q_h)]-  A_{V,h}[\theta_h; (\mathbf{u}_I,p_I ), (\mathbf{v}_h, q_h)] &\leq Ch^k \Big( \big(1+\mu^\ast + \max_{E\in\Omega_h} \|\mu \|_{k,\infty,E}  \big)\| \mathbf{u} \|_{k+1,\Omega} \,+ \nonumber \\ & \quad L_\mu  \| \nabla \mathbf{u}\|_{0,3,\Omega}  \|\theta\|_{k+1,\Omega} + \| p\|_{k,\Omega} \Big) \vertiii{(\mathbf{v}_h, q_h)} \nonumber \\ & \, + L_\mu C_q C_{1\hookrightarrow 6} \| \nabla \mathbf{u}\|_{0,3,\Omega} \| \nabla e^\theta\|_{0,\Omega} \vertiii{(\mathbf{v}_h, q_h)}, \label{err16}
\end{align}
where the positive constant $C$ does not rely on $h$.
\end{lemma}

\begin{proof} The proof of \eqref{err16} can be derived in the following steps:
\begin{align}
	A_I:&=	A_{V}[\theta; (\mathbf{u},p), (\mathbf{v}_h, q_h)]-  A_{V,h}[\theta_h; (\mathbf{u}_I,p_I ), (\mathbf{v}_h, q_h)] \nonumber \\
	&=  a_V(\theta; \mathbf{u}, \mathbf{v}_h) - a_{V,h}(\theta_h; \mathbf{u}_I, \mathbf{v}_h)  +  b_h(\mathbf{v}_h, p_I) -  b(\mathbf{v}_h, p)  +   b(\mathbf{u},q_h) - b_h(\mathbf{u}_I, q_h) - \mathcal{L}[(\mathbf{u}_I,p_I),(\mathbf{v}_h, q_h)]  \nonumber \\ 
	& =: A_{I,1} + A_{I,2} + A_{I,3} - A_{I,4}. \label{err5}
\end{align}
\textbf{Step 1.} Adding and subtracting suitable terms, we have 
\begin{align}
	A_{I,1} &= \sum_{E \in \Omega_h} \Big( \big( \mu(\theta) \nabla \mathbf{u}, \nabla \mathbf{v}_h \big) -  \big( \mu(\Pi^{0,E}_k \theta_h) \boldsymbol{\Pi}^{0,E}_{k-1} \nabla\mathbf{u}_I, \boldsymbol{\Pi}^{0,E}_{k-1} \nabla \mathbf{v}_h  \big) - \mu(\Pi^{0,E}_0 \theta_h) \mathbf{S}^E_{\nabla,k} \big(  \mathbf{u}_I, \mathbf{v}_h\big) \Big) \nonumber \\
	&= \sum_{E \in \Omega_h} \Big( \big( \mu(\theta) (\nabla\mathbf{u} - \boldsymbol{\Pi}^{0,E}_{k-1} \nabla \mathbf{u}), \nabla \mathbf{v}_h \big) + \big( \mu(\theta) \boldsymbol{\Pi}^{0,E}_{k-1} \nabla \mathbf{u}, \nabla \mathbf{v}_h - \boldsymbol{\Pi}^{0,E}_{k-1} \nabla \mathbf{v}_h  \big)\, + \nonumber \\ & \qquad \big( (\mu(\theta)- \mu(\Pi^{0,E}_k \theta_h)) \boldsymbol{\Pi}^{0,E}_{k-1} \nabla \mathbf{u}, \boldsymbol{\Pi}^{0,E}_{k-1} \nabla \mathbf{v}_h \big) + \big( \mu(\Pi^{0,E}_k \theta_h) \boldsymbol{\Pi}^{0,E}_{k-1} (\nabla \mathbf{u}-\nabla \mathbf{u}_I), \boldsymbol{\Pi}^{0,E}_{k-1} \nabla \mathbf{v}_h \big)  \nonumber \\ & \qquad  -\mu(\Pi^{0,E}_0 \theta_h) \mathbf{S}^E_{\nabla,k} \big( \mathbf{u}_I,  \mathbf{v}_h\big) \Big)  \nonumber \\
	&= \sum_{E \in \Omega_h} \Big( \big( \mu(\theta) (\nabla \mathbf{u} - \boldsymbol{\Pi}^{0,E}_{k-1} \nabla \mathbf{u}), \nabla \mathbf{v}_h \big) + \big( \mu(\theta) \boldsymbol{\Pi}^{0,E}_{k-1} \nabla \mathbf{u} - \boldsymbol{\Pi}^{0,E}_{k-1}(\mu(\theta) \nabla \mathbf{u}), \nabla \mathbf{v}_h  - \boldsymbol{\Pi}^{0,E}_{k-1} \nabla \mathbf{v}_h  \big) \nonumber \\ & \qquad + \big( (\mu(\theta)- \mu(\Pi^{0,E}_k \theta_h)) \boldsymbol{\Pi}^{0,E}_{k-1} \nabla \mathbf{u}, \boldsymbol{\Pi}^{0,E}_{k-1} \nabla \mathbf{v}_h \big) + \big( \mu(\Pi^{0,E}_k \theta_h) \boldsymbol{\Pi}^{0,E}_{k-1} (\nabla \mathbf{u}-\nabla \mathbf{u}_I), \boldsymbol{\Pi}^{0,E}_{k-1} \nabla \mathbf{v}_h \big)  \nonumber \\ & \qquad  -\mu(\Pi^{0,E}_0 \theta_h) \mathbf{S}^E_{\nabla,k} \big( \mathbf{u}_I,  \mathbf{v}_h\big) \Big)  \nonumber \\
	& =: I_1 + I_2 + I_3 + I_4 + I_5. \label{err6}
\end{align}
$\bullet$	Employing the data assumption \textbf{(A0)}, the triangle inequality, \eqref{err0} and Lemma \ref{lemmaproj1}, it holds that
\begin{align}
	|I_1 + I_2  + I_4 + I_5| &\leq Ch^k \big( \mu^\ast + \max_{E\in\Omega_h} \|\mu \|_{k,\infty,E}  \big)\| \mathbf{u} \|_{k+1,\Omega} \|\nabla \mathbf{v}_h\|_{0,\Omega}. \label{err7}
\end{align} 
$\bullet$	Concerning $I_3$, we apply the bound \eqref{lqstab}, the Lipschitz continuity of $\mu$ and Lemma \ref{lemmaproj1}, we obtain
\begin{align}
	|I_3| &\leq \sum_{E \in \Omega_h } \|\mu(\theta)- \mu(\Pi^{0,E}_k \theta_h)\|_{0,6,E} \| \nabla \mathbf{u}\|_{0,3,E} \| \nabla \mathbf{v}_h \|_{0,E} \nonumber \\ 
	&\leq \sum_{E \in \Omega_h } L_\mu  \Big( \| e^\theta\|_{0,6,E} + \|\theta- \Pi^{0,E}_k \theta\|_{0,6,E}  \Big) \| \nabla \mathbf{u}\|_{0,3,E} \| \nabla \mathbf{v}_h \|_{0,E} \nonumber \\
	& \leq  L_\mu  \Big( C h^{k} |\theta|_{k,6,\Omega} + \| e^\theta\|_{0,6,\Omega} \Big) \| \nabla \mathbf{u}\|_{0,3,\Omega} \| \nabla \mathbf{v}_h \|_{0,\Omega} \nonumber  \\
	& \leq  L_\mu  \Big( C h^{k}\|\theta\|_{k+1,\Omega} + C_q C_{1\hookrightarrow 6}\| \nabla e^\theta\|_{0,\Omega} \Big) \| \nabla \mathbf{u}\|_{0,3,\Omega} \| \nabla \mathbf{v}_h \|_{0,\Omega}. \label{err8}
\end{align}
Thus, adding \eqref{err7} and \eqref{err8}, it gives
\begin{align}
	|A_{I,1}| &\leq Ch^k \Big( \big(\mu^\ast + \max_{E\in\Omega_h} \|\mu \|_{k,\infty,E}  \big)\| \mathbf{u} \|_{k+1,\Omega} + L_\mu \| \nabla \mathbf{u}\|_{0,3,\Omega}  \|\theta\|_{k+1,\Omega} \Big) \|\nabla \mathbf{v}_h\|_{0,\Omega} \nonumber \\ & \qquad + L_\mu C_q C_{1\hookrightarrow 6} \| \nabla \mathbf{u}\|_{0,3,\Omega} \| \nabla e^\theta\|_{0,\Omega} \|\nabla \mathbf{v}_h\|_{0,\Omega}. \label{err9}
\end{align}
\textbf{Step 2.} Adding and subtracting suitable terms, and using the stability of the projectors and Lemma \ref{lemmaproj1}, it yields that
\begin{align}
	A_{I,2} &= \sum_{E \in \Omega_h} \Big( \big( {\Pi}^{0,E}_{k-1}\nabla \cdot \mathbf{v}_h, \Pi^{0,E}_k p_I\big) - \big( \nabla \cdot \mathbf{v}_h, p \big) \Big)  \nonumber \\
	& \leq \sum_{E \in \Omega_h} \Big( \big( \nabla \cdot \mathbf{v}_h, \Pi^{0,E}_{k-1} (p_I- p)\big) + \big( \nabla \cdot \mathbf{v}_h, \Pi^{0,E}_{k-1} p -p\big) \Big)  \nonumber \\
	&\leq \sum_{E \in \Omega_h}  \Big( \| \nabla \mathbf{v}_h\|_{0,E} \| p-p_I\|_{0,E} + \| \nabla \mathbf{v}_h\|_{0,E} \|p-\Pi^{0,E}_{k-1} p\|_{0,E}  \Big) \nonumber \\
	& \leq Ch^k_E \|p\|_{k,\Omega} \| \nabla \mathbf{v}_h\|_{0,\Omega}. \label{err10}
\end{align}
\textbf{Step 3.} Using the estimation of $A_{I,2}$, we obtain
\begin{align}
	A_{I,3} \leq C h^k \| \mathbf{u}\|_{k+1,\Omega} \| q_h\|_{0,\Omega}. \label{err11}
\end{align}
\textbf{Step 4.} Concerning $A_{I,4}$, we proceed as follows
\begin{align}
	A_{I,4} &= \mathcal{L}_h[(\mathbf{u}_I, p_I), (\mathbf{v}_h, q_h)]=  \mathcal{L}_{1,h}(\mathbf{u}_I, \mathbf{v}_h) + \mathcal{L}_{2,h}(p_I, q_h) \nonumber \\
	&=: \mathcal{L}_1 +\mathcal{L}_2. \label{err12}
\end{align} 
We use the triangle inequality, $\tau_{1,E} \sim h_E$, the bound \eqref{err0} and Lemmas \ref{lemmaproj1} and \ref{lemmaproj2}:
\begin{align}
	\mathcal{L}_1 &= \sum_{E \in \Omega_h} \tau_{1,E} \Big( \big(\widehat{\mathbf{r}}_h(\nabla \mathbf{u}_I),  \widehat{\mathbf{r}}_h(\nabla \mathbf{v}_h)\big) + \mathbf{S}^E_{\nabla,k} \big( \mathbf{u}_I,  \mathbf{v}_h\big) \Big) \nonumber \\
	&\leq {C}  \sum_{E \in \Omega_h} h^{1/2}_E \Big( \| \widehat{\mathbf{r}}_h(\nabla e^\mathbf{u}_I)\|_{0,E} + \|\widehat{\mathbf{r}}_h(\nabla \mathbf{u})\|_{0,E} + \lambda_1^\ast \|\nabla(\mathbf{I}-  \boldsymbol{\Pi}^{\nabla,E}_{k}) \mathbf{u}_I\|_{0,E} \Big) \vertiii{(\mathbf{v}_h, q_h)}_{E}  \nonumber \\
	&\leq {C}\sum_{E \in \Omega_h} h^{1/2}_E \Big( \|\nabla e^\mathbf{u}_I\|_{0,E} + \|(\mathbf{I} -  \boldsymbol{\Pi}^{0,E}_k)\nabla \mathbf{u}\|_{0,E} + \|(\mathbf{I} -  \boldsymbol{\Pi}^{0,E}_{k-1}) \nabla \mathbf{u}\|_{0,E}\,+ \nonumber \\ & \qquad  \lambda_1^\ast h^k_E \| \mathbf{u}\|_{k+1,E} \Big) \vertiii{(\mathbf{v}_h, q_h)}_E\nonumber \\
	&\leq Ch^{k+1/2}  \| \mathbf{u}\|_{k+1,\Omega} \vertiii{(\mathbf{v}_h, q_h)}. \label{err13}
\end{align}
Using the estimation of $\mathcal{L}_1$ and the fact $\tau_{2,E} \sim h^2_E$, we obtain
\begin{align}
	\mathcal{L}_2 &\leq C h^k\| p\|_{k,\Omega} \vertiii{(\mathbf{v}_h, q_h)}. \label{err15}
\end{align}
Finally, combining \eqref{err9}--\eqref{err11}, \eqref{err13} and \eqref{err15}, we easily obtain \eqref{err16}. 
\end{proof}

\begin{lemma} \label{convect}
Let $\mathbf{u} \in [H^{k+1}(\Omega_h)]^2 \cap \mathbf{V}$ and $\mathbf{u}_h \in \mathbf{V}_h$. Then for any $\mathbf{v}_h \in \mathbf{V}_h$, there holds that
\begin{align}
	c^S_{V}(\mathbf{u}; \mathbf{u}, \mathbf{v}_h) - c^S_{V,h}(\mathbf{u}_h; \mathbf{u}_I, \mathbf{v}_h) &\leq  C h^k \Big( \|\mathbf{u}\|_{k,\Omega}  + \| \nabla \mathbf{u}\|_{0,\Omega}\Big) \|\mathbf{u}\|_{k+1,\Omega} \|\nabla \mathbf{v}_h\|_{0,\Omega}\, + \nonumber \\ & \qquad \widetilde{N_0} \|\nabla \mathbf{u}\|_{0,\Omega} \|\nabla e^\mathbf{u}\|_{0,\Omega} \| \nabla \mathbf{v}_h\|_{0,\Omega}. \label{err21}
\end{align}
\end{lemma}

\begin{proof}
Concerning \eqref{err21}, we add and subtract the appropriate terms as outlined below:
\begin{align}
	C_{I}:&= c^S_{V}(\mathbf{u}; \mathbf{u}, \mathbf{v}_h) - c^S_{V,h}(\mathbf{u}_h; \mathbf{u}_I, \mathbf{v}_h)  \nonumber\\
	&=c^S_{V}(\mathbf{u}; \mathbf{u}, \mathbf{v}_h) - c^S_{V,h}(\mathbf{u}; \mathbf{u}_I, \mathbf{v}_h) +  c^S_{V,h}(e^\mathbf{u};\mathbf{u}_I, \mathbf{v}_h) \nonumber\\ 
	&= c^S_{V}(\mathbf{u}; \mathbf{u}, \mathbf{v}_h) - c^S_{V,h}(\mathbf{u}; \mathbf{u}, \mathbf{v}_h) +  c^S_{V,h}(\mathbf{u}; e^\mathbf{u}_I, \mathbf{v}_h)  + c^S_{V,h}( e^\mathbf{u}; \mathbf{u}_I, \mathbf{v}_h)  \nonumber\\
	&=: C_{I,1} + C_{I,2} + C_{I,3}. \label{err17}
\end{align} 
We apply \cite[Lemma 4.3]{mvem9} to approximate $C_{I,1}$:
\begin{align}
	C_{I,1} &\leq Ch^k \big(  \|\mathbf{u}\|_{k,\Omega} +   \| \nabla \mathbf{u}\|_{0,\Omega}  \big) \|\mathbf{u}\|_{k+1,\Omega}  \| \nabla \mathbf{v}_h\|_{0,\Omega}. \label{err18}
\end{align}
Employing the stability of the projectors, the bound \eqref{dcnt-st} and Lemma \ref{lemmaproj2}, we infer
\begin{align}
	C_{I,2} & \leq \widehat{N_0} \|\nabla \mathbf{u}\|_{0,\Omega} \|\nabla e^\mathbf{u}_I\|_{0,\Omega} \| \nabla \mathbf{v}_h\|_{0,\Omega}  \nonumber \\
	&\leq  C h^{k} \|\nabla \mathbf{u}\|_{0,\Omega} \| \mathbf{u}\|_{k+1,\Omega} \| \nabla \mathbf{v}_h\|_{0,\Omega}. \label{err19}
\end{align}
Following $C_{I,2}$, we can arrive at
\begin{align}
	C_{I,3} &\leq \widehat{N_0} \|\nabla \mathbf{u}_I\|_{0,\Omega} \|\nabla e^\mathbf{u}\|_{0,\Omega} \| \nabla \mathbf{v}_h\|_{0,\Omega} \nonumber \\
	&\leq \widetilde{N_0} \|\nabla \mathbf{u}\|_{0,\Omega} \|\nabla e^\mathbf{u}\|_{0,\Omega} \| \nabla \mathbf{v}_h\|_{0,\Omega}. \label{err20}
\end{align}
Thus, the estimate \eqref{err21} can be obtained by adding \eqref{err18}--\eqref{err20}. 
\end{proof}

\begin{lemma}\label{temp}
Under the assumption \textbf{(A2)}, let $\theta$ and $\theta_h$ be the solution of \eqref{dvar-1} and \eqref{dnvem1} with $\mathbf{u} \in \mathbf{V}$ and $\mathbf{u}_h \in \mathbf{V}_h$, respectively. Further, we assume that the stabilization parameter is $\tau_{E} \sim h_E$, for all $E \in \Omega_h$. Then for any $\psi \in \Sigma_{0,h}$, the following holds
\begin{align}
	\Scale[0.9]{A_{T}(\theta; \theta, \psi_h) - A_{T,h}(\theta_h; \theta_I, \psi_h)} & \Scale[0.9]{\leq Ch^k \Big( \big(1+\kappa^\ast + \max_{E\in\Omega_h} \|\kappa \|_{k,\infty,E}  \big)\| \theta \|_{k+1,\Omega} + L_\kappa \| \nabla \theta\|_{0,3,\Omega}  \|\theta\|_{k+1,\Omega}} \nonumber \\& \qquad \Scale[0.9]{ +\, \big( \|\mathbf{u}\|_{k,\Omega}   + \| \nabla \mathbf{u}\|_{0,\Omega}\big) \|\theta\|_{k+1,\Omega}  + \| \nabla \theta\|_{0,\Omega} \| \mathbf{u} \|_{k+1,\Omega}  \Big) \|\nabla \psi_h \|_{0,\Omega}}\,+ \nonumber \nonumber \\ & \qquad \Scale[0.9]{+\, \big(L_\kappa C_q C_{1\hookrightarrow 6} \| \nabla \theta\|_{0,3,\Omega} \| \nabla e^\theta\|_{0,\Omega} + \widetilde{N_0} \| \nabla \theta\|_{0,\Omega}  \|\nabla e^\mathbf{u} \|_{0,\Omega}\big) \|\nabla \psi_h \|_{0,\Omega}}. \label{errtemp}
\end{align}
\end{lemma}

\begin{proof} We begin with definition of $A_T$ and $A_{T,h}$
\begin{align}
	T_A :&= A_{T}(\theta; \theta, \psi_h) - A_{T,h}(\theta_h; \theta_I, \psi_h) \nonumber \\
	& = a_T(\theta; \theta,\psi_h) - a_{T,h}(\theta_h; \theta_I, \psi_h) + c^S_T(\mathbf{u}; \theta, \psi_h) - c^S_{T,h}(\mathbf{u}_h; \theta_I, \psi_h) - \mathcal{L}_{T,h}(\theta_I,\psi_h) \nonumber \\
	&=: T_{A,1} + T_{A,2} +T_{A,3}. \label{err23}
\end{align}
We will estimate the above terms in the following steps: \newline
\noindent \textbf{Step 1.} Adding and subtracting suitable terms, we have 
\begin{align}
	T_{A,1} &= \sum_{E \in \Omega_h} \Big( \big( \kappa(\theta) \nabla \theta, \nabla \psi_h \big) -  \big( \kappa(\Pi^{0,E}_k \theta_h) \boldsymbol{\Pi}^{0,E}_{k-1} \nabla\theta_I, \boldsymbol{\Pi}^{0,E}_{k-1} \nabla \psi_h  \big) - \kappa(\Pi^{0,E}_0 \theta_h) S^E_{\nabla,k} \big(  \theta_I, \psi_h\big) \Big) \nonumber \\
	&=  \sum_{E \in \Omega_h} \Big( \big( \kappa(\theta) (\nabla\theta - \boldsymbol{\Pi}^{0,E}_{k-1} \nabla \theta), \nabla \psi_h \big) + \big( \kappa(\theta) \boldsymbol{\Pi}^{0,E}_{k-1} \nabla \theta, \nabla \psi_h -   \boldsymbol{\Pi}^{0,E}_{k-1} \nabla \psi_h  \big) + \nonumber \\ & \qquad  \big( (\kappa(\theta)- \kappa(\Pi^{0,E}_k \theta_h)) \boldsymbol{\Pi}^{0,E}_{k-1} \nabla \theta, \boldsymbol{\Pi}^{0,E}_{k-1} \nabla \psi_h \big) + \big( \kappa(\Pi^{0,E}_k \theta_h) \boldsymbol{\Pi}^{0,E}_{k-1} (\nabla \theta-\nabla \theta_I), \boldsymbol{\Pi}^{0,E}_{k-1} \nabla \psi_h \big)  \nonumber \\ & \qquad  -\kappa(\Pi^{0,E}_0 \theta_h) S^E_{\nabla,k} \big( \theta_I,  \psi_h\big) \Big)  \nonumber \\
	&= \sum_{E \in \Omega_h} \Big( \big( \kappa(\theta) (\nabla \theta - \boldsymbol{\Pi}^{0,E}_{k-1} \nabla \theta), \nabla \psi_h \big) + \big( \kappa(\theta) \boldsymbol{\Pi}^{0,E}_{k-1} \nabla \theta - \boldsymbol{\Pi}^{0,E}_{k-1}(\kappa(\theta) \nabla \theta), \nabla \psi_h  - \boldsymbol{\Pi}^{0,E}_{k-1} \nabla \psi_h  \big) \nonumber \\ & \qquad \big( (\kappa(\theta)- \kappa(\Pi^{0,E}_k \theta_h)) \boldsymbol{\Pi}^{0,E}_{k-1} \nabla \theta, \boldsymbol{\Pi}^{0,E}_{k-1} \nabla \psi_h \big) + \big( \kappa(\Pi^{0,E}_k \theta_h) \boldsymbol{\Pi}^{0,E}_{k-1} (\nabla \theta-\nabla \theta_I), \boldsymbol{\Pi}^{0,E}_{k-1} \nabla \psi_h \big)  \nonumber \\ & \qquad  -\kappa(\Pi^{0,E}_0 \theta_h) S^E_{\nabla,k} \big( \theta_I,  \psi_h\big) \Big)  \nonumber \\
	& =: T_{A,11} + T_{A,12} + T_{A,13} + T_{A,14} + T_{A,15}. \label{err24}
\end{align}
$\bullet$ Under the data assumption \textbf{(A0)}, we use the triangle inequality, \eqref{err0} and Lemmas \ref{lemmaproj1} and \ref{lemmaproj2}:
\begin{align}
	|T_{A,11} + T_{A,12}  + T_{A,14} + T_{A,15}| &\leq Ch^k \big( \kappa^\ast + \max_{E\in\Omega_h} \|\kappa \|_{k,\infty,E}  \big)\| \theta \|_{k+1,\Omega} \|\nabla \psi_h\|_{0,\Omega}. \label{err25}
\end{align} 
$\bullet$ To estimate $T_{A,13}$, we employ \eqref{lqstab} and the Sobolev embedding theorem, we obtain
\begin{align}
	|T_{A,13}| &\leq \sum_{E \in \Omega_h } \|\kappa(\theta)- \kappa(\Pi^{0,E}_k \theta_h)\|_{0,6,E} \| \nabla \theta\|_{0,3,E} \| \nabla \psi_h \|_{0,E} \nonumber \\ 
	&\leq \sum_{E \in \Omega_h } L_\kappa  \Big( \| e^\theta\|_{0,6,E} + \|\theta- \Pi^{0,E}_k \theta\|_{0,6,E}  \Big) \| \nabla \theta\|_{0,3,E} \| \nabla \psi_h \|_{0,E} \nonumber \\
	& \leq  L_\kappa  \Big( C h^{k} |\theta|_{k,6,\Omega} + \| e^\theta\|_{0,6,\Omega} \Big) \| \nabla \theta\|_{0,3,\Omega} \| \nabla \psi_h \|_{0,\Omega} \nonumber  \\
	& \leq  L_\kappa  \Big( C h^{k}\|\theta\|_{k+1,\Omega} + C_q C_{1\hookrightarrow 6}\| \nabla e^\theta\|_{0,\Omega} \Big) \| \nabla \theta\|_{0,3,\Omega} \| \nabla \psi_h \|_{0,\Omega}. \label{err26}
\end{align}
We now add \eqref{err25} and \eqref{err26}:
\begin{align}
	|T_{A,1}| &\leq Ch^k \Big( \big(\kappa^\ast + \max_{E\in\Omega_h} \|\kappa \|_{k,\infty,E}  \big)\| \theta \|_{k+1,\Omega} + L_\kappa \| \nabla \theta\|_{0,3,\Omega}  \|\theta\|_{k+1,\Omega} \Big) \|\nabla \psi_h\|_{0,\Omega} \nonumber \\ & \qquad + L_\kappa C_q C_{1\hookrightarrow 6} \| \nabla \theta\|_{0,3,\Omega} \| \nabla e^\theta\|_{0,\Omega} \|\nabla \psi_h\|_{0,\Omega}. \label{err27}
\end{align}
\noindent \textbf{Step 2.} To estimate $T_{A,2}$, we have the following 
\begin{align}
	T_{A,2}&= c^S_{T}(\mathbf{u}; \theta, \psi_h) - c^S_{T,h}(\mathbf{u}_h; \theta_I, \psi_h)  \nonumber\\
	&=c^S_{T}(\mathbf{u}; \theta, \psi_h) - c^S_{T,h}(\mathbf{u}; \theta, \psi_h)+ c^S_{T,h}(\mathbf{u}; e^\theta_I, \psi_h) + c^S_{T,h}(e^\mathbf{u}; \theta_I, \psi_h)  \nonumber\\
	&=: T_{c,1} + T_{c,2} + T_{c,3}. \label{err28}
\end{align} 
$\bullet$	Applying \cite[Lemma 4.3]{mvem9}, we arrive at
\begin{align}
	T_{c,1} &\leq C h^k \Big(  \|\mathbf{u}\|_{k,\Omega}  \|\theta\|_{k+1,\Omega} +  \| \nabla \mathbf{u}\|_{0,\Omega}  \| \theta \|_{k+1,\Omega} + \| \mathbf{u} \|_{k+1,\Omega}  \| \nabla \theta\|_{0,\Omega} \Big) \| \nabla \psi_h\|_{0,\Omega}. \label{err29}
\end{align}
$\bullet$	Using \eqref{dcnt-st} and Lemma \ref{lemmaproj2}, we get
\begin{align}
	T_{c,2} & \leq \widehat{N_0} \|\nabla \mathbf{u}\|_{0,\Omega} \|\nabla e^\theta_I\|_{0,\Omega} \|\nabla \psi_h\|_{0,\Omega} \nonumber \\
	&\leq Ch^k \|\nabla \mathbf{u}\|_{0,\Omega} \| \theta\|_{k+1,\Omega} \| \nabla \psi_h\|_{0,\Omega}. \label{err30}
\end{align}
$\bullet$	Applying \eqref{dcnt-st}, it gives
\begin{align}
	T_{c,3} & \leq  \widetilde{N_0} \|\nabla e^\mathbf{u}\|_{0,\Omega} \| \nabla \theta\|_{0,\Omega} \| \nabla \psi_h\|_{0,\Omega}. \label{err31}
\end{align}
Adding the estimates \eqref{err29}--\eqref{err31}, it gives
\begin{align}
	T_{A,2} &\leq C h^k \Big( \big( \|\mathbf{u}\|_{k,\Omega}   + \| \nabla \mathbf{u}\|_{0,\Omega}\big) \|\theta\|_{k+1,\Omega}  + \| \nabla \theta\|_{0,\Omega} \| \mathbf{u} \|_{k+1,\Omega}  \Big) \|\nabla \psi_h \|_{0,\Omega}  \nonumber \\ & \qquad +  \widetilde{N_0} \| \nabla \theta\|_{0,\Omega}  \|\nabla e^\mathbf{u} \|_{0,\Omega} \|\nabla \psi_h \|_{0,\Omega}. \label{err32}
\end{align}
\noindent \textbf{Step 3.} Following the estimation of \cite[Lemma 5.10]{vem28m} with $\tau_E \sim h_E$, it yields that
\begin{align}
	A_{T,3} \leq C h^{k+1/2}  \|\theta\|_{k+1,\Omega} \| \nabla \psi_h\|_{0,\Omega}. \label{err33} 
\end{align}
Finally, adding the estimates \eqref{err27}, \eqref{err32}, and \eqref{err33}, we readily arrive at the result \eqref{errtemp}. 

\end{proof}

Hereafter, we present the convergence estimate in the energy norm:
\begin{theorem} \label{convergece}
Under the assumptions \textbf{(A0)}, \textbf{(A1)}, and the assumptions of Theorem \ref{cunique} and Theorem \ref{WELLP}, let $(\mathbf{u}, p, \phi) \in \mathbf{V} \times Q \times \Sigma$ and $(\mathbf{u}_h, p_h, \phi_h) \in \mathbf{V}_h \times Q_h \times \Sigma_h$ be the solution of the problems \eqref{variation-1} and \eqref{nvem}. Further we assume that $\mu^\ast$, $\mu_\ast$ and $\kappa_\ast$ in \textbf{(A0)} are such that
\begin{align}	
	1- \frac{\widetilde{N_0} \mu_\ast}{N_0} \big( C^{-1}_{o,V} + C^{-1}_{o,T} \big)  - \kappa_\ast \big( C^{-1}_{o,T} + C^{-1}_{o,V} \big) >0. \label{conv-cod}
\end{align}
Then under the assumption \textbf{(A2)}, we have the following error estimate
\begin{align}
	\vertiii{( \mathbf{u} - \mathbf{u}_h, p-p_h)} + \vertiii{\theta - \theta_h} &\leq C h^k \Big( \big( 1+\|\mathbf{g}\|_{0,\Omega}  + \|\mathbf{g}\|_{k,\Omega} \big)  \|\theta \|_{k+1,\Omega} + \|\mathbf{f}\|_{k,\Omega} + \|\mathcal{Q}\|_{k,\Omega} \, + \nonumber\\ & \qquad \qquad   \|\mathbf{u}\|_{k,\Omega}   \|\theta\|_{k+1,\Omega} + \|\mathbf{u}\|_{k,\Omega} \|\mathbf{u}\|_{k+1,\Omega} + \| p\|_{k,\Omega} \Big), \label{converge0}
\end{align}
where the positive constant $C$ does not rely on $h$. 
\end{theorem}

\begin{proof}
Under the assumption of Theorem \ref{cunique} and using \eqref{unique-4}, we obtain the following
\begin{align}
	\|\nabla \mathbf{u}\|_{0,\Omega} + \| \nabla \theta\|_{0,\Omega} &\leq \frac{\mu_\ast}{N_0}, \label{cerr1}\\
	\big( L_\kappa C_q C_{1\hookrightarrow 6} \|\nabla {\theta}\|_{0,3,\Omega} + \alpha C_q^2 C^2_{1 \hookrightarrow 4} \|\mathbf{g}\|_{0,\Omega} +  L_\mu C_q C_{1 \hookrightarrow 6}  \|\nabla {\mathbf{u}}\|_{0,3,\Omega} \big) &< \kappa_\ast. \label{cerr2}
\end{align}
Additionally, we  also have
\begin{align}
	\vertiii{(e^\mathbf{u}_I, e^p_I)} + \vertiii{e^\theta_I} \leq Ch^k\big(  \|\mathbf{u}\|_{k+1,\Omega} + \| p\|_{k,\Omega} + \|\theta\|_{k+1,\Omega} \big).\label{cerr-int}
\end{align}
We will derive the estimate \eqref{converge0} in the following steps: \newline
\textbf{Step 1.}
Subtracting the system \eqref{variation-01} from the system \eqref{nvem}, and adding and subtracting suitable terms, we arrive at
\begin{align}
	A_{V,h}[\theta_h; (e^\mathbf{u}_h,e^p_h), (\mathbf{v}_h, q_h)] &+ c^S_{V,h}(\mathbf{u}_h; e^\mathbf{u}_h,\mathbf{v}_h) \nonumber \\ &= F_{h,\theta_h}(\mathbf{v}_h) - F_{\theta}(\mathbf{v}_h) +  A_{V}[\theta; (\mathbf{u},p), (\mathbf{v}_h, q_h)] \, - \nonumber \\ & \quad A_{V,h}[\theta_h; (\mathbf{u}_I,p_I ), (\mathbf{v}_h, q_h)] + c^S_{V}(\mathbf{u}; \mathbf{u},\mathbf{v}_h) - c^S_{V,h}(\mathbf{u}_h; \mathbf{u}_I,\mathbf{v}_h) \nonumber \\
	&=: F_I + A_I +C_I. \label{cerr3}\\
	A_{T,h}(\theta_h; \theta_h-\theta_I, e^\theta_h) &= (\mathcal{Q}_h, e^\theta_h) - (\mathcal{Q}, e^\theta_h) + A_{T}(\theta; \theta, e^\theta_h) - A_{T,h}(\theta_h; \theta_I, e^\theta_h) \nonumber \\
	&=:T_\mathcal{Q} + T_A. \label{cerr4}
\end{align}
\textbf{Step 2.} In this step, we estimate the error equation \eqref{cerr3}. Concerning this, we proceed as follows: \newline
$\bullet$ 	Applying the bound \eqref{lqstab} and Lemma \ref{lemmaproj1}, we infer
\begin{align}
	F_I:&=F_{h,\theta_h}(\mathbf{v}_h) - F_{\theta}(\mathbf{v}_h) =\sum_{E \in \Omega_h} \big[ \big(\alpha \mathbf{g} \Pi^{0,E}_k \theta_h, \boldsymbol{\Pi}^{0,E}_k \mathbf{v}_h \big) - \big( \alpha \mathbf{g} \theta,  \mathbf{v}_h \big) +\big( \mathbf{f},  \mathbf{v}_h - \boldsymbol{\Pi}^{0,E}_k \mathbf{v}_h \big) \big]  \nonumber \\
	&= \sum_{E \in \Omega_h} \big[ \big(\alpha \mathbf{g} (\Pi^{0,E}_k \theta_h - \theta), \boldsymbol{\Pi}^{0,E}_k \mathbf{v}_h \big) + \big( \alpha \mathbf{g} \theta, \boldsymbol{\Pi}^{0,E}_k \mathbf{v}_h- \mathbf{v}_h \big) +  \big(\mathbf{f}-  \boldsymbol{\Pi}^{0,E}_k \mathbf{f},  \mathbf{v}_h - \boldsymbol{\Pi}^{0,E}_k \mathbf{v}_h \big)\big]  \nonumber \\
	&=\sum_{E \in \Omega_h}  \big[ \alpha \big( \mathbf{g} (\Pi^{0,E}_k \theta_h - \theta), \boldsymbol{\Pi}^{0,E}_k \mathbf{v}_h \big) + \alpha \big( \mathbf{g} \theta - \boldsymbol{\Pi}^{0,E}_k(\mathbf{g} \theta), \boldsymbol{\Pi}^{0,E}_k \mathbf{v}_h- \mathbf{v}_h \big)+ \big(\mathbf{f}-  \boldsymbol{\Pi}^{0,E}_k \mathbf{f},  \mathbf{v}_h - \boldsymbol{\Pi}^{0,E}_k \mathbf{v}_h \big) \big] \nonumber \\
	&\leq \sum_{E \in \Omega_h}  \big[ \alpha \|\mathbf{g}\|_{0,E}  \|\Pi^{0,E}_k \theta_h - \theta\|_{0,4,E} \| \mathbf{v}_h \|_{0,4,E} + \alpha \| \mathbf{g} \theta - \boldsymbol{\Pi}^{0,E}_k(\mathbf{g} \theta)\|_{0,E} \|\boldsymbol{\Pi}^{0,E}_k \mathbf{v}_h- \mathbf{v}_h \|_{0,E} \nonumber \\ 
	& \qquad \qquad  + C h^{k+1}_E \|\mathbf{f}\|_{k,E} \|\nabla \mathbf{v}_h\|_{0,E}  \big]	\nonumber \\
	& \leq \sum_{E \in \Omega_h}  \big[ \alpha \|\mathbf{g}\|_{0,E} \big( \| e^\theta\|_{0,4,E} +  \| \theta- \Pi^{0,E}_k \theta \|_{0,4,E} \big)\| \mathbf{v}_h \|_{0,4,E} + \alpha C h^{k}_E |\mathbf{g} \theta|_{k-1,E} \| \nabla \mathbf{v}_h \|_{0,E} \nonumber \\ 
	& \qquad \qquad  + C h^{k+1}_E \|\mathbf{f}\|_{k,E} \|\nabla \mathbf{v}_h\|_{0,E}  \big]	\nonumber \\
	&\leq Ch^k \big[ \alpha \|\mathbf{g}\|_{0,\Omega}  \| \theta \|_{k,4,\Omega}  + \alpha \|\mathbf{g}\|_{k,\Omega}  \|\theta \|_{k,\Omega} + \|\mathbf{f}\|_{k,\Omega} \big]\| \nabla \mathbf{v}_h \|_{0,\Omega}  +  \alpha \|\mathbf{g}\|_{0,\Omega} \| e^\theta\|_{0,4,\Omega}	\|  \mathbf{v}_h \|_{0,4,\Omega}. \nonumber
	\intertext{Using the Sobolev embedding theorem, it holds that}
	F_I &\leq Ch^k \big[ \alpha (\|\mathbf{g}\|_{0,\Omega}  + \|\mathbf{g}\|_{k,\Omega})  \|\theta \|_{k+1,\Omega} + \|\mathbf{f}\|_{k,\Omega} \big]\| \nabla \mathbf{v}_h \|_{0,\Omega}  +  \alpha C_q^2 C^2_{1 \hookrightarrow 4} \|\mathbf{g}\|_{0,\Omega} \|\nabla e^\theta\|_{0,\Omega}	\| \nabla \mathbf{v}_h \|_{0,\Omega}. \label{cerr5}
\end{align}
$\bullet$ Recalling Lemma \ref{velocity} and the bounds \eqref{cerr1} and \eqref{cerr2}, we obtain
\begin{align}
	A_I &\leq Ch^k \Big( \big(1+\mu^\ast + \max_{E\in\Omega_h} \|\mu \|_{k,\infty,E}  \big)\| \mathbf{u} \|_{k+1,\Omega} + \kappa_\ast C_q^{-1} C_{1\hookrightarrow 6}^{-1} \|\theta\|_{k+1,\Omega} + \| p\|_{k,\Omega} \Big) \vertiii{(\mathbf{v}_h, q_h)} \nonumber \\ & \, + L_\mu C_q C_{1\hookrightarrow 6} \| \nabla \mathbf{u}\|_{0,3,\Omega} \| \nabla e^\theta\|_{0,\Omega} \vertiii{(\mathbf{v}_h, q_h)}. \label{cerr6}
\end{align}
$\bullet$ Concerning $C_I$, we recall Lemma \ref{convect} and \eqref{cerr1}
\begin{align}
	C_I&\leq  C h^k \Big( \|\mathbf{u}\|_{k,\Omega}  + \mu_\ast N_0^{-1}\Big) \|\mathbf{u}\|_{k+1,\Omega} \|\nabla \mathbf{v}_h\|_{0,\Omega}\, +  \widetilde{N_0} \|\nabla \mathbf{u}\|_{0,\Omega} \|\nabla e^\mathbf{u}\|_{0,\Omega} \| \nabla \mathbf{v}_h\|_{0,\Omega}. \label{cerr7}
\end{align}
Now, substituting \eqref{cerr5}, \eqref{cerr6} and \eqref{cerr7} in \eqref{cerr3} and using \eqref{wellposed-1}, it gives 
\begin{align}
	C_{o,V}\vertiii{(e^\mathbf{u}_h, e^p_h)} &\leq Ch^k \Big(  (\|\mathbf{g}\|_{0,\Omega}  + \|\mathbf{g}\|_{k,\Omega})  \|\theta \|_{k+1,\Omega} + \|\mathbf{f}\|_{k,\Omega}   + \| \mathbf{u} \|_{k+1,\Omega} +  \|\theta\|_{k+1,\Omega} \nonumber \\ & \quad \|\mathbf{u}\|_{k,\Omega} \|\mathbf{u}\|_{k+1,\Omega} + \| p\|_{k,\Omega} \Big)  + L_\mu C_q C_{1\hookrightarrow 6} \| \nabla \mathbf{u}\|_{0,3,\Omega} \vertiii{ e^\theta_h}\nonumber \\ & \quad + \alpha C_q^2 C^2_{1 \hookrightarrow 4} \|\mathbf{g}\|_{0,\Omega} \vertiii{ e^\theta_h}   +  \widetilde{N_0} \|\nabla \mathbf{u}\|_{0,\Omega} \vertiii{ (e^\mathbf{u}_h, e^p_h)}. \label{velerr}
\end{align}
\textbf{Step 3.} In this step, we will provide the estimation of \eqref{cerr3}.To obtain this, we proceed as follows: \newline
$\bullet$ Using the stability of the projectors and Lemma \ref{lemmaproj1}, it holds that
\begin{align}
	T_{\mathcal{Q}} &= \sum_{E \in \Omega_h } \Big( \big(\mathcal{Q}, \Pi^{0,E}_k e^\theta_h \big) - \big( \mathcal{Q}, e^\theta_h\big) \Big) =  \sum_{E \in \Omega_h } \big(\mathcal{Q} - \Pi^{0,E}_k \mathcal{Q}, \Pi^{0,E}_k e^\theta_h - e^\theta_h \big) \nonumber \\
	& \leq  \sum_{E \in \Omega_h } \|\mathcal{Q} - \Pi^{0,E}_k  \mathcal{Q}\|_{0,E} \|e^\theta_h - \Pi^{0,E}_k e^\theta_h\|_{0,E} \nonumber \\
	& \leq Ch^{k+1} \|\mathcal{Q}\|_{k,\Omega} \| \nabla e^\theta_h\|_{0,\Omega}. \label{cerr8}
\end{align}
$\bullet$ Recalling Lemma \ref{temp}, the estimates \eqref{cerr1} and \eqref{cerr2}, we have
\begin{align}
	T_A & \leq Ch^k \Big( \big(1+\kappa^\ast + \max_{E\in\Omega_h} \|\kappa \|_{k,\infty,E}  \big)\| \theta \|_{k+1,\Omega} + \kappa_\ast C_q^{-1} C_{1\hookrightarrow 6}^{-1} \|\theta\|_{k+1,\Omega} +   \|\mathbf{u}\|_{k,\Omega}   \|\theta\|_{k+1,\Omega}  \nonumber \\& \qquad  + \mu_\ast N_0^{-1} \big(  \|\theta\|_{k+1,\Omega}  +  \| \mathbf{u} \|_{k+1,\Omega} \big)  \Big) \|\nabla e^\theta_h \|_{0,\Omega}\,+ \nonumber \nonumber \\ & \qquad +\, \big( L_\kappa C_q C_{1\hookrightarrow 6} \| \nabla \theta\|_{0,3,\Omega} \| \nabla e^\theta\|_{0,\Omega} + \widetilde{N_0} \| \nabla \theta\|_{0,\Omega}  \|\nabla e^\mathbf{u} \|_{0,\Omega}\big) \|\nabla e^\theta_h \|_{0,\Omega}. \label{cerr9}
\end{align}
Substituting \eqref{cerr8} and \eqref{cerr9} in \eqref{cerr4}, and recalling \eqref{est3}, we obtain
\begin{align}
	C_{o,T} \vertiii{e^\theta_h} &\leq C h^k \Big( \|\mathcal{Q}\|_{k,\Omega} + \|\theta\|_{k+1,\Omega} +   \|\mathbf{u}\|_{k,\Omega}   \|\theta\|_{k+1,\Omega} +  \| \mathbf{u} \|_{k+1,\Omega} \Big) \, + \nonumber \\& \qquad \big( L_\kappa C_q C_{1\hookrightarrow 6} \| \nabla \theta\|_{0,3,\Omega} \vertiii{ e^\theta_h} + \widetilde{N_0} \| \nabla \theta\|_{0,\Omega}  \vertiii{ (e^\mathbf{u}_h, e^p_h)}\big). \label{temerr}
\end{align}
Finally, combining the estimates \eqref{velerr} and \eqref{temerr}, we can conclude that
\begin{align}
	\vertiii{(e^\mathbf{u}_h, e^p_h)} + \vertiii{e^\theta_h} &\leq C h^k \Big( (\|\mathbf{g}\|_{0,\Omega}  + \|\mathbf{g}\|_{k,\Omega})  \|\theta \|_{k+1,\Omega} + \|\mathbf{f}\|_{k,\Omega} + \|\mathcal{Q}\|_{k,\Omega} + \|\theta\|_{k+1,\Omega}  +  \| \mathbf{u} \|_{k+1,\Omega} \nonumber\\ & \qquad \|\mathbf{u}\|_{k,\Omega}   \|\theta\|_{k+1,\Omega} + \|\mathbf{u}\|_{k,\Omega} \|\mathbf{u}\|_{k+1,\Omega} + \| p\|_{k,\Omega} \Big) \, + \nonumber \\ & \qquad  \widetilde{N_0} \big( C^{-1}_{o,V} \| \nabla \mathbf{u}\|_{0,\Omega} + C^{-1}_{o,T}\| \nabla \theta\|_{0,\Omega} \big) \vertiii{ (e^\mathbf{u}_h, e^p_h)} + \big[ C^{-1}_{o,T} L_\kappa C_q C_{1\hookrightarrow 6} \| \nabla \theta\|_{0,3,\Omega} \nonumber \\ & \qquad + C^{-1}_{o,V} \big( L_\mu C_q C_{1\hookrightarrow 6} \| \nabla \mathbf{u}\|_{0,3,\Omega} + \alpha C_q^2 C^2_{1 \hookrightarrow 4} \|\mathbf{g}\|_{0,\Omega} \big) \big] \vertiii{ e^\theta_h}. \nonumber
	\intertext{Using the estimates \eqref{cerr1} and \eqref{cerr2}, we obtain}
	\vertiii{(e^\mathbf{u}_h, e^p_h)} + \vertiii{e^\theta_h} &\leq C h^k \Big( (\|\mathbf{g}\|_{0,\Omega}  + \|\mathbf{g}\|_{k,\Omega})  \|\theta \|_{k+1,\Omega} + \|\mathbf{f}\|_{k,\Omega} + \|\mathcal{Q}\|_{k,\Omega} + \|\theta\|_{k+1,\Omega}  +  \| \mathbf{u} \|_{k+1,\Omega} \nonumber\\ & \qquad \|\mathbf{u}\|_{k,\Omega}   \|\theta\|_{k+1,\Omega} + \|\mathbf{u}\|_{k,\Omega} \|\mathbf{u}\|_{k+1,\Omega} + \| p\|_{k,\Omega} \Big) \, + \nonumber \\ & \qquad  \frac{\widetilde{N_0} \mu_\ast}{N_0} \big( C^{-1}_{o,V} + C^{-1}_{o,T} \big) \vertiii{ (e^\mathbf{u}_h, e^p_h)} + \kappa_\ast \big( C^{-1}_{o,T} + C^{-1}_{o,V} \big) \vertiii{ e^\theta_h}. \label{conv-fin}
\end{align}
Thus, the estimate \eqref{converge0} readily follows from \eqref{conv-cod}, \eqref{cerr-int} and \eqref{conv-fin}. 
\end{proof}

\section{Numerical results} 
\label{sec-5}
In this section, we present a series of five examples to investigate the efficiency and robustness of the proposed stabilized VEM \eqref{nvem}. Moreover, we also demonstrate its advantages in addressing practical problems.

\subsection{Fixed-point iteration scheme} \label{fixedpoint}
Let $(\mathbf{u}^0_h,p_h^0,\theta_h^0)=(\mathbf{0},0,0)$ be the initial guess. For $n \in \mathbb N$, solve the following system until convergence is achieved
\begin{align}
	\begin{cases}
		\text{find}\,\, (\mathbf{u}^{n+1}_h, p^{n+1}_h,\theta^{n+1}_h) \in \mathbf{V}_h \times Q_h \times \Sigma_{D,h},\,\, \text{such that} \\
		A_{V,h}[ \theta^n_h; (\mathbf{u}^{n+1}_h, p^{n+1}_h), (\mathbf{v}_h, q_h)]+ c^S_{h}(\mathbf{u}^n_h; \mathbf{u}^{n+1}_h, \mathbf{v}_h)= F_{h,\theta^n_h}(\mathbf{v}_h) \quad \text{for all} \, (\mathbf{v}_h, q_h) \in \mathbf{V}_h \times Q_h, \\
		A_{T,h}(\theta^n_h, \mathbf{u}^n_h; \theta^{n+1}_h, \psi_h)=(\mathcal{Q}_h, \psi_h) \quad \text{for all} \,\, \psi_h \in \Sigma_{0,h}.
	\end{cases}
	\label{fixvem}
\end{align}
\begin{proposition}
	Under the assumptions \textbf{(A0)}, \textbf{(A1)} and the assumption of Proposition \ref{WELLP}, let $(\mathbf{u}_h, p_h, \theta_h) \in \mathbf{V}_h \times Q_h \times \Sigma_{D,h}$ be the solution of \eqref{nvem}. Then, the sequence $\big\{\mathbf{u}^n_h,p^n_h,\theta^n_h \big\}$ generated by the fixed-point iteration scheme satisfies the following
	\begin{align}
		\vertiii{(\mathbf{u}^n_h,p^n_h) - (\mathbf{u}_h, p_h)} \xrightarrow{n \rightarrow \infty} 0 \qquad \text{and} \qquad \vertiii{\theta^n_h - \theta_h} \xrightarrow{n \rightarrow \infty} 0. \label{fix1}
	\end{align}
\end{proposition}

\begin{proof}
	The proof of \eqref{fix1} can be easily established by following the proof of Proposition \ref{WELLP}. 
\end{proof}

\begin{figure}[h]
		\centering
		\subfloat[$\Omega_1$]{\includegraphics[height=3cm, width=3cm]{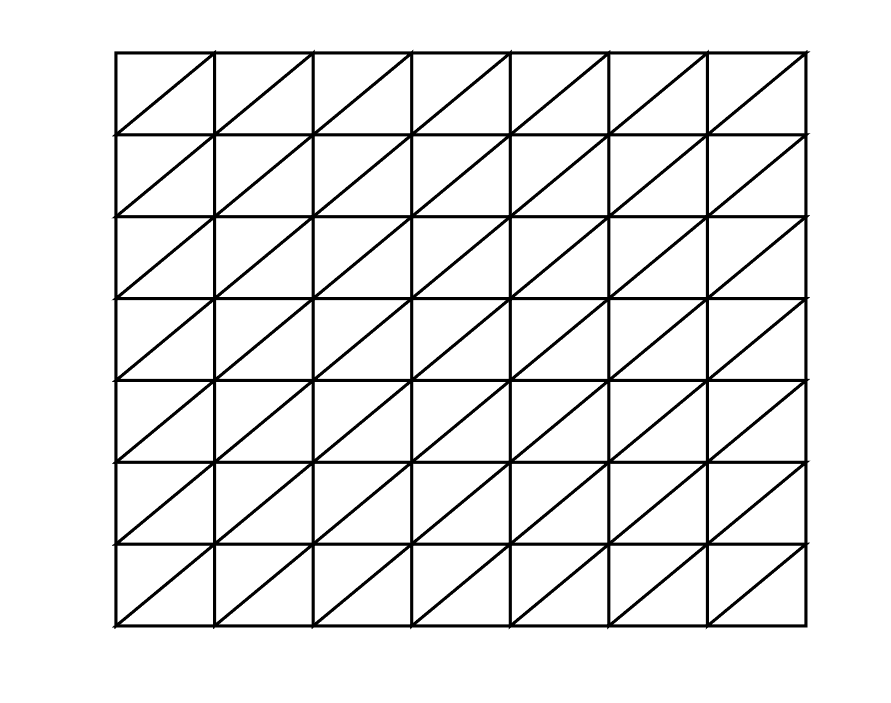}}
		\subfloat[$\Omega_2$]{\includegraphics[height=3cm, width=4cm]{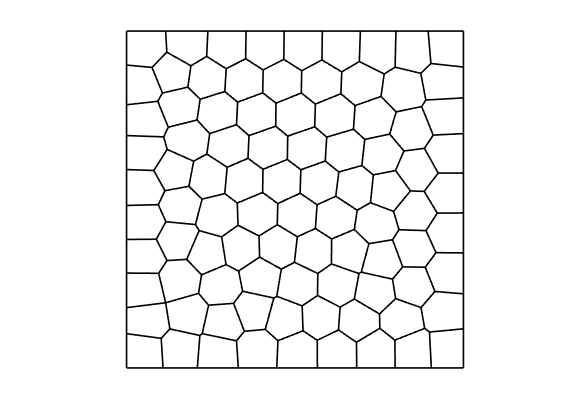}}
		\subfloat[$\Omega_3$]{\includegraphics[height=3cm, width=3cm]{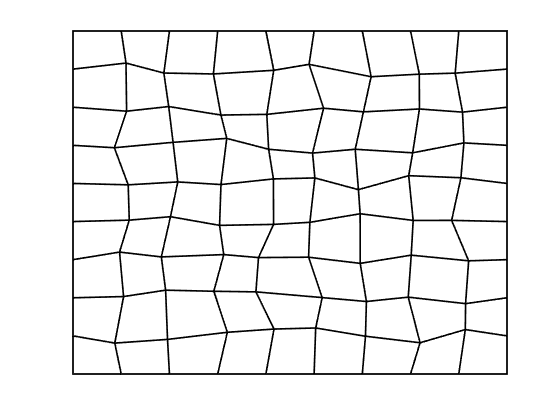}}~~
		\subfloat[$\Omega_4$]{\includegraphics[height=3cm, width=3cm]{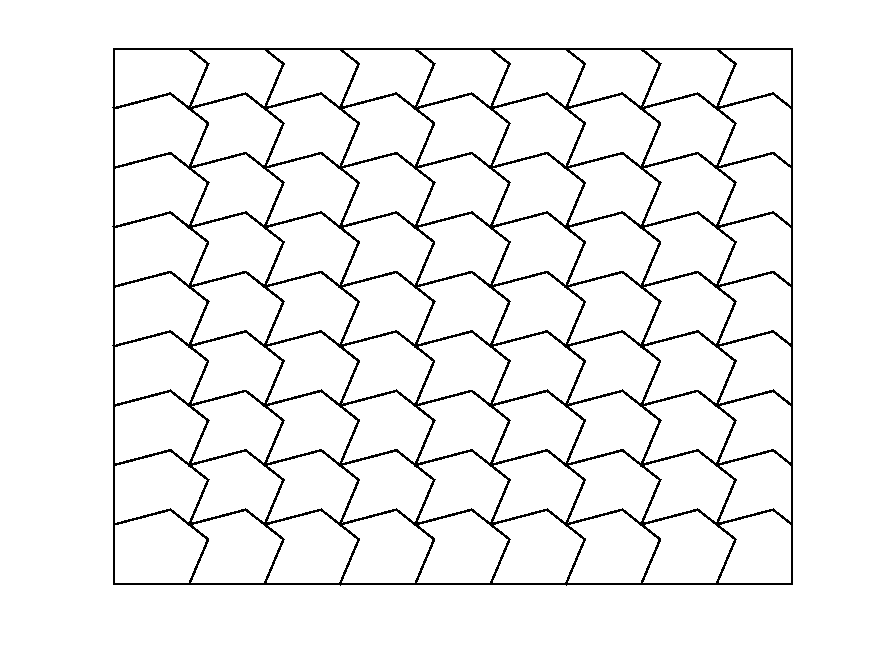}}
		\subfloat[{$\Omega_5$}]{\includegraphics[height=3cm, width=3.5cm]{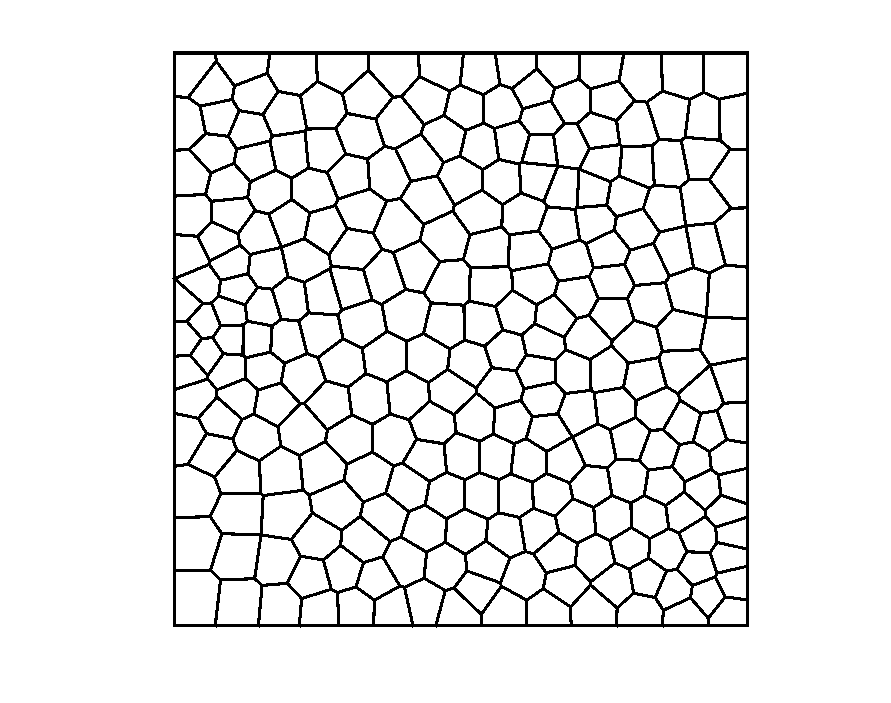}}
	\caption{Schematic representation of different mesh discretizations employed in this study}
	\label{samp} 
\end{figure}

\subsection{Meshes and computational error norms} \label{sec-error}

We present our numerical results for VEM order $k=1$ and $k=2$, employing convex and non-convex meshes. The samples of the computational meshes are shown in Figure \ref{samp}. We employ a sequence of meshes with decreasing sizes 
$h \in \{1/5, 1/10, 1/20, 1/40, 1/80\}$. The stopping criteria for the iterative scheme introduced in Section \ref{fixedpoint} is fixed at $10^{-6}$. The numerical studies of Examples 1 and 3 are performed on the unit square $\Omega=(0,1)^2$. Additionally,  the stabilization parameters are chosen as follows: $\tau_{1,E} \sim h_E$, $\tau_{2,E} \sim h_E^2$ and $\tau_{E} \sim h_E$. 

The relative errors in the $H^1(\Omega)$ semi-norm and $L^2(\Omega)$ norm are defined as follows: 
\begin{align}
	E^{\mathbf{u}}_{H^1} := \frac{\sqrt{\sum\limits_{E\in\Omega_h}  \| \nabla (\mathbf{u} - \boldsymbol{\Pi}^{\nabla,E}_k \mathbf{u}_h ) \|_E^2}}{\sqrt{\sum\limits_{E\in\Omega_h}  \| \nabla \mathbf{u} \|_E^2}}, \qquad \qquad  E^\mathbf{u}_{L^2} :=\frac{\sqrt{\sum\limits_{E\in\Omega_h} \| \mathbf{u}-\boldsymbol{\Pi}^{0,E}_k \mathbf{u}_h\|^2_E}} {\sqrt{\sum\limits_{E\in\Omega_h}  \| \mathbf{u} \|_E^2}}, \nonumber 
\end{align}
where $E^{\mathbf{u}}_{H^1}$ and $E^{\mathbf{u}}_{L^2}$ represent the relative errors in $H^1(\Omega)$ semi-norm and $L^2(\Omega)$ norm for the velocity field, respectively, with exact velocity $\mathbf{u}$ and discrete velocity $\mathbf{u}_h$. Furthermore, the relative errors for temperature and pressure fields can be defined following the above and denoted by $E^{\theta}_{H^1}$, $E^{\theta}_{L^2}$, and $E^{p}_{L^2}$.

\begin{figure}[h]
	\centering
	\subfloat[$\Omega_1$.]{\includegraphics[height=4.5cm, width=6cm]{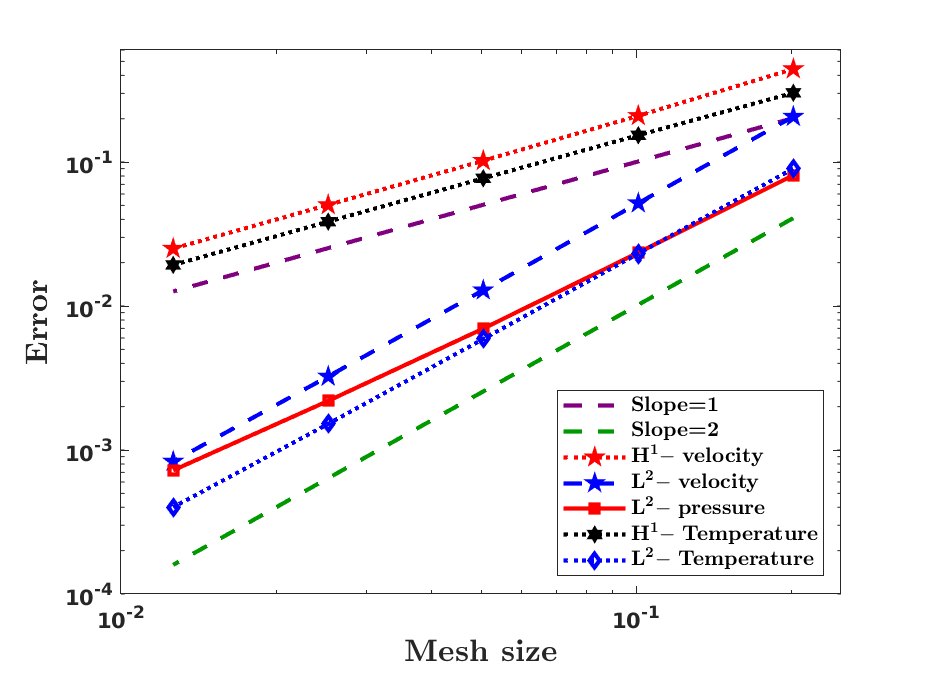}}
	\subfloat[$\Omega_1$.]{\includegraphics[height=4.5cm, width=6cm]{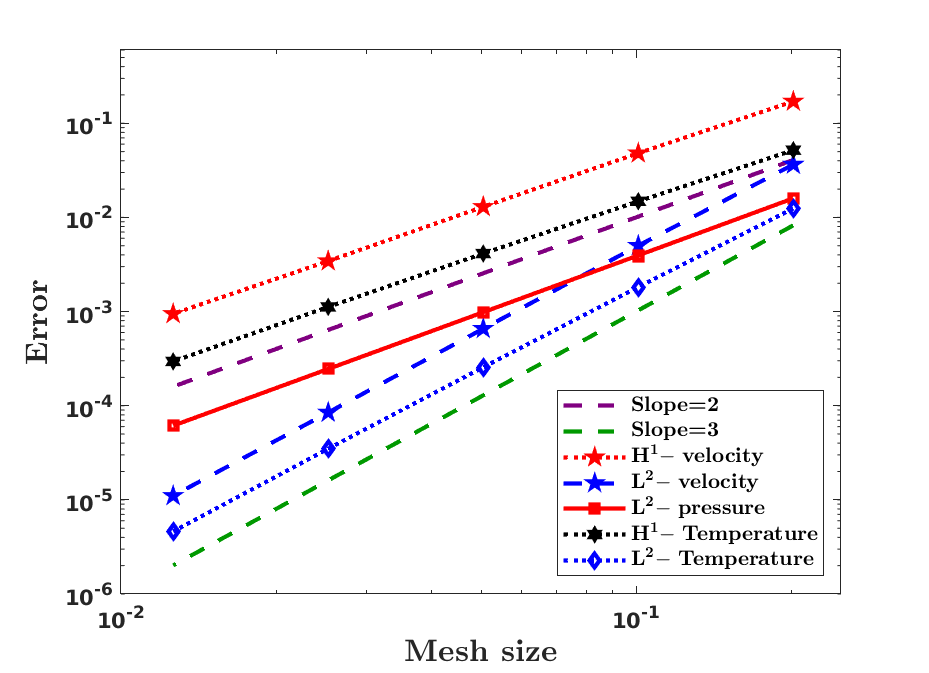}}\\
	\subfloat[$\Omega_2$.]{\includegraphics[height=4.5cm, width=6cm]{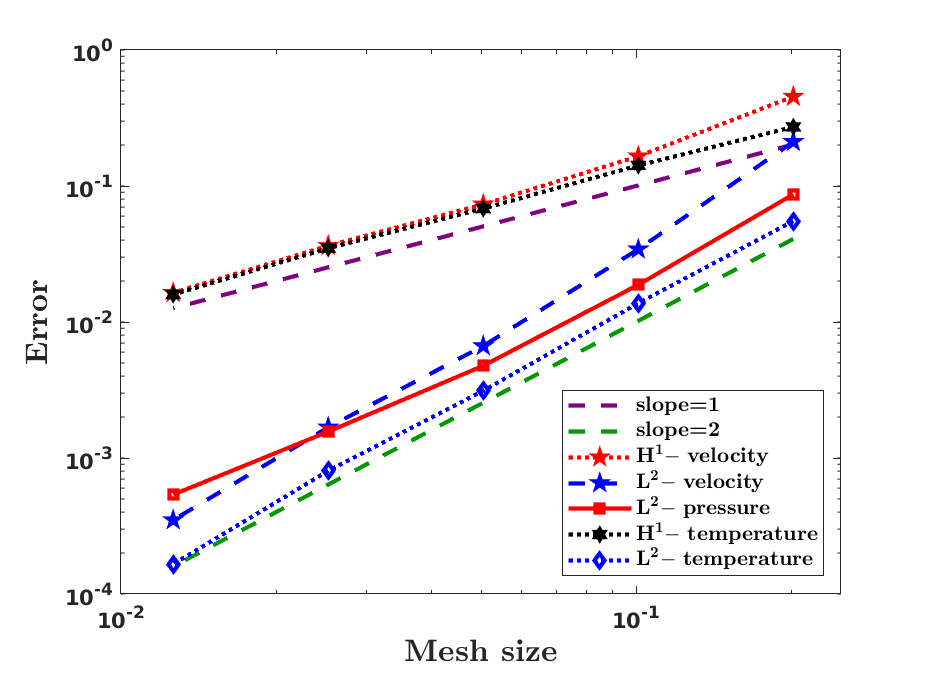}}
	\subfloat[$\Omega_2$.]{\includegraphics[height=4.5cm, width=6cm]{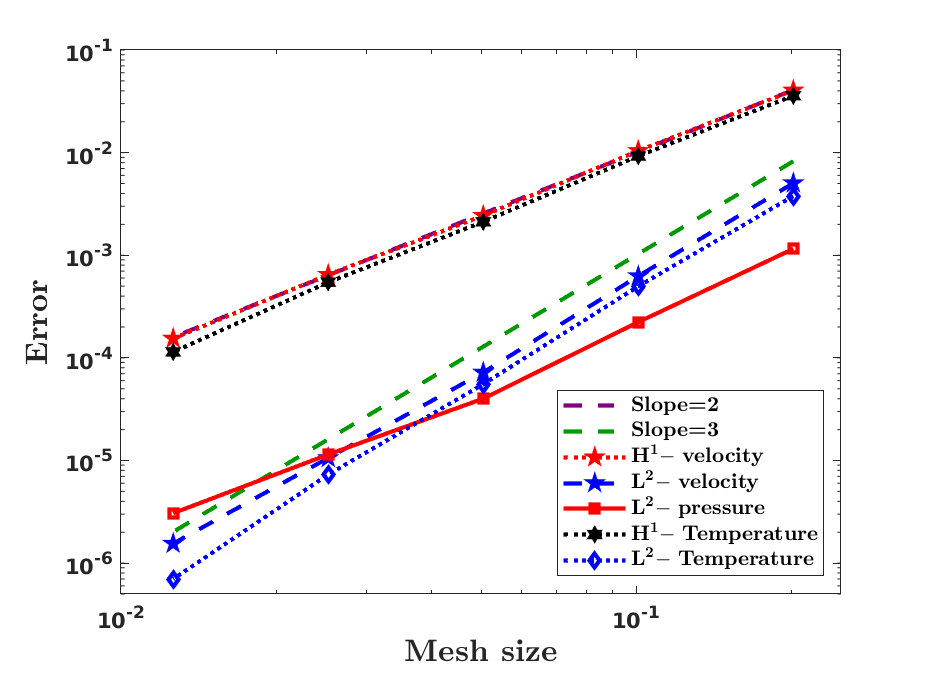}} \\
	\subfloat[$\Omega_3$.]{\includegraphics[height=4.5cm, width=6cm]{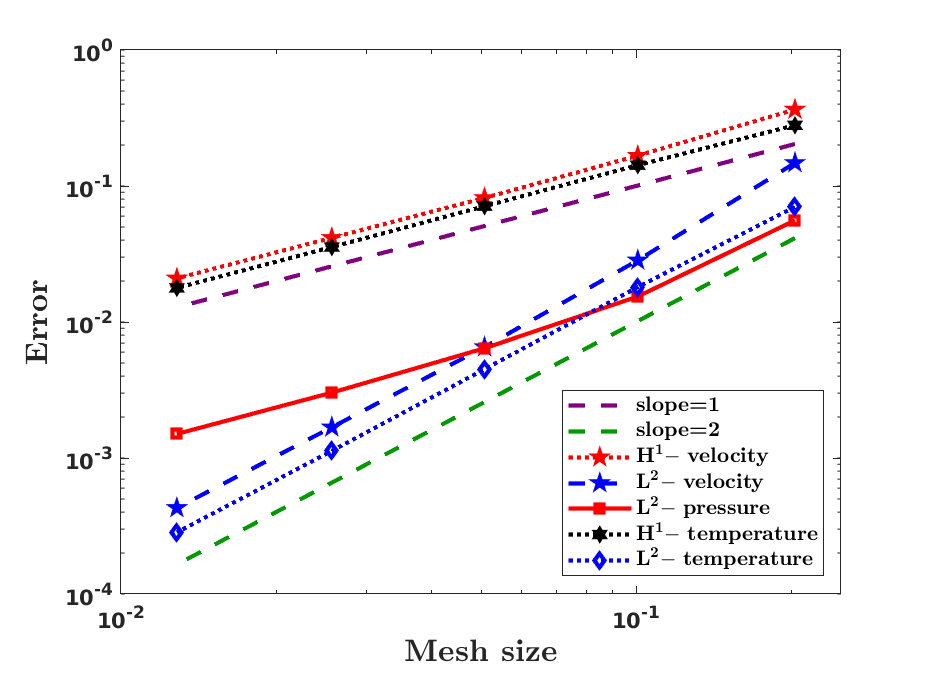}} 
	\subfloat[$\Omega_3$.]{\includegraphics[height=4.5cm, width=6cm]{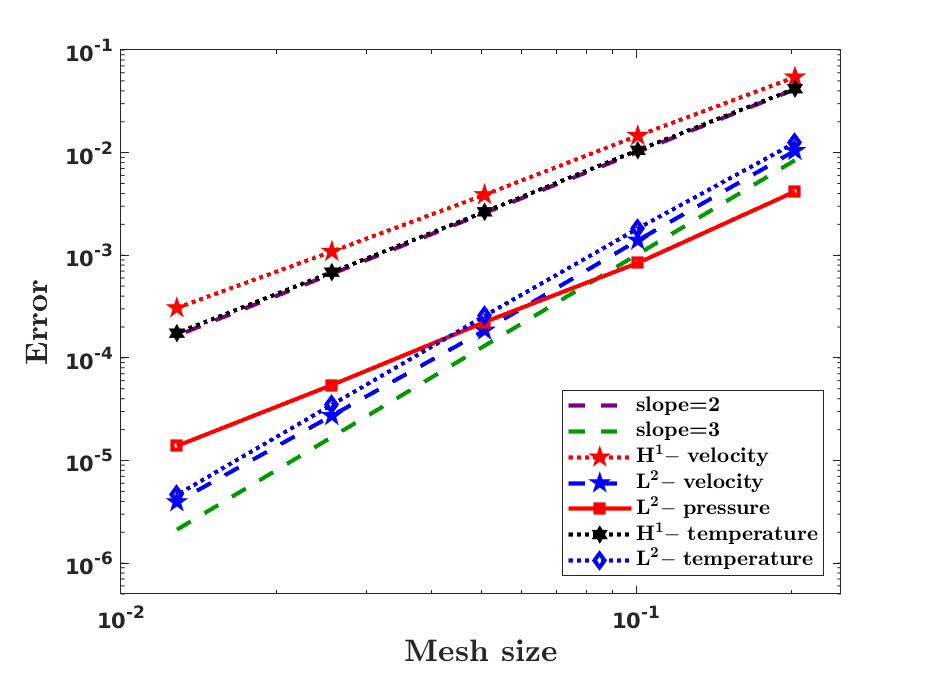}}
	\caption{Example 1: convergence studies on convex domain for VEM order $k=1$ (left) and $k=2$ (right).}
	\label{ex1converg} 
\end{figure}

\subsection{Example 1: Convergence studies on convex domains}
\label{case1}
We begin by investigating the convergence behavior of the proposed method for a regular problem with known solutions on convex domain. The parameters for the considered problem $(P)$ are given as follows:
\begin{align}
	\mu(\theta)= \exp(-\theta),  \qquad \alpha=1, \qquad \mathbf{g}=(0,1)^T, \qquad \kappa(\theta)=1 + \theta^2 + \sin^2(\theta). \nonumber 
\end{align}
The external loads and boundary conditions are given by the exact solution of problem $(P)$: 
\begin{empheq}[left= \empheqlbrace]{align}
	\mathbf{u}(x,y) &= \big[ 10 x^2 y(x-1)^2(y-1)(2y-1), \, -10 xy^2(x-1)(2x-1)(y-1)^2 \big]^T,\nonumber \\
	p(x,y) &= 10 (2x-1)(2y-1), \nonumber\\
	\theta(x,y) &= 10 xy(x-1)(y-1)(x-y)(2xy-x-y-1). \nonumber 
\end{empheq} 
Additionally, we remark that the thermal conductivity $\kappa(\cdot)$ is locally Lipschitz continuous on [-10,10] as $-10 \leq \theta \leq 10$.

We conduct the convergence study of this example for VEM order $k=1$ and $k=2$ using the meshes {$\Omega_1$--$\Omega_5$}. For a regular problem, Proposition \ref{convergece} states that the expected convergence rates are $\mathcal{O}(h^k)$ for the discrete velocity and temperature in the $H^1$ semi-norm and $\mathcal{O}(h^k)$ for the discrete pressure in the $L^2$ norm. Additionally, interpolation estimates claim that the optimal convergence rates for the velocity and temperature fields in the $L^2$ norm may reach $\mathcal{O}(h^{k+1})$. 

The convergence curves using the stabilized VEM are shown in {Figures \ref{ex1converg} and \ref{ex1convergA} for convex and non-convex meshes $\Omega_1$--$\Omega_5$} with VEM order $k=1$  and $k=2$. The convergence curves demonstrate that the proposed stabilized VEM confirms the expected convergence rates as derived in theoretical estimates for $k=1$ and 2. We further observe from {Figures \ref{ex1converg} and \ref{ex1convergA} that the discrete pressure has a convergence rate of $\mathcal{O}(h)$ on distorted $\Omega_3$ and irregular Voronoi $\Omega_5$ meshes} for lowest order $k=1$, whereas it seems to have better convergence for regular meshes $\Omega_1$, $\Omega_2$, and $\Omega_4$. {In addition, the number of the fixed-point iteration for VEM order $k=1$ and $k=2$ is 5 for all meshes $\Omega_1$--$\Omega_5$.} 

\begin{figure}[h]
		\centering
		\subfloat[$\Omega_4$.]{\includegraphics[height=4.5cm, width=6cm]{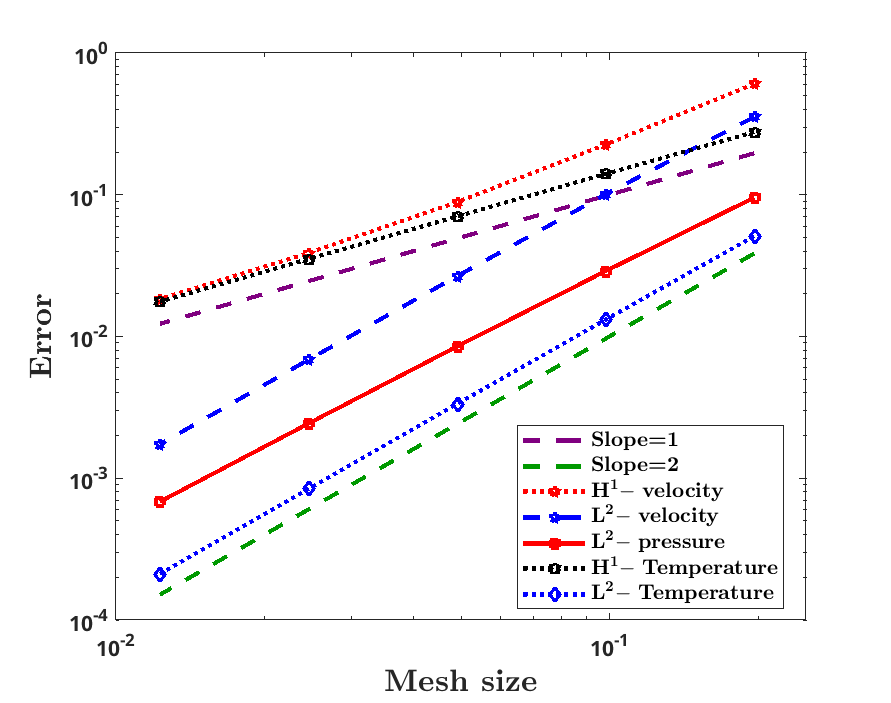}}~~
		\subfloat[$\Omega_4$.]{\includegraphics[height=4.5cm, width=6cm]{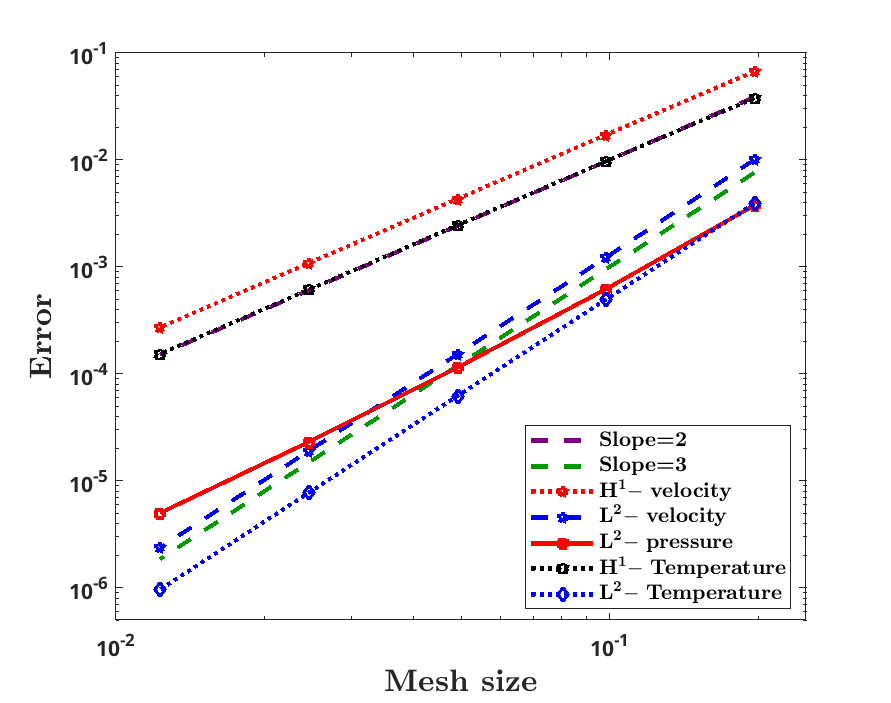}}\\
		\subfloat[$\Omega_5$.]{\includegraphics[height=4.5cm, width=6cm]{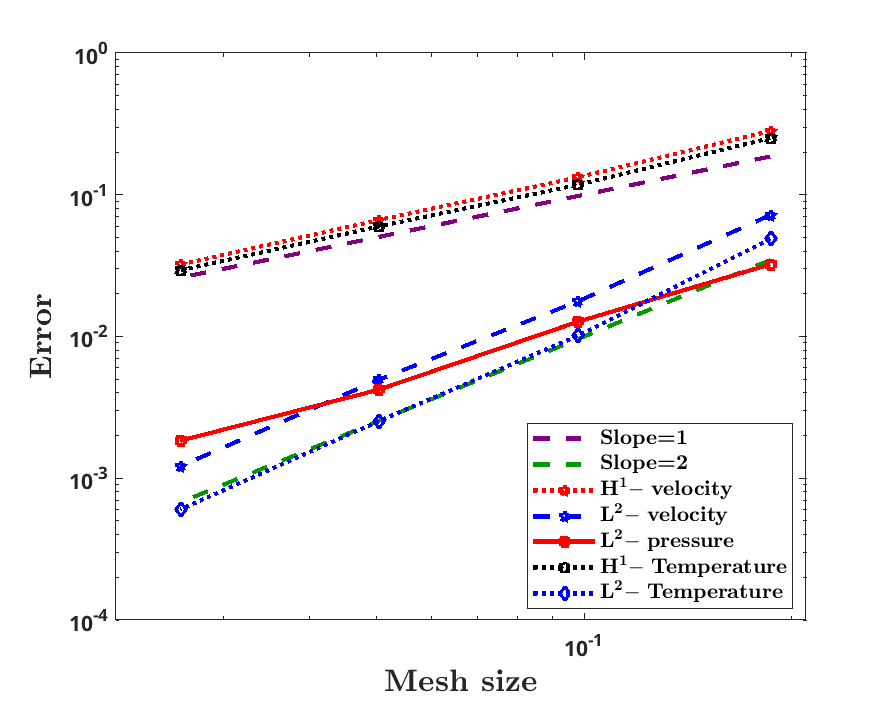}}~~
		\subfloat[$\Omega_5$.]{\includegraphics[height=4.5cm, width=6cm]{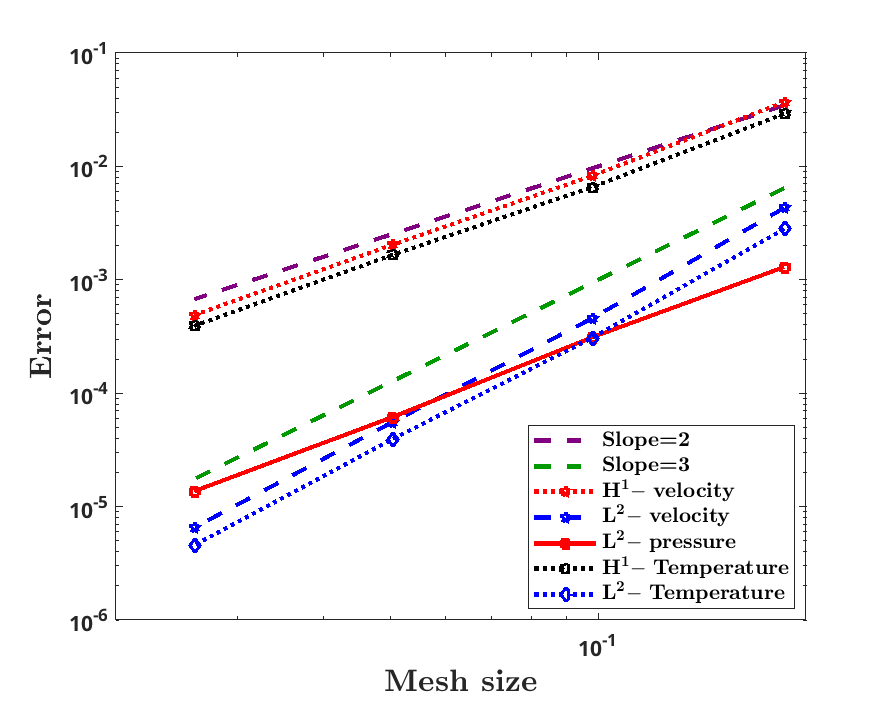}} 
	\caption{Example 1: convergence studies on $\Omega_4$ and $\Omega_5$ for VEM order $k=1$ (left) and $k=2$ (right).}
	\label{ex1convergA} 
\end{figure}

\subsection{Example 2: Convergence studies on non-convex domains}
\label{case2}

We now aim to examine the numerical behavior of the proposed method for a smooth problem with known solution on non-convex domains, as shown in Figure \ref{samp2}. In this sequel, we consider the following parameters of problem $(P)$:
\begin{align}
	\mu(\theta)= \frac{\nu}{1+ \theta^2} ,  \qquad \alpha=1,  \qquad \mathbf{g}=(0,1)^T, \qquad \Large{\text{$\kappa$}}(\theta)= \kappa \sqrt{1 +\theta^2},\nonumber 
\end{align}
where $\nu$ and $\kappa$ are positive parameters.
Furthermore, the external loads and boundary conditions can be obtained using the exact solution of problem $(P)$ given as follows 
\begin{empheq}[left= \empheqlbrace]{align}
	\mathbf{u}(x,y) &= \big[ 4y(x^2-1)^2 (y^2-1), \, -4x(x^2-1)(y^2-1)^2 \big]^T,\nonumber \\
	p(x,y) &= \pi^2 \sin(2\pi x) \sin(2\pi y), \nonumber\\
	\theta(x,y) &= \exp(-(x^2+y^2))-0.5. \nonumber 
\end{empheq} 

In the first case, we will demonstrate the convergence behavior of the proposed method on H-shaped and wrench domains for parameters $\nu=1$ and $\kappa=1$ with VEM order $k=1$ and $k=2$. The obtained numerical results are shown in Figure \ref{ex2fig1} and Tables \ref{ex2_tb1}--\ref{ex2_tb4}. The stabilized discrete solutions are almost similar to analytic solutions, as shown in Figure \ref{ex2fig1}. Moreover, the magnitude of computational errors and convergence rates using the proposed method are presented in Tables \ref{ex2_tb1}--\ref{ex2_tb4}, validating the expected convergence rates derived in the theoretical analysis. Once again, we observe that the discrete pressure exhibits a convergence rate of $\mathcal{O}(h)$ on the H-shaped domain with distorted meshes for VEM order $k=1$, while a better convergence behavior appears to emerge on the wrench-shaped domain with central Voronoi meshes.

\begin{figure}[h]
	\centering
	\subfloat[H-shaped domain.]{\includegraphics[height=3cm, width=4cm]{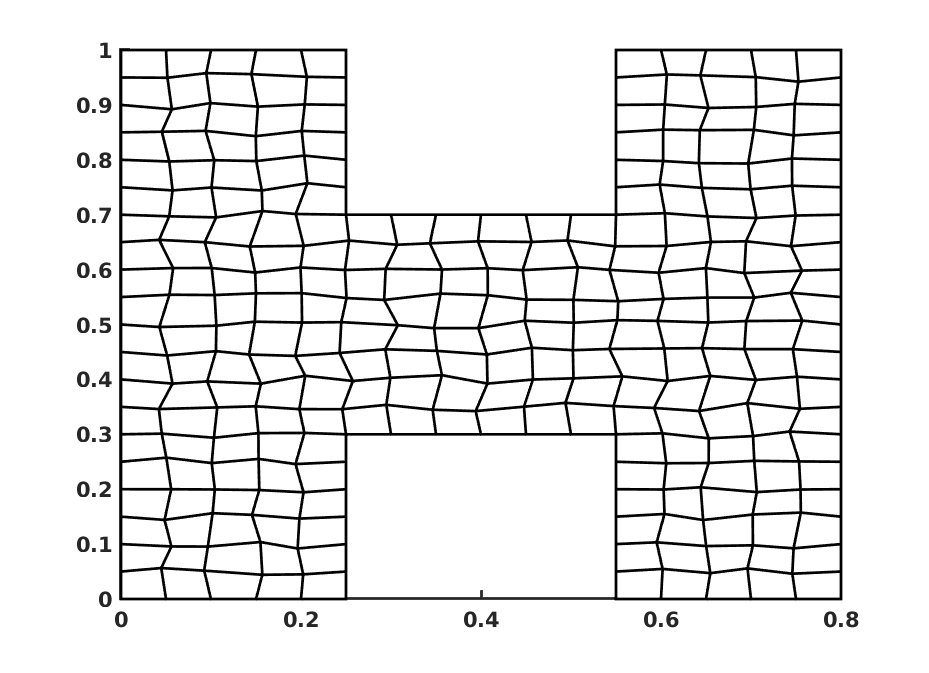}}
	\subfloat[Wrench domain.]{\includegraphics[height=3cm, width=6cm]{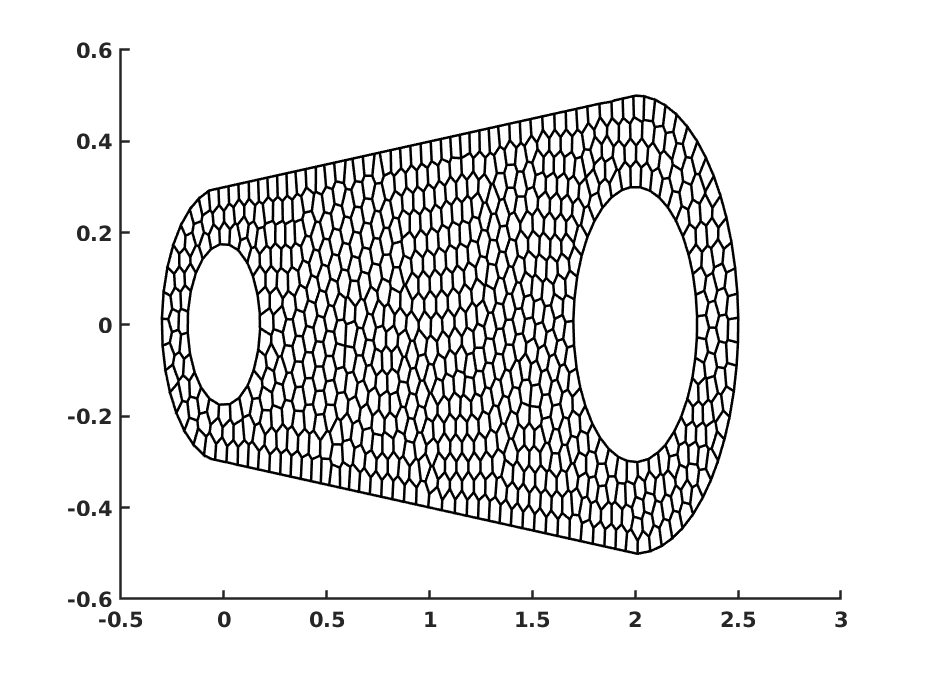}}
	\caption{Example 2: computational domains and representative meshes employed for convergence study.}
	\label{samp2} 
\end{figure}

\begin{figure}[h]
	\centering
	\subfloat[Exact velocity norm.]{\includegraphics[height=3cm, width=4cm]{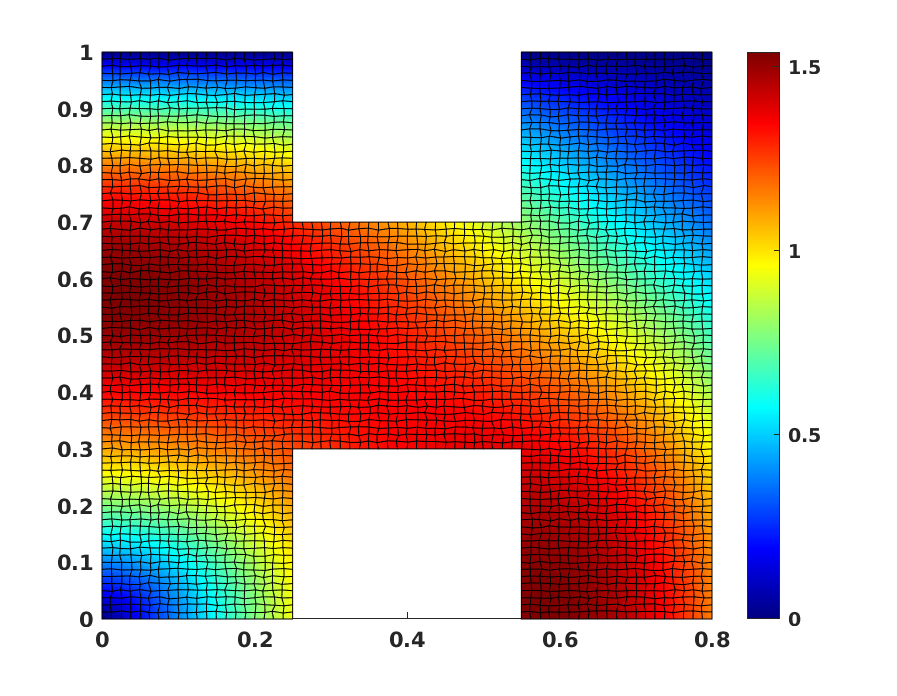}}~~~
	\subfloat[Exact pressure.]{\includegraphics[height=3cm, width=4cm]{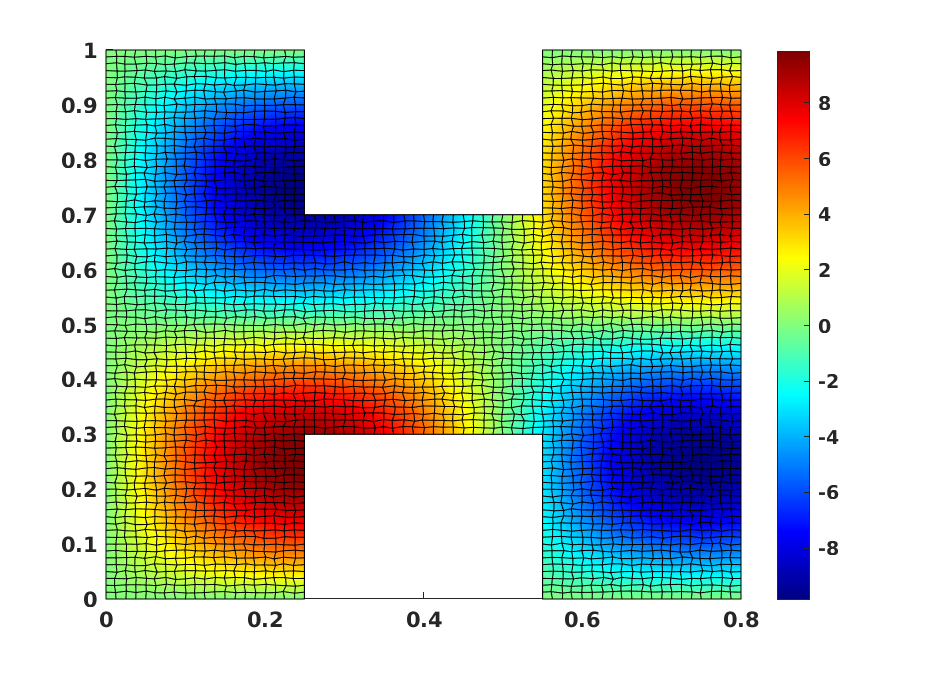}}~~~
	\subfloat[Exact temperature.]{\includegraphics[height=3cm, width=4cm]{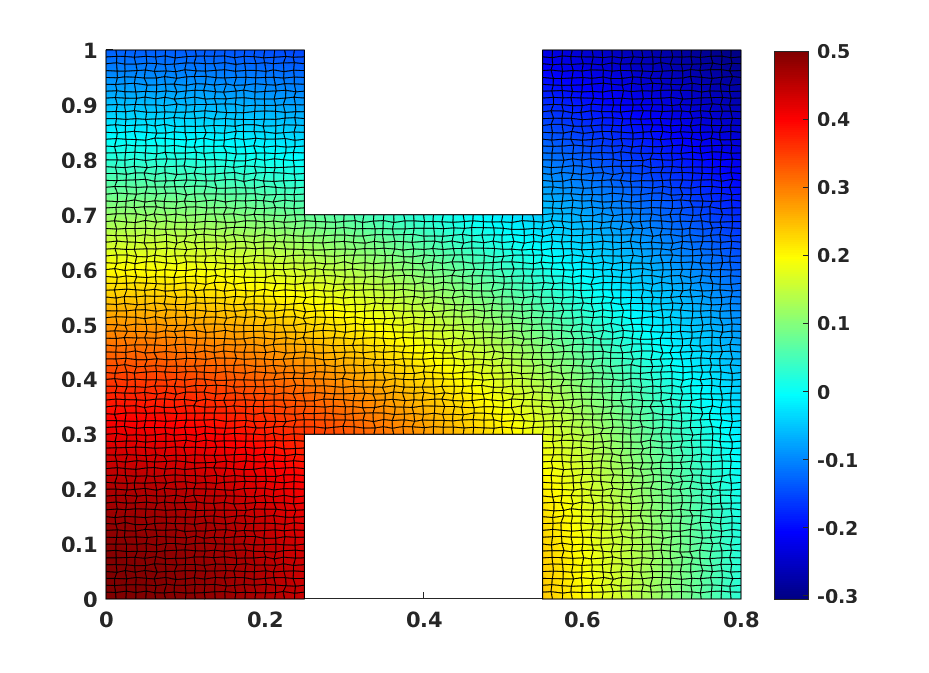}} \\
	\subfloat[Discrete velocity norm.]{\includegraphics[height=3cm, width=4cm]{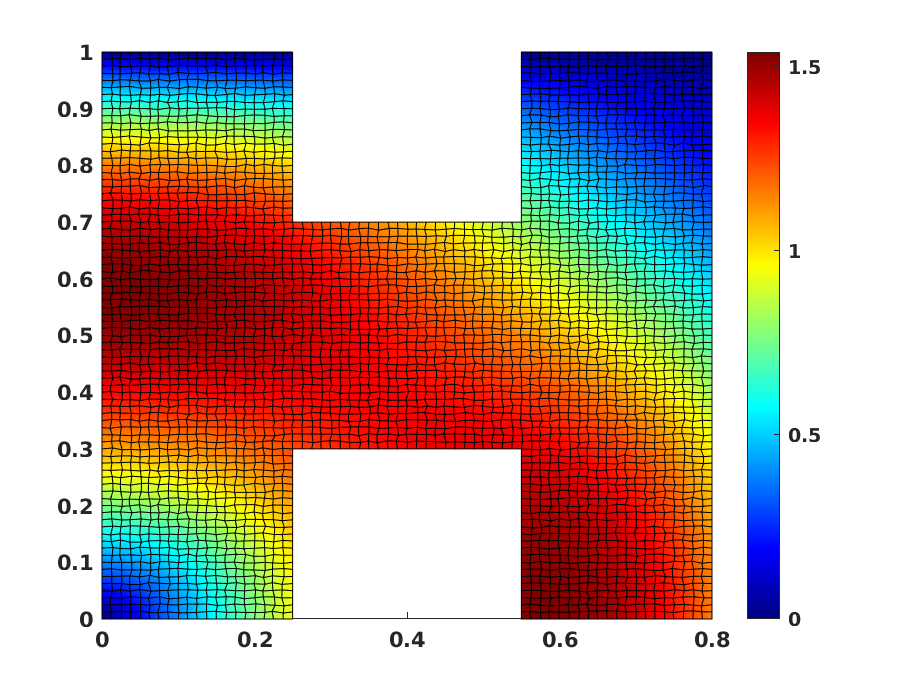}}~~~
	\subfloat[Discrete pressure.]{\includegraphics[height=3cm, width=4cm]{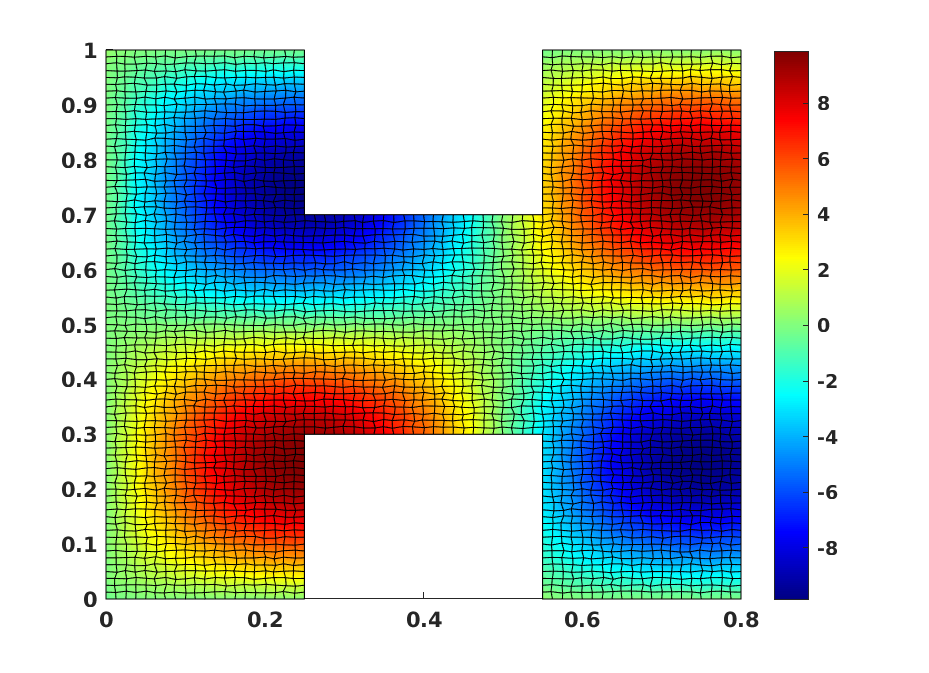}}~~~
	\subfloat[Discrete temperature.]{\includegraphics[height=3cm, width=4cm]{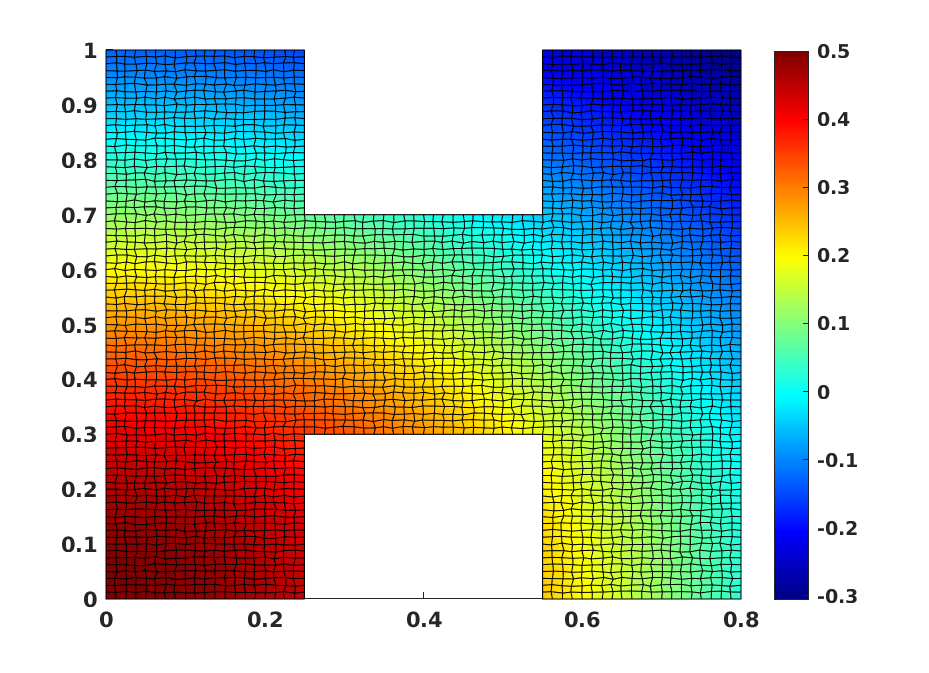}} \\
	\caption{Example 2. (Case $\nu=1$, $\kappa=1$.) Exact solutions (top) and stabilized solutions (bottom) for VEM order $k=1$ with $h=0.0229$.}
	\label{ex2fig1} 
\end{figure}

In the second case, we investigate the efficiency and robustness of the proposed method with respect to parameters $\nu$ and $\kappa$. To achieve this, we first examine the variation of computational errors by fixing $\nu = 1$ and varying $\kappa$ from 1 to $10^{-9}$, for VEM orders $k = 1$ and $k = 2$. The numerical results obtained using the proposed method are shown in Figure~\ref{ex2fig2} for both non-convex domains. Notably, as $\kappa$ decreases from 1 to $10^{-4}$, we observe an increase in the $L^2$ norm and $H^1$ semi-norm of the temperature error on both domains, while the variation of errors in the discrete velocity and pressure remains very small. Upon further reduction of $\kappa$ to $10^{-9}$, the variation of errors in the discrete temperature remains nearly constant for all $\kappa \geq 10^{-5}$. Consequently, the error curves for all virtual element triplets appear almost flat for $\kappa \leq 10^{-5}$, which demonstrates the robustness of the proposed method with respect to $\kappa$. 

Additionally, the variation of computational errors is depicted in Figure \ref{ex2fig3} by fixing $\kappa=1$ and varying $\nu$ from $1$ to $10^{-9}$ for H-shaped and wrench domains with VEM order $k=1$ and $k=2$. We observe almost identical behavior in the variation of errors as we found in the case of varying $\kappa$ with $\nu=1$. In short, the error curves become flat for all errors when $\nu \leq 10^{-4}$.

\begin{table}[t!]
	\setlength{\tabcolsep}{7pt}
	\centering
	\caption{ Example 2. Convergence studies on H-shaped domain using proposed method for VEM order $k=1$ with $\nu=1$ and $\kappa=1$.}
	\label{ex2_tb1}
	\begin{tabular}{cccccccccccc}
		\toprule
		{$h$}&  {$E^\mathbf{u}_{H^1}$} & {rate} & {$E^{\mathbf{u}}_{L^2}$} & {rate} & {$E^p_{L^2}$} & {rate} & {$E^{\theta}_{H^1}$} & {rate} & {$E^{\theta}_{L^2}$} & {rate} &{{itr}} \\
		\midrule
		$0.0910$  & 1.3806e-01 & --   & 2.62560e-02 & --   & 1.4220e-01 &--    & 2.9530e-02 & --   & 1.9487e-03 & -- & {6}\\ 
		$0.0457$ & 4.7017e-02 & 1.55 & 5.0668e-03 & 2.37 & 3.2867e-02 & 2.11 & 1.4764e-02 & 1.00 & 4.7630e-04 & 2.03 & {6}\\ 
		$0.0229$ & 2.0230e-02 & 1.22 & 1.0172e-03 & 2.32 & 1.0756e-02 & 1.61 & 7.3826e-03 & 1.00 & 1.1868e-04 & 2.01 & {6}\\ 
		$0.0115$ & 9.8703e-03 & 1.04 & 2.4474e-04 & 2.06 & 4.7502e-03 & 1.17 & 3.6959e-03 & 1.00 & 2.9769e-05 & 2.00 &{6} \\
		$0.0058$ & 4.9403e-03 & 1.00 & 6.2416e-05 & 1.97 & 2.2979e-03 & 1.05 & 1.8491e-03 & 1.00 & 7.4365e-06 & 2.00 & {6}\\ 
		\bottomrule		
	\end{tabular}
\end{table}

\begin{table}[t!]
	\setlength{\tabcolsep}{7pt}
	\centering
	\caption{Example 2. Convergence studies on wrench domain using proposed method for VEM order $k=1$
		with $\nu=1$ and $\kappa=1$.}
	\label{ex2_tb3}
	\begin{tabular}{cccccccccccc }
		\toprule
		{$h$}&  {$E^\mathbf{u}_{H^1}$} & {rate} & {$E^{\mathbf{u}}_{L^2}$} & {rate} & {$E^p_{L^2}$} & {rate} & {$E^{\theta}_{H^1}$} & {rate} & {$E^{\theta}_{L^2}$} & {rate} & {itr}  \\
		\midrule
		$1/5$  & 1.1762e-01 & --   & 1.3534e-02 & --   & 5.6078e-01 &--    & 7.9305e-02 & --   & 5.2798e-03 & -- & {8}\\ 
		$1/10$ & 4.8373e-02 & 1.28 & 2.4863e-03 & 2.44 & 3.6689e-01 & 0.61 & 3.8181e-02 & 1.06 & 1.4744e-03 & 1.84 & {8}\\ 
		$1/20$ & 2.3215e-02 & 1.06 & 6.0924e-04 & 2.03 & 1.3095e-01 & 1.47 & 1.8542e-02 & 1.04 & 4.3208e-04 & 1.77 & {8} \\ 
		$1/40$ & 1.1540e-02 & 1.01 & 1.6728e-04 & 1.87 & 4.1755e-02 & 1.65 & 9.0309e-03 & 1.04 & 1.1409e-04 & 1.92  & {8} \\
		$1/80$ & 5.5402e-03 & 1.06 & 4.4520e-05 & 1.91 & 1.3033e-02 & 1.68 & 4.3638e-03 & 1.05 & 2.8449e-05 & 2.00 & {8}\\ 
		\bottomrule		
	\end{tabular}
\end{table}

\begin{table}[t!]
	\setlength{\tabcolsep}{7pt}
	\centering
	\caption{Example 2. Convergence studies on H-shaped domain using proposed method for VEM order $k=2$ with $\nu=1$ and $\kappa=1$.}
	\label{ex2_tb2}
	\begin{tabular}{cccccccccccc }
		\toprule
		{$h$}&  {$E^\mathbf{u}_{H^1}$} & {rate} & {$E^{\mathbf{u}}_{L^2}$} & {rate} & {$E^p_{L^2}$} & {rate} & {$E^{\theta}_{H^1}$} & {rate} & {$E^{\theta}_{L^2}$} & {rate}  & {itr} \\
		\midrule
		$0.0910$  & 3.3007e-03 & --   & 1.1588e-04 & --   & 5.6861e-03 &--    & 5.7857e-04 & --   & 1.4516e-05 & -- & {6}\\ 
		$0.0457$ & 9.0300e-04 & 1.87 & 1.6217e-05 & 2.84 & 1.4335e-03 & 1.99 & 1.4496e-04 & 2.00 & 1.8229e-06 & 2.99 & {6}\\ 
		$0.0229$ & 2.2952e-04 & 1.98 & 2.0731e-06 & 2.97 & 3.5919e-04 & 2.00 & 3.6188e-05 & 2.00 & 2.2760e-07 & 3.00 & {6}\\ 
		$0.0115$ & 5.8781e-05 & 1.97 & 2.6582e-07 & 2.96 & 8.9545e-05 & 2.01 & 9.0722e-06 & 2.00 & 2.8585e-08 & 2.99 & {6}\\
		$0.0058$ & 1.4720e-05 & 2.00 & 3.3387e-08 & 2.99 & 2.2435e-05 & 2.00 & 2.2676e-06 & 2.00 & 3.5709e-09 & 3.00 & {6}\\ 
		\bottomrule
	\end{tabular}
\end{table}

\begin{table}[t!]
	\setlength{\tabcolsep}{7pt}
	\centering
	\caption{Example 2. Convergence studies on wrench domain using proposed method for VEM order $k=2$ with $\nu=1$ and $\kappa=1$.}
	\label{ex2_tb4}
	\begin{tabular}{cccccccccccc}
		\toprule
		{$h$}&  {$E^\mathbf{u}_{H^1}$} & {rate} & {$E^{\mathbf{u}}_{L^2}$} & {rate} & {$E^p_{L^2}$} & {rate} & {$E^{\theta}_{H^1}$} & {rate} & {$E^{\theta}_{L^2}$} & {rate} & {itr} \\
		\midrule
		$1/5$  & 1.8306e-02 & --   & 9.2206e-04 & --   & 1.0122e-01 &--    & 5.6642e-03 & --   & 2.2154e-04 & -- & {8} \\ 
		$1/10$ & 3.9297e-03 & 2.22 & 9.8515e-05 & 3.23 & 1.1743e-02 & 3.11 & 1.2039e-03 & 2.23 & 1.8716e-05 & 3.56 &{8}\\ 
		$1/20$ & 8.2115e-04 & 2.26 & 9.4529e-06 & 3.38 & 1.5555e-03 & 2.91 & 2.5434e-04 & 2.24 & 1.6283e-06 & 3.52 & {8}\\ 
		$1/40$ & 2.0681e-04 & 1.99 & 1.1705e-06 & 3.01 & 3.6105e-04 & 2.11 & 6.2347e-05 & 2.03 & 1.9402e-07 & 3.07 & {8}\\
		$1/80$ & 4.7739e-05 & 2.12 & 1.3094e-07 & 3.16 & 8.2928e-05 & 2.12 & 1.4513e-05 & 2.10 & 2.1262e-08 & 3.19 &{9} \\ 
		\bottomrule		
	\end{tabular}
\end{table}
\begin{figure}[h]
	\centering
	\subfloat[Order 1.]{\includegraphics[height=3cm, width=4cm]{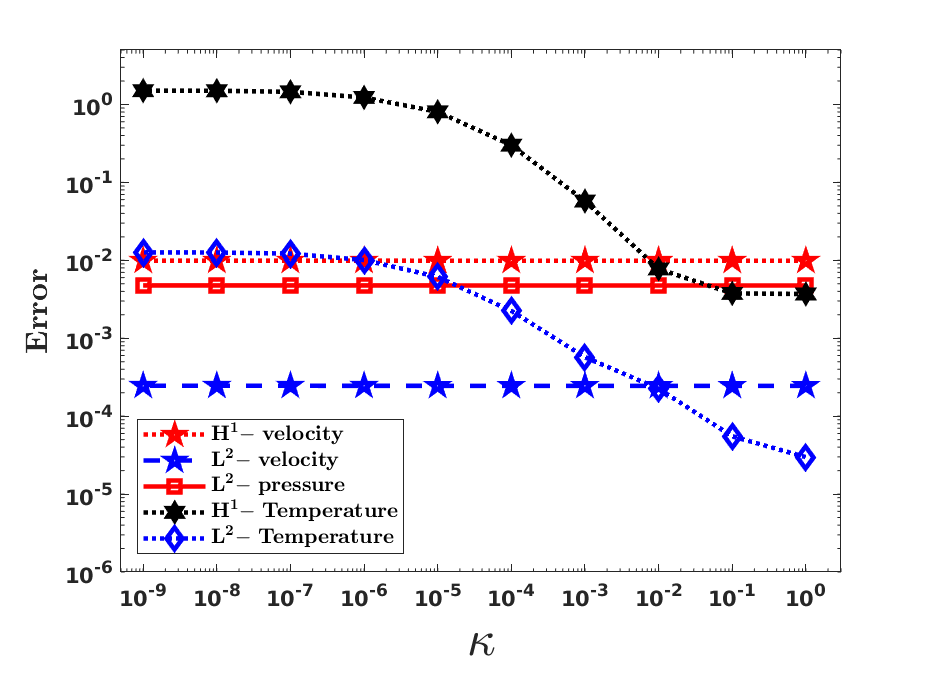}}
	\subfloat[Order 2.]{\includegraphics[height=3cm, width=4cm]{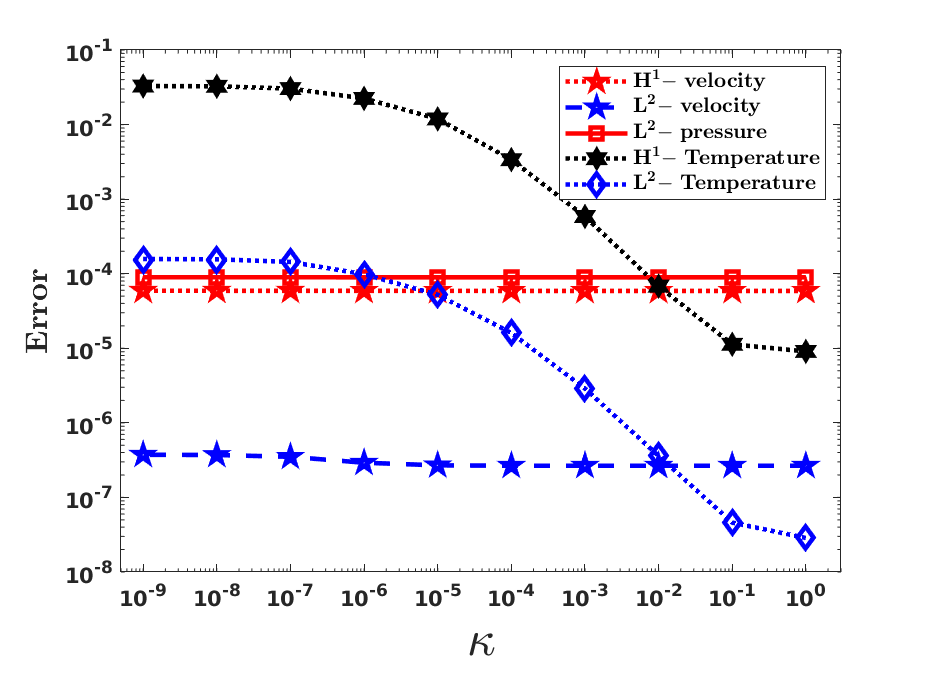}}\\
	\subfloat[Order 1.]{\includegraphics[height=3cm, width=4cm]{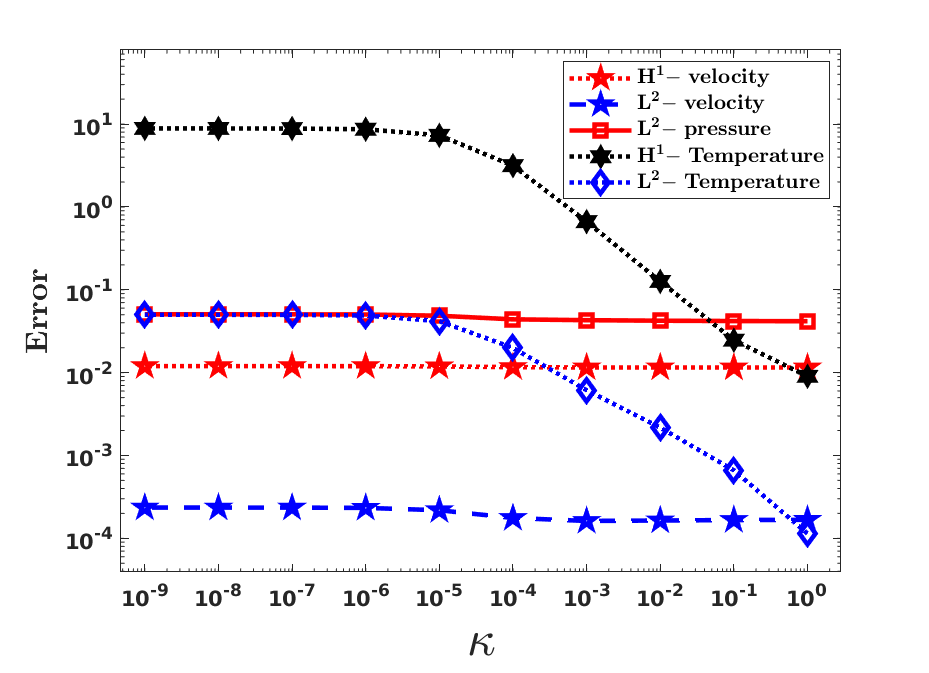}} 
	\subfloat[Order 2.]{\includegraphics[height=3cm, width=4cm]{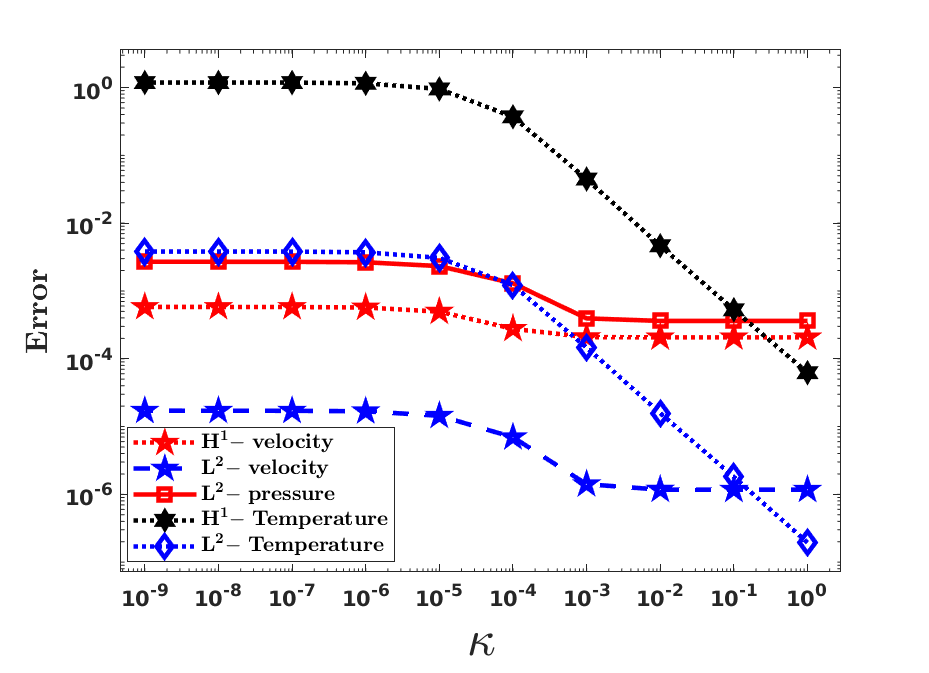}} 
	\caption{Example 2. Variation of computational error on H-shaped (top) and wrench (bottom) domains by fixing $\nu=1$ and varying $\kappa$ with mesh sizes $h=0.0115$ for H-shaped and $h=1/40$ for wrench domains.}
	\label{ex2fig2} 
	\subfloat[Order 1.]{\includegraphics[height=3cm, width=4cm]{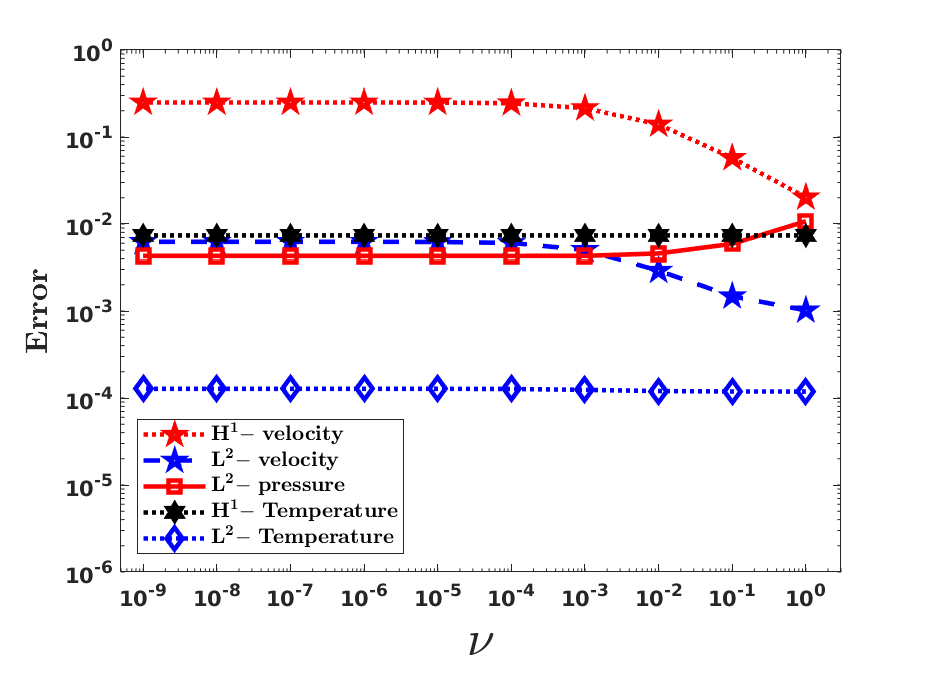}}
	\subfloat[Order 2.]{\includegraphics[height=3cm, width=4cm]{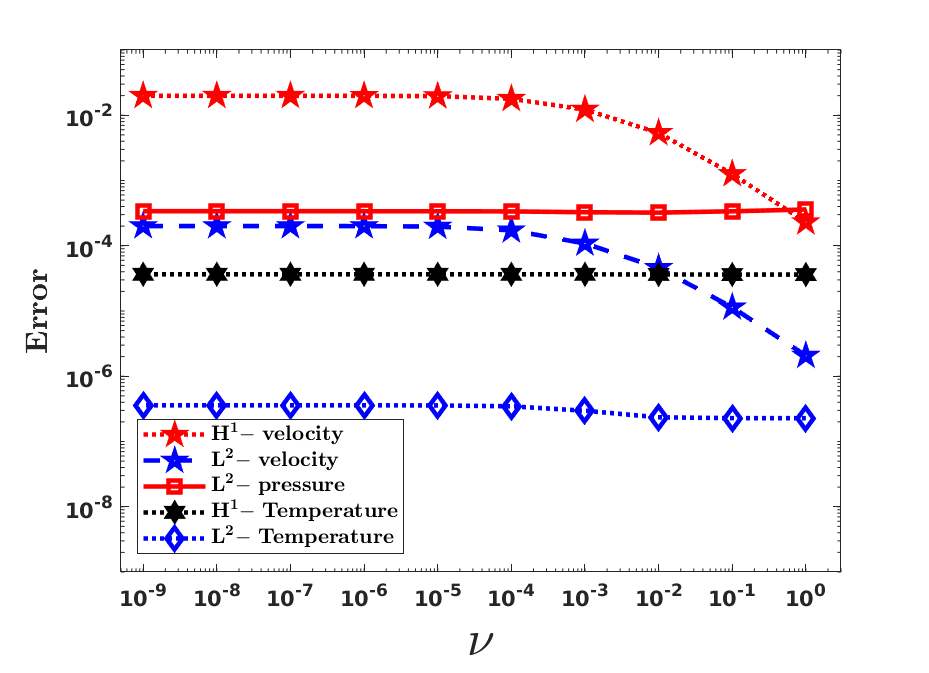}} \\
	\subfloat[Order 1.]{\includegraphics[height=3cm, width=4cm]{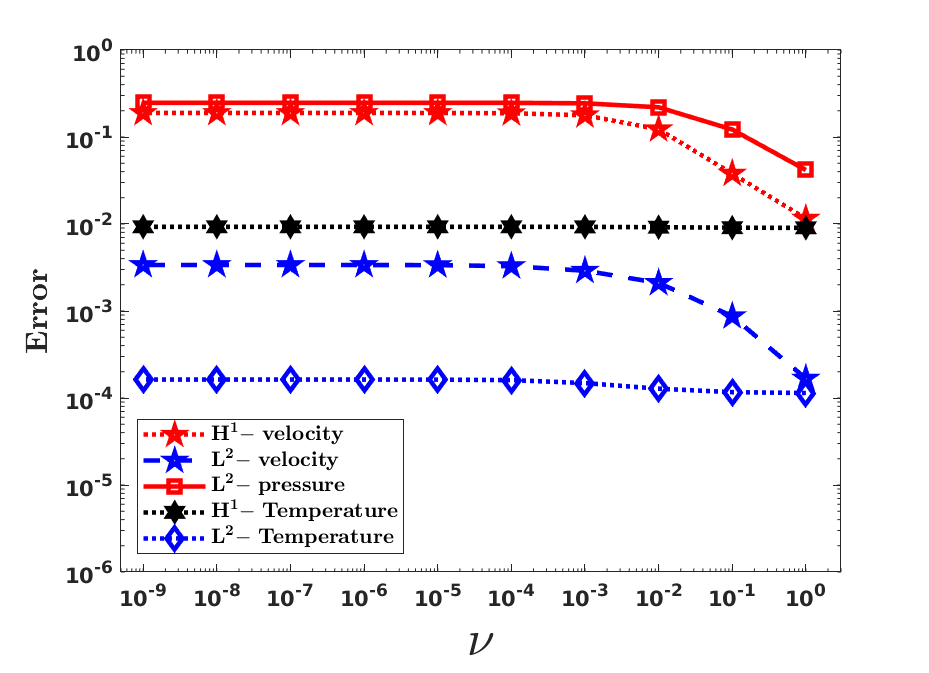}}
	\subfloat[Order 2.]{\includegraphics[height=3cm, width=4cm]{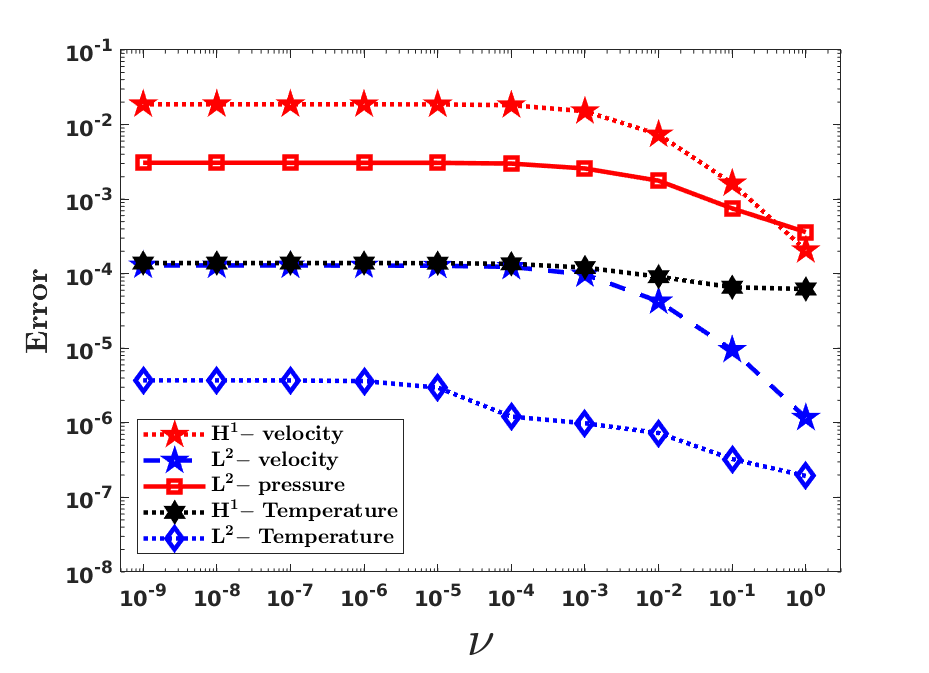}} 
	\caption{Example 2. Variation of computational error on H-shaped (top) and wrench (bottom) domains by fixing $\kappa=1$ and varying $\nu$ with mesh sizes $h=0.0115$ for H-shaped and $h=1/40$ for wrench domains.}
	\label{ex2fig3}
\end{figure}

\subsection{Example 3: A Problem with an Interior Temperature Layer}
\label{case3}

In the third example, we focus on investigating the efficiency of the proposed method for a problem with known solution having an interior temperature layer. To show this, we consider the parameters of the problem $(P)$ as follows:
\begin{align}
	\mu(\theta)= \sqrt{1+ \theta^2},  \qquad \alpha=1,  \qquad \mathbf{g}=(0,1)^T, \qquad \Large \text{$\kappa$}(\theta)=\kappa \exp(\theta), \nonumber 
\end{align}
where $\kappa$ is a positive constant. Furthermore, the external loads and boundary conditions can be determined using the exact solution of problem $(P)$, given by 
\begin{empheq}[left= \empheqlbrace]{align}
	\mathbf{u}(x,y) &= \big[ 2 x^2 y(x-1)^2(y-1)(2y-1), \, -2 xy^2(x-1)(2x-1)(y-1)^2 \big]^T,\nonumber \\
	p(x,y) &= \exp(y)(x-0.5)^3, \nonumber\\
	\theta(x,y) &= 16xy(1-x)(1-y)\Big(\frac{1}{2}+\frac{\tan^{-1}{\big(2 \kappa^{-1/2}(0.25^2-(x-0.5)^2-(y-0.5)^2)\big)}}{\pi} \Big). \nonumber 
\end{empheq} 

We remark that the exact solution of the transport temperature equation represents a hump and is characterized by circular interior layers; for more details, see \cite{vemex4}. 

The convergence study of Example 3 is conducted using $\Omega_1$ and $\Omega_4$. Following \cite{vemex4}, we examine the convergence behavior of the proposed method for $\kappa=10^{-2}$ and $10^{-4}$ with the lowest order VEM $k=1$. The magnitude of numerical errors and convergence rates are reported in Tables \ref{ex3_tb1}--\ref{ex3_tb4} for $\kappa=10^{-2}$ and $10^{-4}$. We observe that the discrete velocity and temperature exhibit convergence rates of $\mathcal{O}(h)$ and $\mathcal{O}(h^2)$ in the $H^1$ semi-norm and $L^2$ norm, respectively, confirming the theoretical results. Moreover, the discrete pressure exhibits a better convergence rate of $\mathcal{O}(h^{1.5})$ in the $L^2$ norm.

We now investigate the efficiency of the proposed method for capturing the circular interior layers for $\kappa=10^{-2}, 10^{-4}$ and $10^{-6}$. The exact and discrete temperatures are depicted in Figure \ref{ex3sol1} for $\kappa=10^{-2}, 10^{-4}$ and $10^{-6}$ employing VEM order $k=1$. We notice that the discrete temperature is almost identical to the exact temperature for $\kappa=10^{-2}$ and $10^{-4}$. Furthermore, the discrete temperature exhibits some mild oscillations near the outer circular layers for $\kappa= 10^{-6}$ with the mesh size $h=1/80$, as shown in Figure \ref{ex3sol1}(f). To eliminate these oscillations, we present the discrete solutions computed using the higher-order VEM ($k = 2$) with the mesh size $h = 1/80$ in Figure \ref{ex3sol2}. Additionally, the isotherms corresponding to $\kappa =  10^{-6}$ are depicted in Figure \ref{ex3sol3} for VEM order $k = 1$ and $k = 2$. It is noteworthy that the higher-order VEM ($k = 2$) effectively suppresses all non-physical oscillations and accurately captures the circular interior layers, as illustrated in Figures \ref{ex3sol2} and \ref{ex3sol3}(c). Thus, we can conclude that the efficiency of the proposed method increases as we increase the order of the method.

\begin{table}[t!]
	\setlength{\tabcolsep}{7pt}
	\centering
	\caption{Example 3. Convergence studies on $\Omega_1$ using proposed method for VEM order $k=1$ with $\kappa=10^{-2}$.}
	\label{ex3_tb1}
	\begin{tabular}{ccccccccccc}
  	\toprule
		{$h$}&  {$E^\mathbf{u}_{H^1}$} & {rate} & {$E^{\mathbf{u}}_{L^2}$} & {rate} & {$E^p_{L^2}$} & {rate} & {$E^{\theta}_{H^1}$} & {rate} & {$E^{\theta}_{L^2}$} & {rate}  \\
		\midrule
		$1/5$  & 4.0205e-01 & --   & 1.6808e-01 & --   & 2.6401e-01 &--    & 3.7163e-01 & --   & 1.0280e-01 & -- \\ 
		$1/10$ & 2.0187e-01 & 0.99 & 4.1987e-02 & 2.00 & 9.0583e-02 & 1.54 & 1.9369e-01 & 0.94 & 2.9486e-02 & 1.80\\ 
		$1/20$ & 1.0045e-01 & 1.01 & 1.0375e-02 & 2.02 & 2.9928e-02 & 1.60 & 9.9388e-02 & 0.96 & 8.2466e-03 & 1.83\\ 
		$1/40$ & 5.0030e-02 & 1.01 & 2.6020e-03 & 1.99 & 1.0218e-02 & 1.55 & 5.0004e-02 & 0.99 & 2.3672e-03 & 1.80\\
		$1/80$ & 2.4959e-02 & 1.00 & 6.6672e-04 & 1.97 & 3.5647e-03 & 1.52 & 2.5042e-02 & 1.00 & 6.5453e-04 & 1.86\\ 
		\bottomrule		
	\end{tabular}
\end{table}

\begin{table}[t!]
	\setlength{\tabcolsep}{7pt}
	\centering
	\caption{Example 3. Convergence studies on $\Omega_4$ using proposed method for VEM order $k=1$ with $\kappa=10^{-2}$.}
	\label{ex3_tb2}
	\begin{tabular}{ccccccccccc }
		\toprule
		{$h$}&  {$E^\mathbf{u}_{H^1}$} & {rate} & {$E^{\mathbf{u}}_{L^2}$} & {rate} & {$E^p_{L^2}$} & {rate} & {$E^{\theta}_{H^1}$} & {rate} & {$E^{\theta}_{L^2}$} & {rate}  \\
		\midrule
		$1/5$  & 4.6982e-01 & --   & 3.6279e-01 & --   & 4.2851e-01 &--    & 2.7352e-01 & --   & 4.6178e-02 & -- \\ 
		$1/10$ & 2.0107e-01 & 1.22 & 9.3554e-02 & 1.96 & 1.6717e-01 & 1.36 & 1.4172e-01 & 0.95 & 1.1505e-02 & 2.01\\ 
		$1/20$ & 8.6032e-02 & 1.23 & 1.8522e-02 & 2.34 & 5.5279e-02 & 1.60 & 7.1380e-02 & 0.99 & 2.7600e-03 & 2.06\\ 
		$1/40$ & 3.9072e-02 & 1.14 & 3.5994e-03 & 2.36 & 1.7597e-02 & 1.65 & 3.5752e-02 & 1.00 & 6.6704e-04 & 2.05\\
		$1/80$ & 1.8580e-02 & 1.07 & 8.0964e-04 & 2.15 & 5.5066e-03 & 1.68 & 1.7884e-02 & 1.00 & 1.6397e-04 & 2.02\\ 
		\bottomrule		
	\end{tabular}
\end{table}

\begin{table}[t!]
		\setlength{\tabcolsep}{7pt}
		\centering
	\caption{Example 3. Convergence studies on $\Omega_1$ using proposed method for VEM order $k=1$ with $\kappa=10^{-4}$.}
	\label{ex3_tb3}
	\begin{tabular}{ccccccccccc }
		\toprule
		{$h$}&  {$E^\mathbf{u}_{H^1}$} & {rate} & {$E^{\mathbf{u}}_{L^2}$} & {rate} & {$E^p_{L^2}$} & {rate} & {$E^{\theta}_{H^1}$} & {rate} & {$E^{\theta}_{L^2}$} & {rate}  \\
		\midrule
		$1/5$  & 4.1165e-01 & --   & 2.3286e-01 & --   & 1.2118 &--    & 9.2846e-01 & --   & 8.9681e-01 & -- \\ 
		$1/10$ & 2.0335e-01 & 1.02 & 5.2208e-02 & 2.16 & 2.6960e-01 & 2.17 & 6.6528e-01 & 0.48 & 2.2589e-01 & 1.99\\ 
		$1/20$ & 1.0057e-01 & 1.02 & 1.1357e-02 & 2.20 & 4.7818e-02 & 2.50 & 5.1114e-01 & 0.38 & 6.6937e-02 & 1.76\\ 
		$1/40$ & 5.0046e-02 & 1.01 & 2.4273e-03 & 2.23 & 1.0375e-02 & 2.21 & 3.2584e-01 & 0.65 & 2.1894e-02 & 1.61\\
		$1/80$ & 2.4962e-02 & 1.00 & 5.8673e-04 & 2.05 & 3.5372e-03 & 1.55 & 1.8099e-01 & 0.85 & 6.6456e-03 & 1.72\\ 
		$1/160$ & 1.2466e-02 & 1.00 & 1.5059e-04 & 1.96 & 1.2439e-03 & 1.51 & 9.4173e-02 & 0.94 & 1.8213e-03 & 1.87\\ 
		\bottomrule		
	\end{tabular}
\end{table}

\begin{table}[t!]
		\setlength{\tabcolsep}{7pt}
		\centering
	\caption{Example 3. Convergence studies on $\Omega_4$ using proposed method for VEM order $k=1$ with $\kappa=10^{-4}$.}
	\label{ex3_tb4}
	\begin{tabular}{ccccccccccc }
		\toprule
		{$h$}&  {$E^\mathbf{u}_{H^1}$} & {rate} & {$E^{\mathbf{u}}_{L^2}$} & {rate} & {$E^p_{L^2}$} & {rate} & {$E^{\theta}_{H^1}$} & {rate} & {$E^{\theta}_{L^2}$} & {rate}  \\
		\midrule
		$1/5$  & 4.6372e-01 & --   & 3.4995e-01 & --   & 4.9029e-01 &--    & 7.7865e-01 & --   & 3.8045e-01 & -- \\ 
		$1/10$ & 1.9998e-01 & 1.21 & 8.5833e-02 & 2.03 & 1.8172e-01 & 1.43 & 6.1302e-01 & 0.35 & 1.1699e-01 & 1.70\\ 
		$1/20$ & 8.6400e-02 & 1.21 & 1.7706e-02 & 2.28 & 7.0746e-02 & 1.36 & 4.4205e-01 & 0.47 & 5.1677e-02 & 1.18\\ 
		$1/40$ & 3.9169e-02 & 1.14 & 3.8757e-03 & 2.19 & 2.2806e-02 & 1.63 & 2.5820e-01 & 0.78 & 1.5948e-02 & 1.70\\
		$1/80$ & 1.8599e-02 & 1.07 & 9.9667e-04 & 1.96 & 7.0198e-03 & 1.69 & 1.3499e-01 & 0.94 & 4.4007e-03 & 1.86\\ 
		\bottomrule		
	\end{tabular}
\end{table}


\begin{figure}[h]
	\centering
	\subfloat[$\kappa=10^{-2}$.]{\includegraphics[height=3cm, width=5cm]{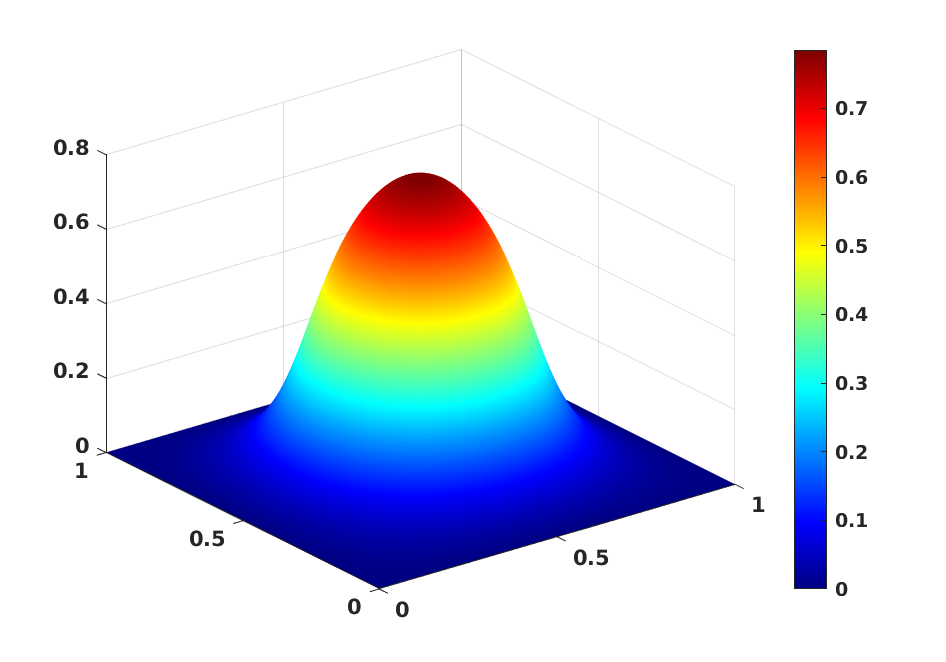}}
	\subfloat[$\kappa=10^{-4}$.]{\includegraphics[height=3cm, width=5cm]{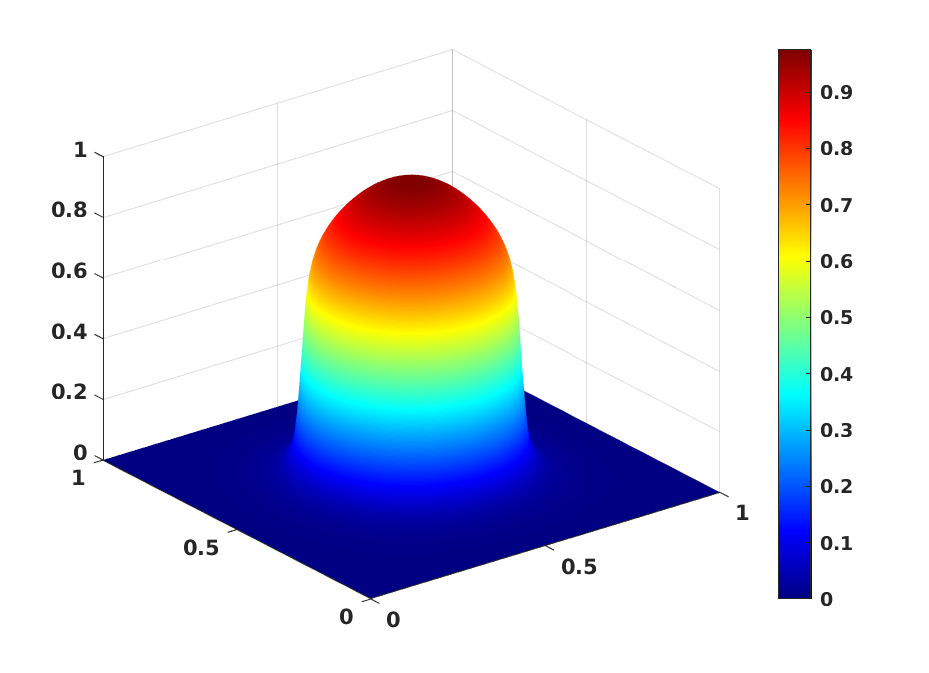}}
	\subfloat[$\kappa= 10^{-6}$.]{\includegraphics[height=3cm, width=5cm]{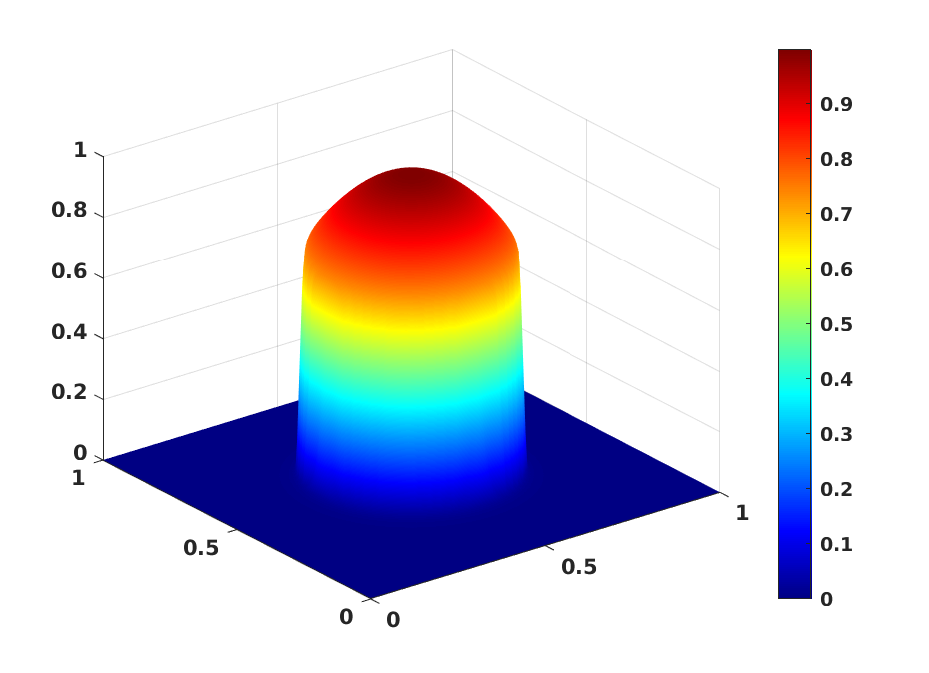}} \\
	\subfloat[$\kappa=10^{-2}$.]{\includegraphics[height=3cm, width=5cm]{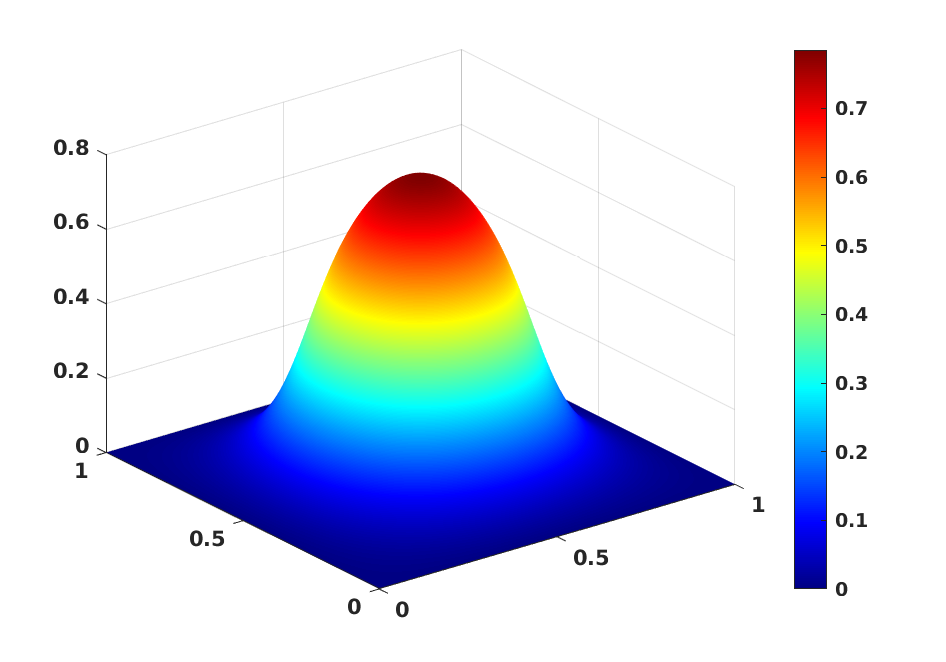}}
	\subfloat[$\kappa=10^{-4}$.]{\includegraphics[height=3cm, width=5cm]{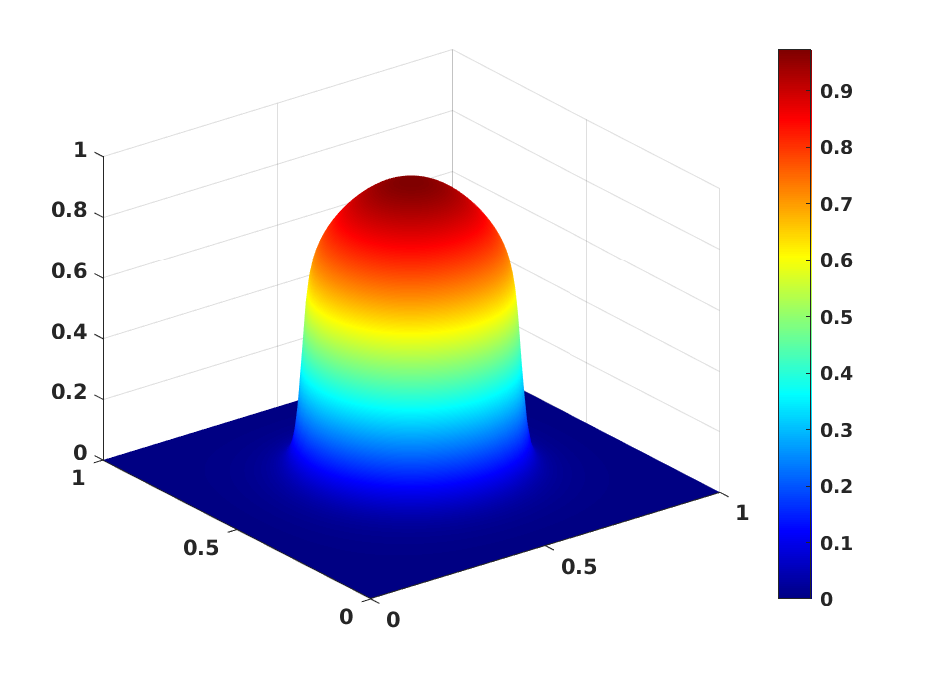}}
	\subfloat[$\kappa= 10^{-6}$.]{\includegraphics[height=3cm, width=5cm]{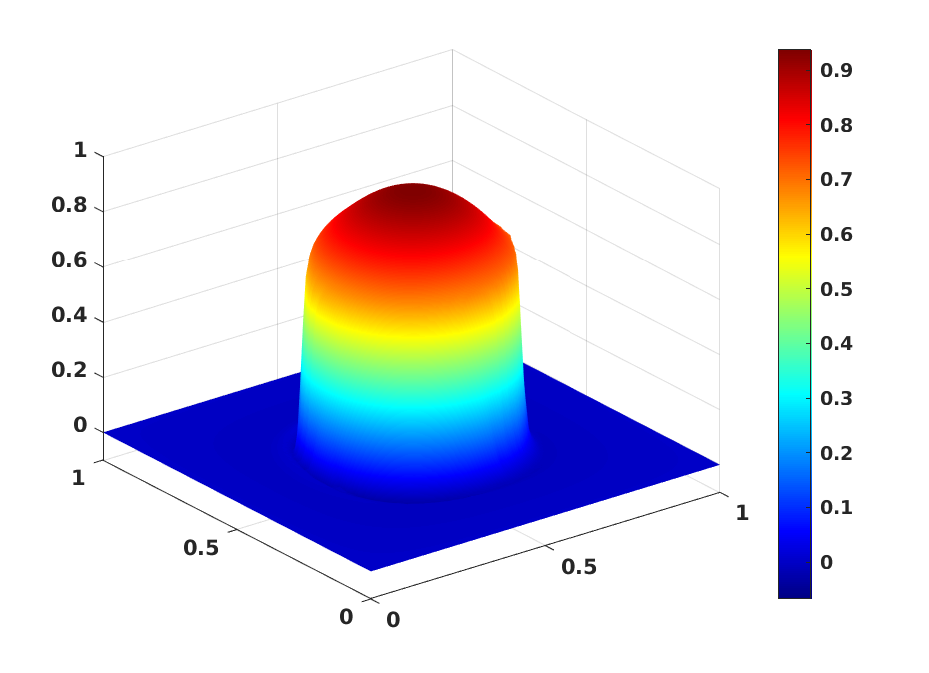}} 
	\caption{ Example 3. Exact temperature (top) and stabilized temperature (bottom) for VEM order $k=1$ with $h=1/80$.}
	\label{ex3sol1} 
\end{figure}

\begin{figure}[h]
	\centering
	\subfloat[$\kappa=10^{-6}$.]{\includegraphics[height=3cm, width=5cm]{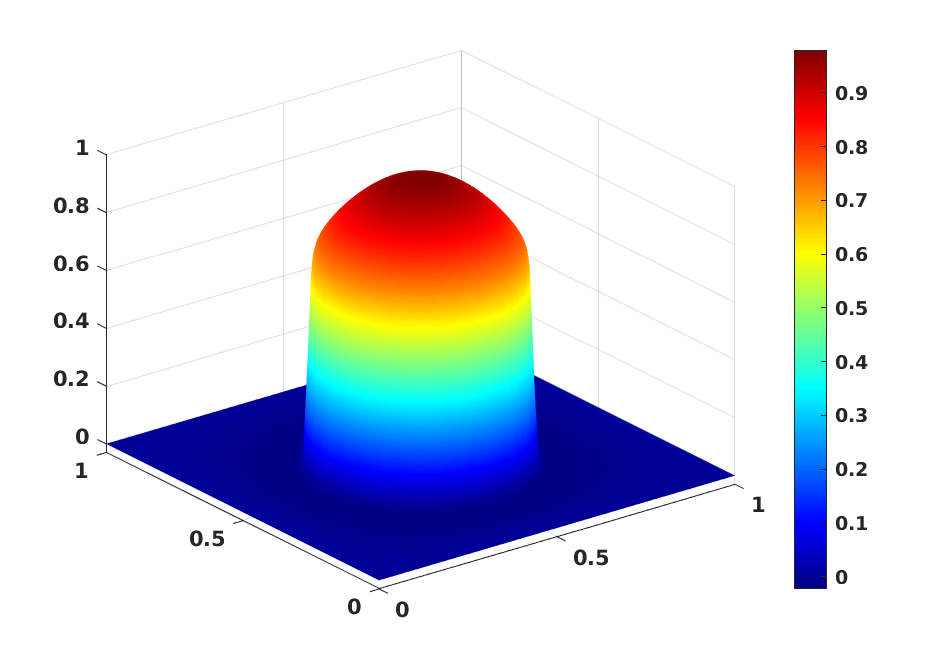}} 
	\caption{Example 3. Stabilized temperature for VEM order $k=2$ with $h=1/80$.}
	\label{ex3sol2}
	\subfloat[Exact.]{\includegraphics[height=3cm, width=5cm]{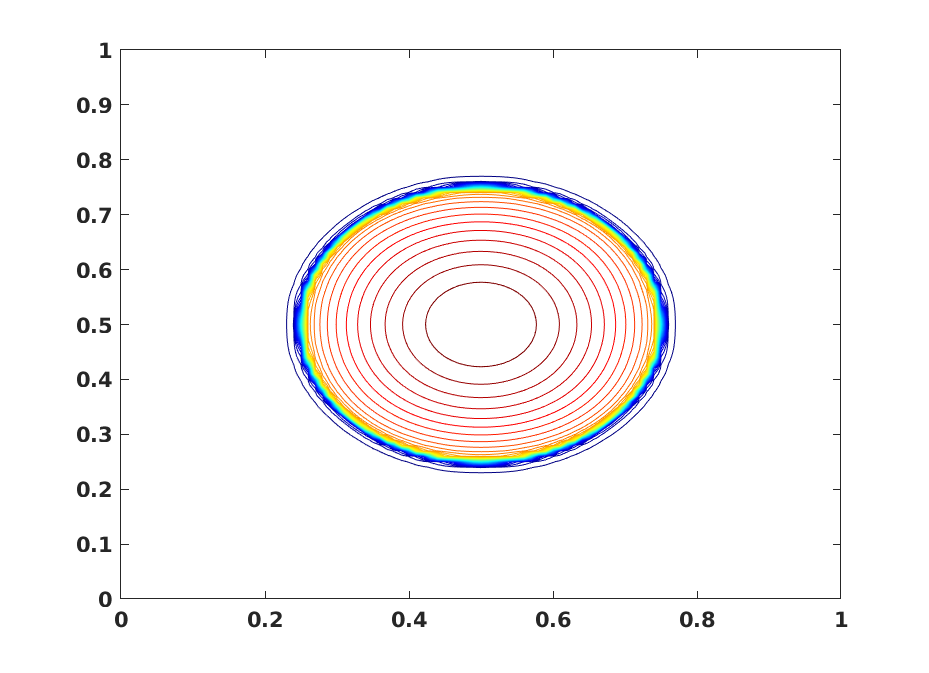}}
	\subfloat[Order 1.]{\includegraphics[height=3cm, width=5cm]{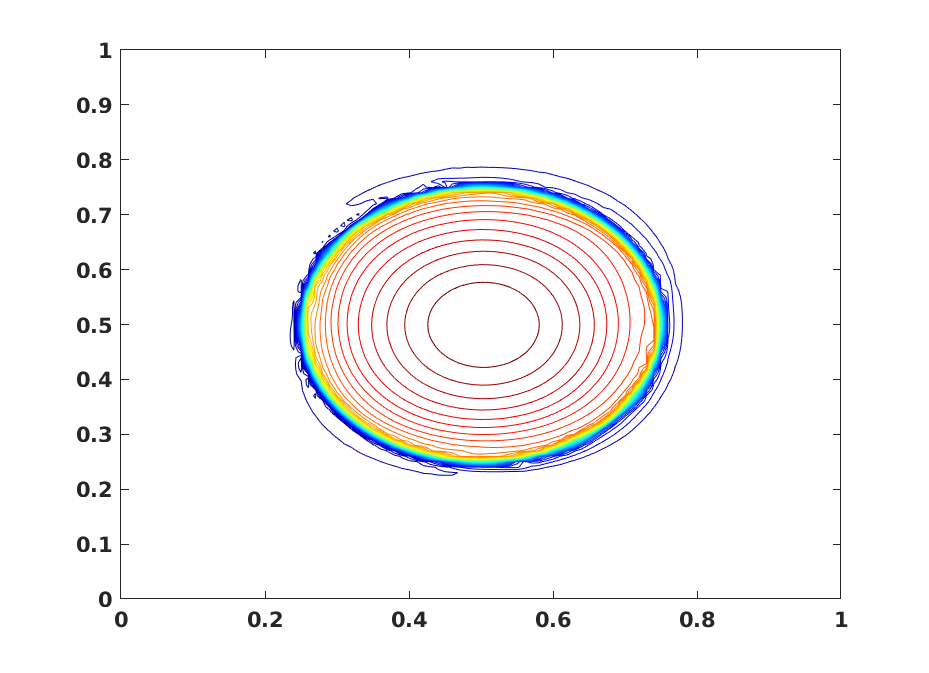}} 
	\subfloat[ Order 2.]{\includegraphics[height=3cm, width=5cm]{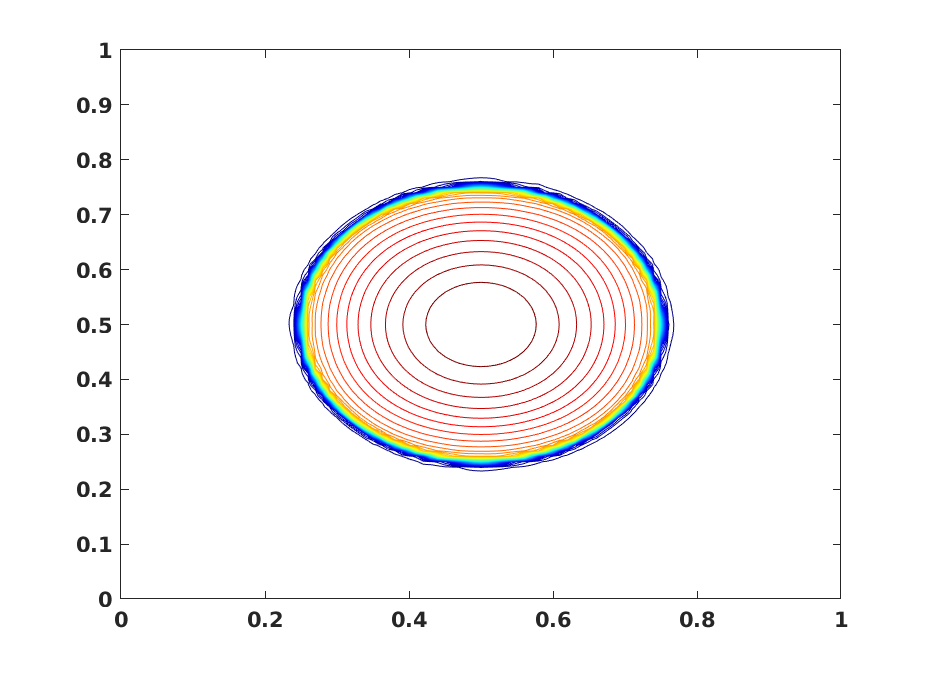}}
	\caption{ Example 3. Exact isotherms (a), stabilized isotherms for $k=1$ (b), and stabilized isotherms for $k=2$ (c), with $\kappa= 10^{-6}$ and $h=1/80$.}
	\label{ex3sol3} 
\end{figure}

\subsection{Example 4: Natural convection in a square cavity}
\label{case4}
Hereafter, we focus on showing the advantages of the proposed method in addressing the practical applications. To do this, we investigate the flow dynamics of a fluid in a unit square cavity with differentially heated walls. Recalling \cite{dalal2006natural}, we introduce the dimensionless form of the problem $(P)$: find \((\mathbf{u}, p, \theta)\) such that

\begin{equation}
	\label{eq:main_system}
	\begin{aligned}
		- \Pr \, \Delta \mathbf{u} + (\nabla \mathbf{u})\, \mathbf{u} + \nabla p &= \Pr Ra \, \theta\, \mathbf{g} && \text{in } \Omega, \\
		\nabla \cdot \mathbf{u} &= 0 && \text{in } \Omega, \\
		\mathbf{u} &= \mathbf{0} && \text{in } \partial \Omega, \\
		-  \Delta \theta + \mathbf{u} \cdot \nabla \theta &= 0 && \text{in } \Omega,\\
		\theta &= \theta_D  &&  \text{on}\,\,  \partial \Omega, 
	\end{aligned}
\end{equation}
where $\Pr$ and $Ra$ are the Prandtl and Rayleigh numbers, respectively. Following \cite{almonacid2018,colmenares2020banach,mfem2020divergence}, we consider the following Prandtl and Rayleigh numbers: $$\Pr=0.5 \quad  \text{and} \quad Ra=2000.$$ We fix $\mathbf{g}=(0,1)^T$. Furthermore, the boundary condition for temperature is given as follows:
\begin{equation*}
	\theta_D(x,y)=\begin{cases}
		0 &\, \text{if $ x=0$, or $x=1$, or $y=1$, }\\
		0.5 (1-\cos(2\pi x) )(1-y) &\, \text{if $y=0$.}\\
	\end{cases}
\end{equation*} 

\begin{figure}[h]
	\centering
	\subfloat[]{\includegraphics[height=4cm, width=5cm]{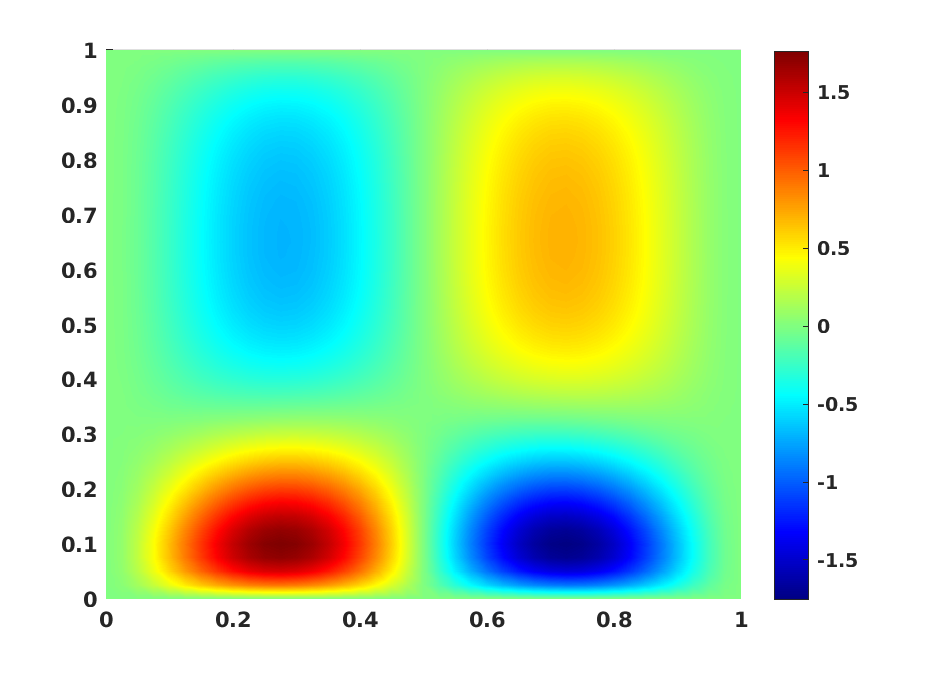}}
	\subfloat[]{\includegraphics[height=4cm, width=5cm]{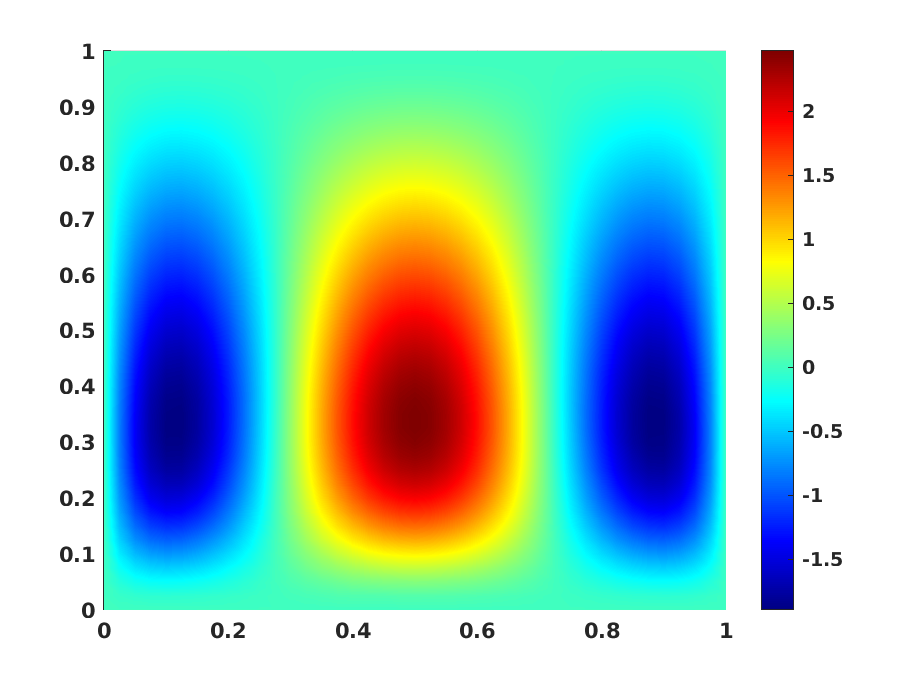}}
	\subfloat[]{\includegraphics[height=4cm, width=5cm]{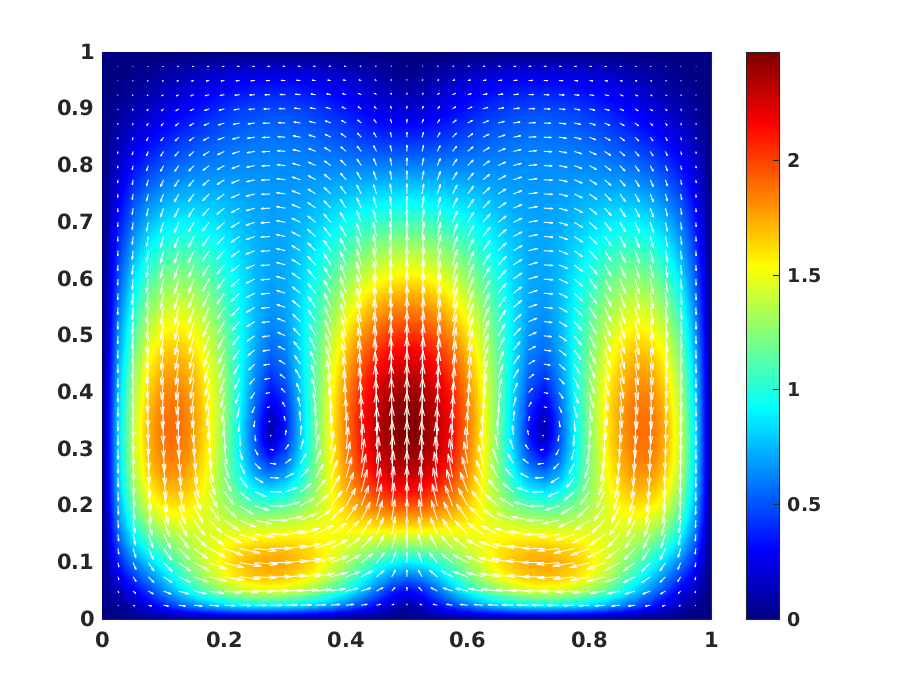}} \\
	\subfloat[]{\includegraphics[height=4cm, width=5cm]{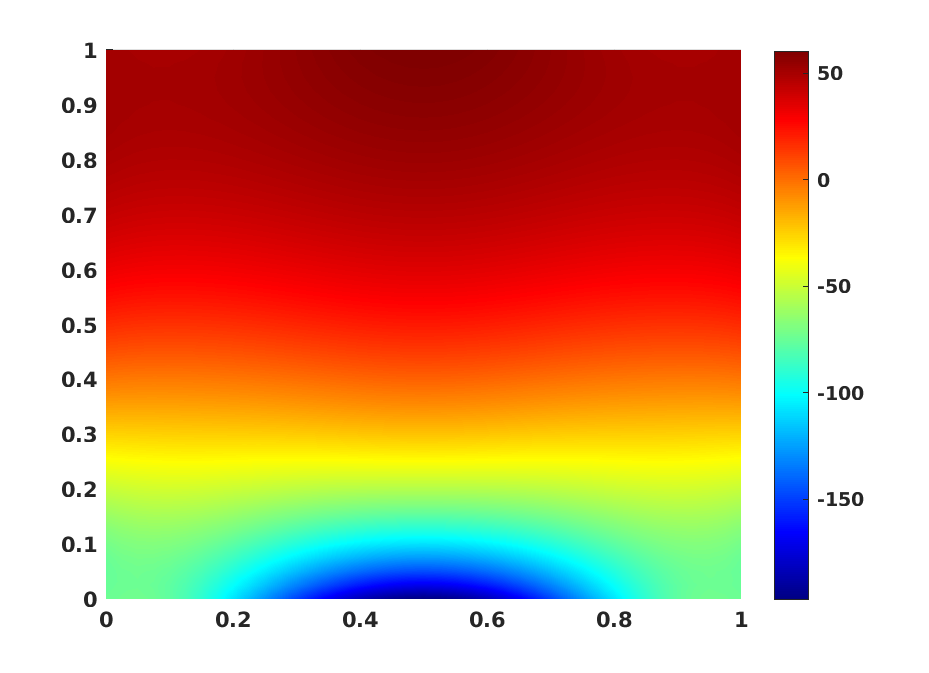}}
	\subfloat[]{\includegraphics[height=4cm, width=5cm]{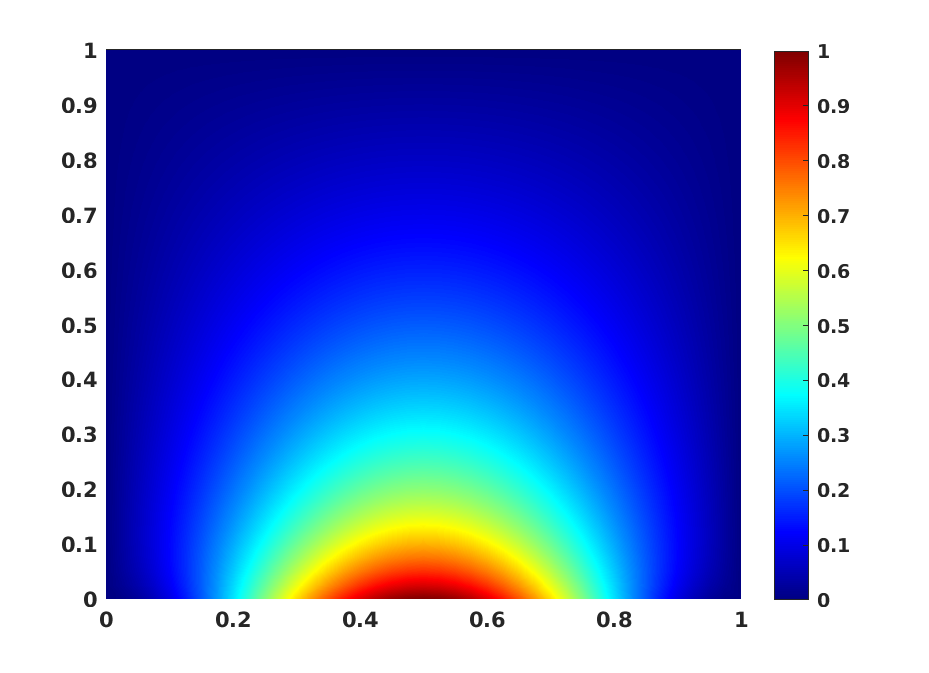}}
	\caption{ Example 4. Components of discrete velocity (a)-(b), velocity magnitude (c), discrete pressure (d), and discrete temperature (e) for VEM order $k=1$ with $\Omega_4$.}
	\label{ex4sol1} 
\end{figure}

We conduct this simulation on a non-convex mesh $\Omega_4$ with $25600$ elements and $76801$ degrees of freedom, employing lowest order VEM. The components of the discrete velocity field and their magnitude, pressure, and temperature are shown in Figure \ref{ex4sol1}. We observe that the flow dynamics within the cavity exhibit a symmetric pattern, consistent with the physical behavior reported in \cite{dalal2006natural,mfem2020divergence,colmenares2020banach}. 
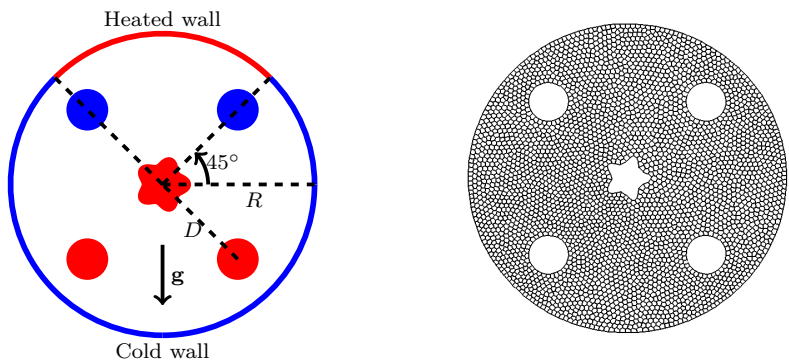
\begin{figure}[htbp]
	\centering
	\begin{tikzpicture}[scale=1]
		\draw [line width=.75mm, color=red] (1.4142,1.4142) arc[start angle=45, end angle=135, radius=2];
		\draw [line width=0.75mm, color=blue] (2,0) arc[start angle=0, end angle=45, radius=2];
		\draw [line width=0.75mm, color=blue] (2,0) arc[start angle=0, end angle=-90, radius=2];
		\draw [line width=0.75mm, color=blue] (-1.4142,1.4142) arc[start angle=135, end angle=270, radius=2];
		
		\draw[thick, fill=blue, color=blue, line width=0.5mm] (0.99,0.99) circle (0.25);
		\draw[thick, fill=blue, color=blue,line width=0.5mm] (-0.99,0.99) circle (0.25);
		\draw[thick, fill=blue, color=red, line width=0.5mm] (0.99,-0.99) circle (0.25);
		\draw[thick, color=red, fill=red, line width=0.5mm] (-0.99,-0.99) circle (0.25);
		\draw[thick, ->, line width=0.5mm] (0,-0.8) -- (0,-1.6);
		\node at (0.2, -1.25) {$\mathbf{g}$};
		\def\radius{0.3}
		\def\ampl{0.05}
		\def\nwaves{5}
		\draw[thick, fill=red, color=red, line width=0.5mm, domain=0:360, samples=400, smooth, variable=\t]
		plot ({( 0.3 + \ampl*cos(\nwaves*\t)) * cos(\t)},
		{( 0.3 + \ampl*cos(\nwaves*\t)) * sin(\t)});
		\draw[dashed, line width=0.5mm] (0,0) -- (1.4142, 1.4142);
		\draw[dashed, line width=0.5mm] (0,0) -- (-1.4142, 1.4142);
		\draw[dashed, line width=0.5mm] (0,0) -- (2,0);
		\draw [line width=0.5mm, ->] (0.6,0) arc[start angle=0, end angle=45, radius=0.6];
		\node at (0.8, 0.3) {$45^\circ$};
		\node at (1.2, -0.2) {$R$};
		\draw[dashed, line width=0.5mm] (0.99,-0.99) --(0, 0);
		\node at (0.4, -0.6) {$D$};
		\node at (0, 2.2) {Heated wall};
		\node at (0, -2.2) {Cold wall};
		\node at (6,0) 	{\includegraphics[height=5cm, width=6.5cm]{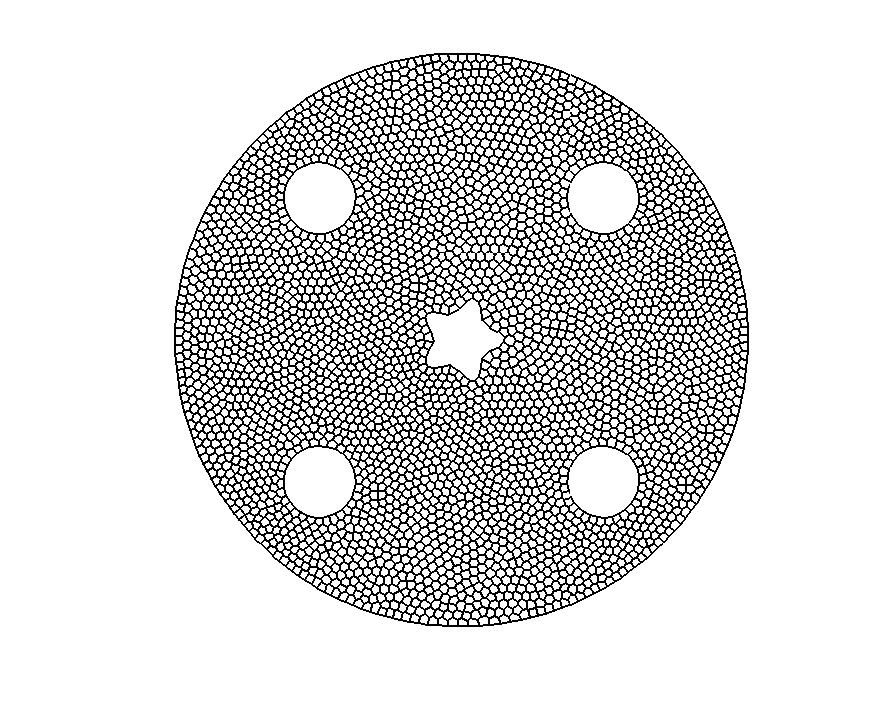}};
	\end{tikzpicture}
	\caption{Example 5. Domain geometry with boundary conditions.}
	\label{island}
\end{figure}

\subsection{Example 5: Natural convection in a circular enclosure with multiple obstacles}
\label{case5}
Finally, we discuss the natural convection in a two-dimensional slice of a shell-and-tube geometry, motivated by \cite{gharibi2025mixed}, as shown in Figure \ref{island}. The numerical experiment is conducted employing the dimensionless equation  \eqref{eq:main_system} presented in Section \ref{case4}. For computational domain $\Omega$, we consider a circular cylinder of radius $R=2$, which contains a sequence of four circular cylinders with equal radii $0.25$ and a star-shaped cylinder with variable radius $r(\phi)= 0.25+0.05\cos(5 \phi)$ where $\phi \in [0^\circ,360^\circ]$. The star-shaped cylinder is placed at the origin and $D=1.4$. Therefore, the position of the rest of the cylinders can be easily determined using Figure \ref{island}. The wall of the outer cylinder is decomposed into two parts: a heated wall and a cold wall. The outer heated wall is maintained with a Dirichlet condition $\theta_D = 1$, while the cold wall is kept with $\theta_D = 0$. The inner top two cylinders are kept at $\theta_D = 0$, whereas the bottom three are maintained with $\theta_D = 1$. Moreover, the non-slip velocity conditions are imposed for the velocity field on all boundaries. Concerning \eqref{eq:main_system}, we choose $\text{Pr} = 6.2$ and $\mathbf{g}=(0,1)^T$. We investigate the performance of the proposed method for Problem~\eqref{eq:main_system} under varying Rayleigh numbers, specifically ${Ra} =10^2$, $10^3$, $10^4$, and $4 \times 10^4$. 

\begin{figure}[h]
	\centering
	{\includegraphics[height=4cm, width=5cm]{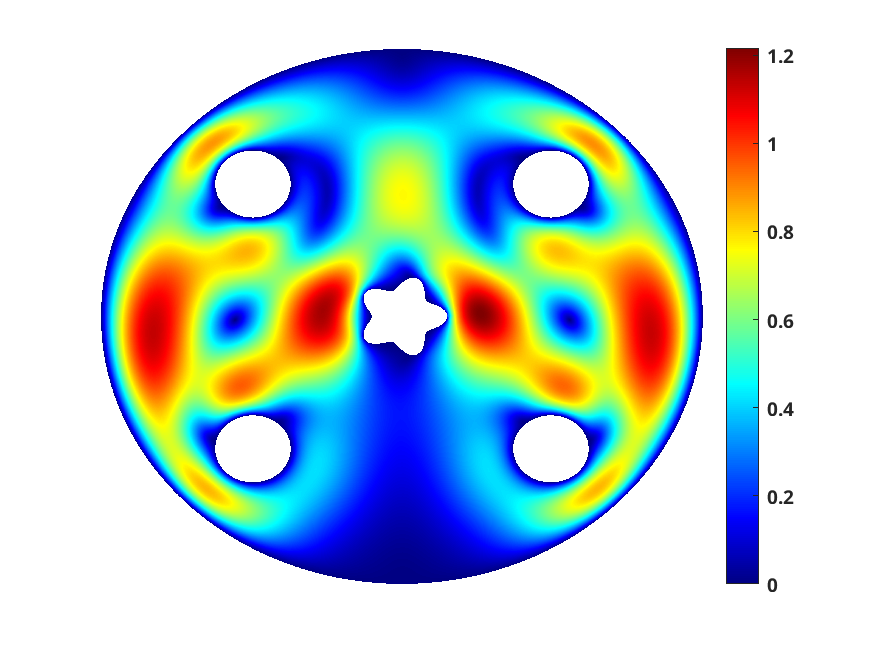}}
	{\includegraphics[height=4cm, width=5cm]{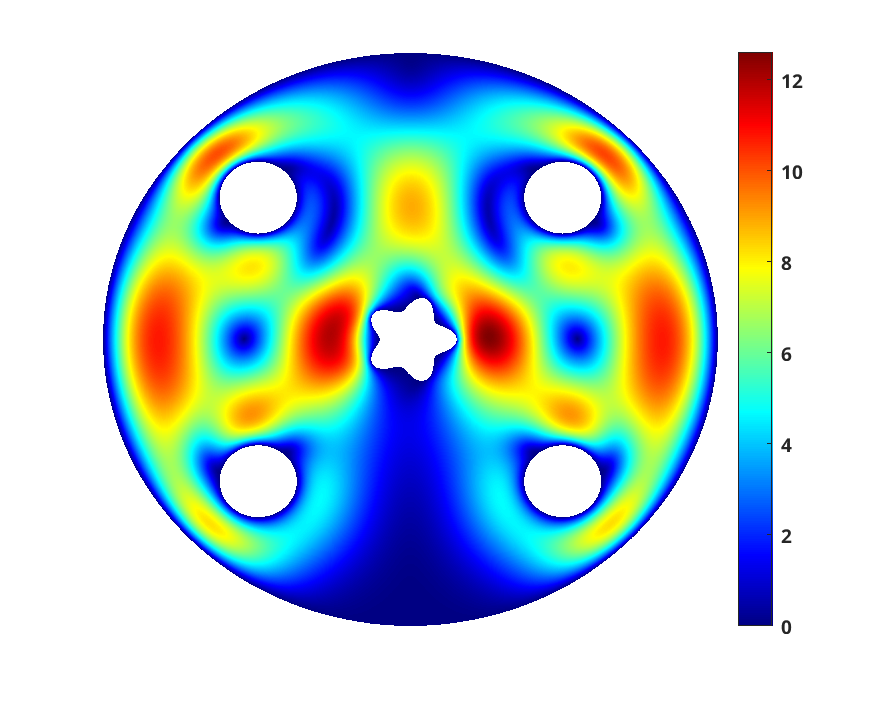}}
	{\includegraphics[height=4cm, width=5cm]{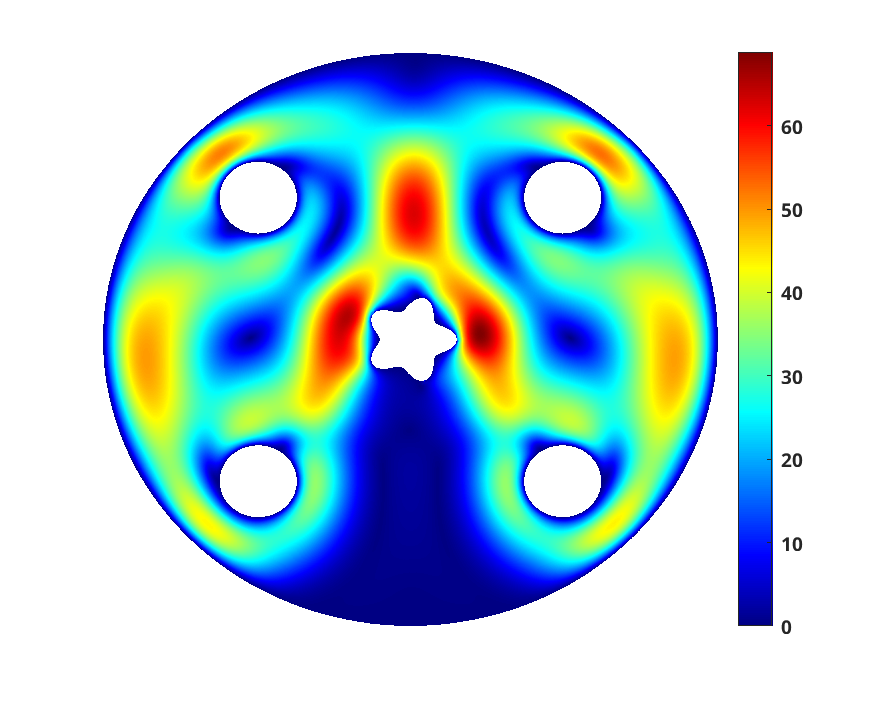}} \\
	{\includegraphics[height=4cm, width=5cm]{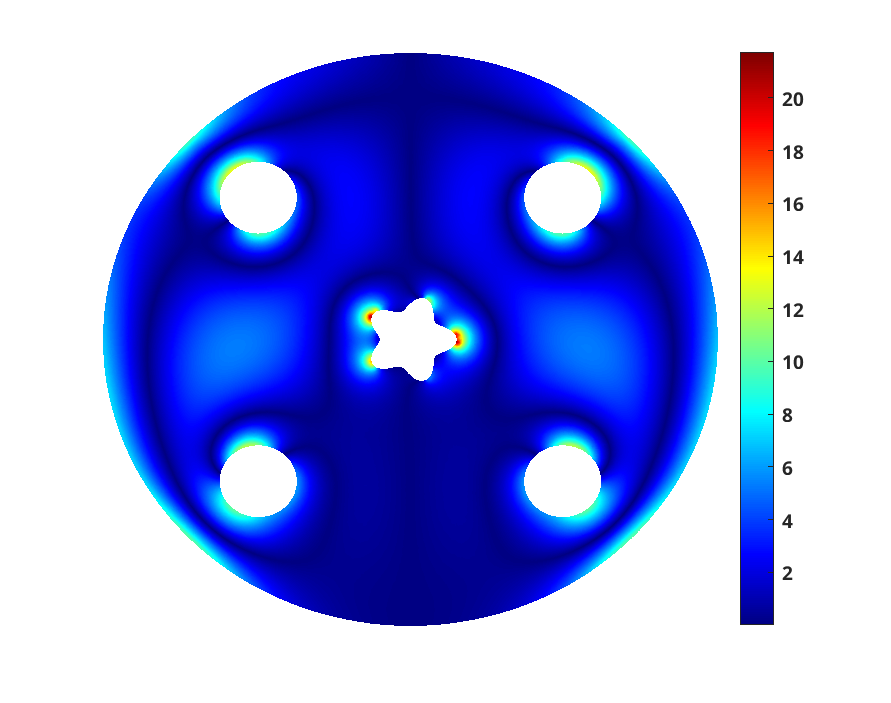}}
	{\includegraphics[height=4cm, width=5cm]{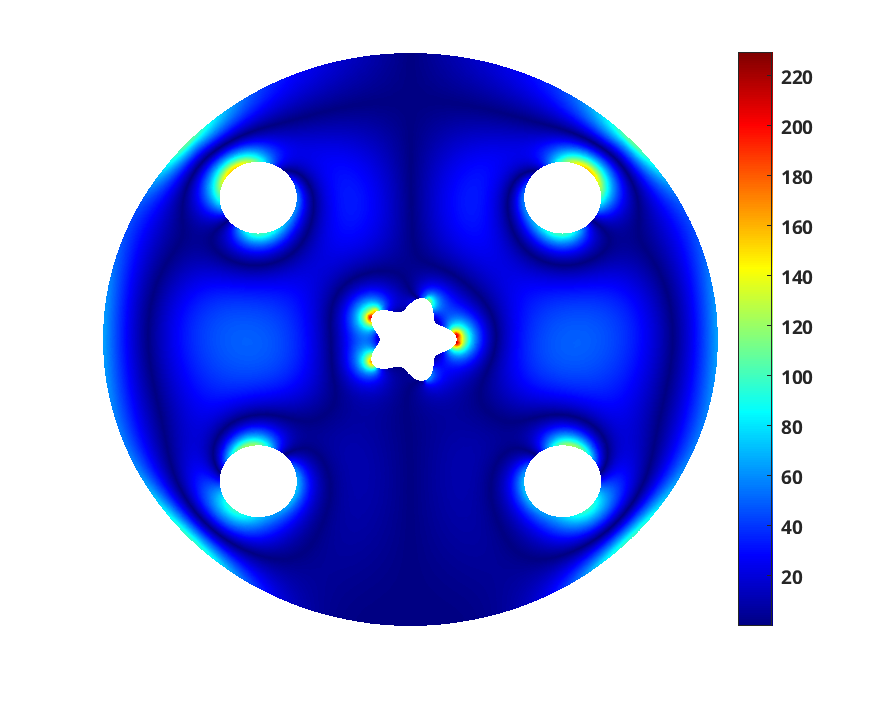}}
	{\includegraphics[height=4cm, width=5cm]{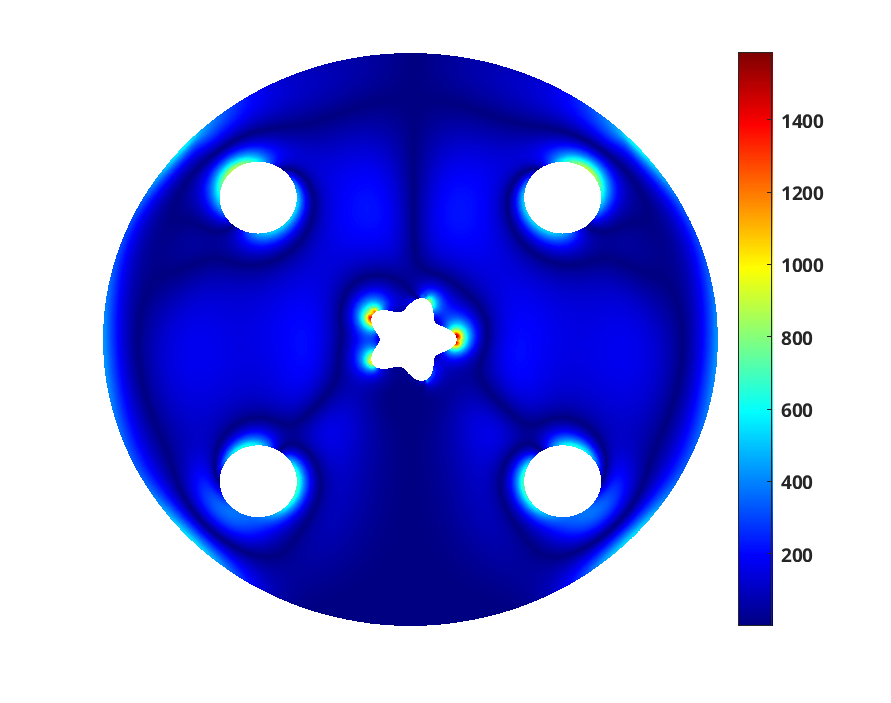}} \\
	{\includegraphics[height=4cm, width=5cm]{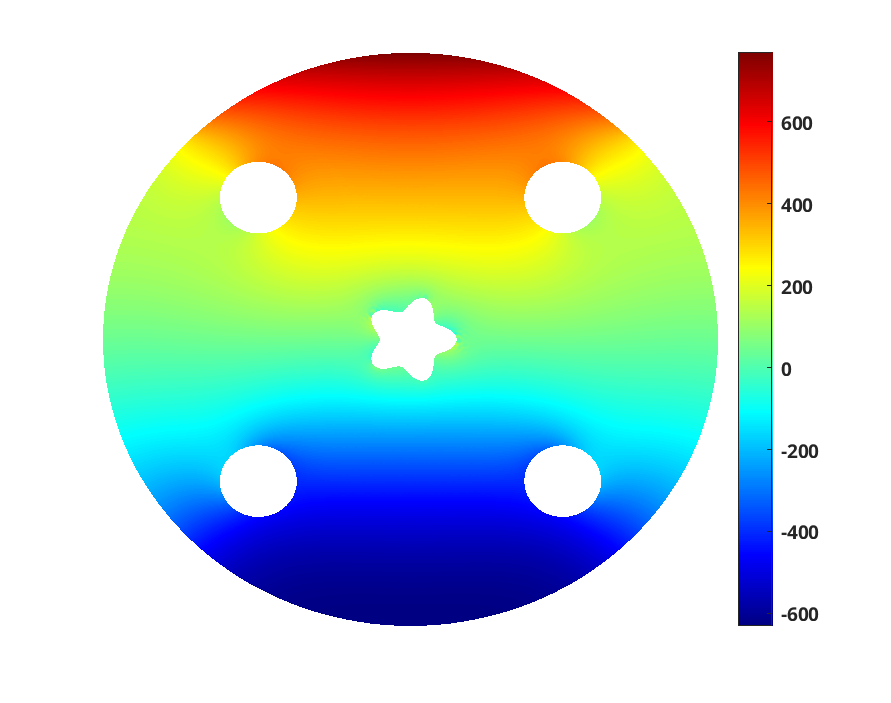}}
	{\includegraphics[height=4cm, width=5cm]{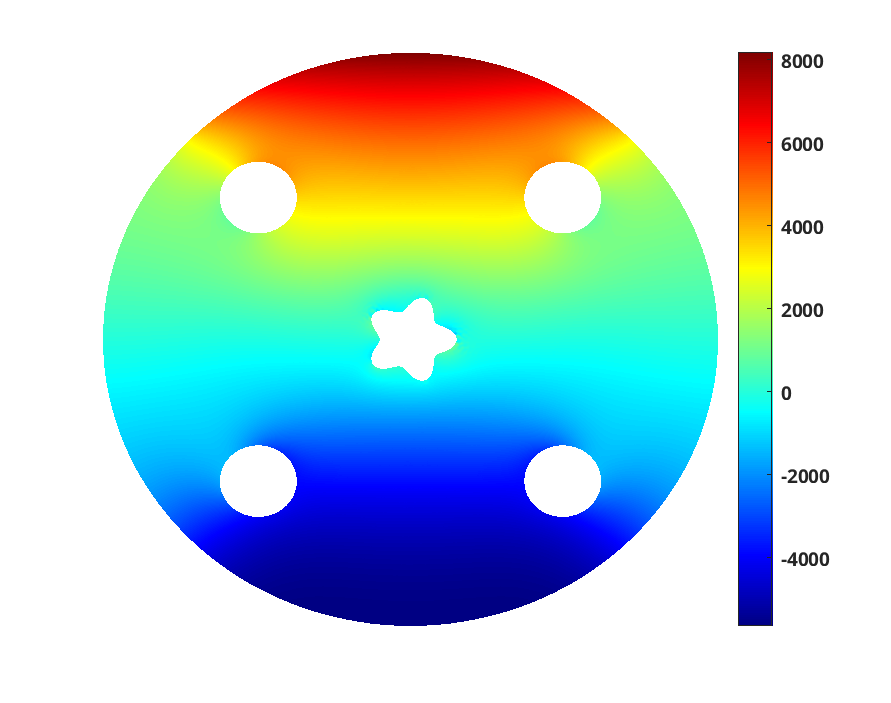}}
	{\includegraphics[height=4cm, width=5cm]{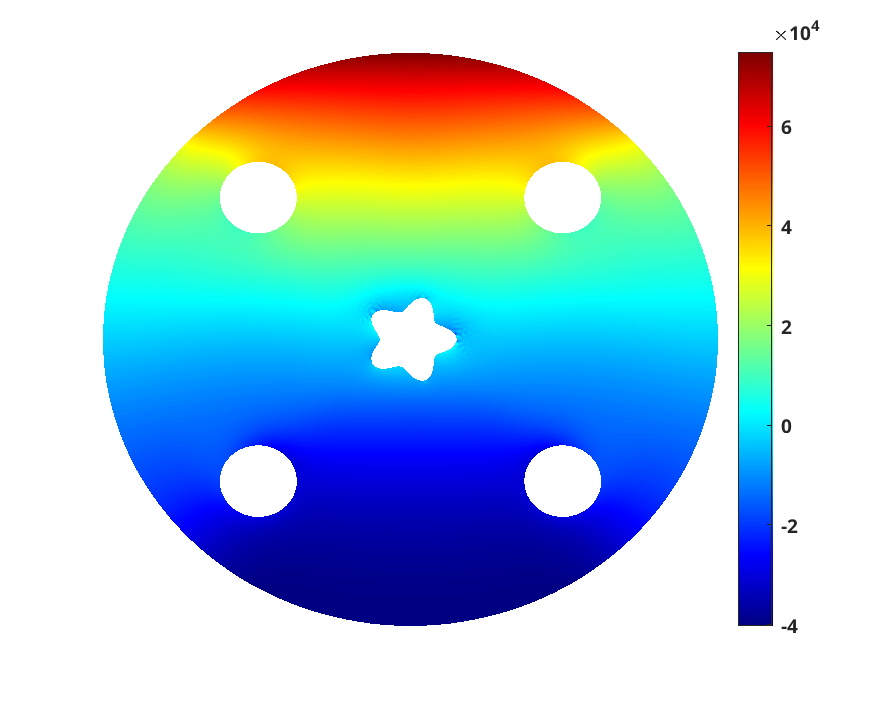}} \\
	{\includegraphics[height=4cm, width=5cm]{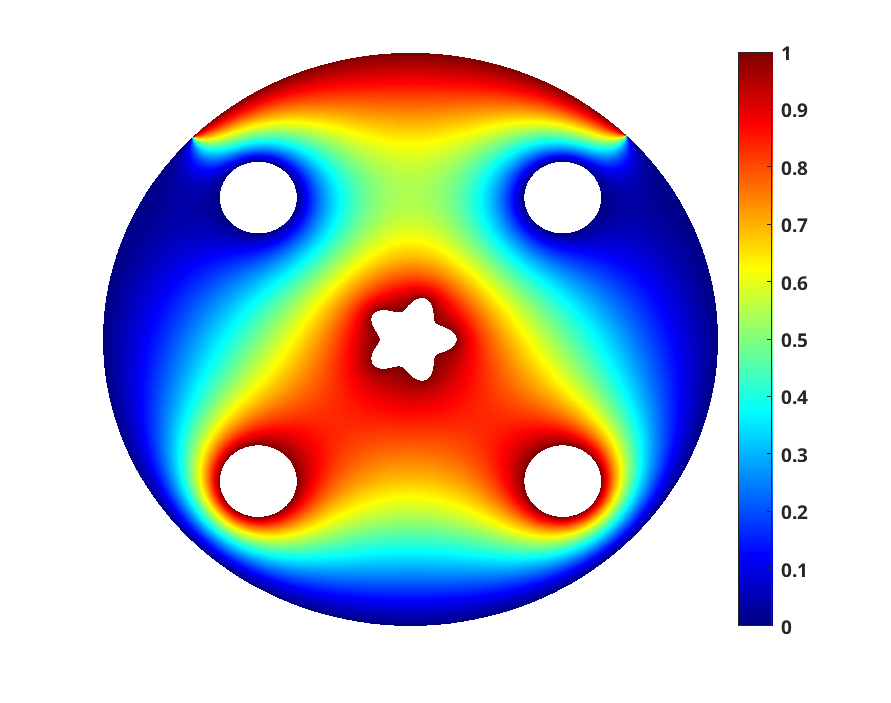}}
	{\includegraphics[height=4cm, width=5cm]{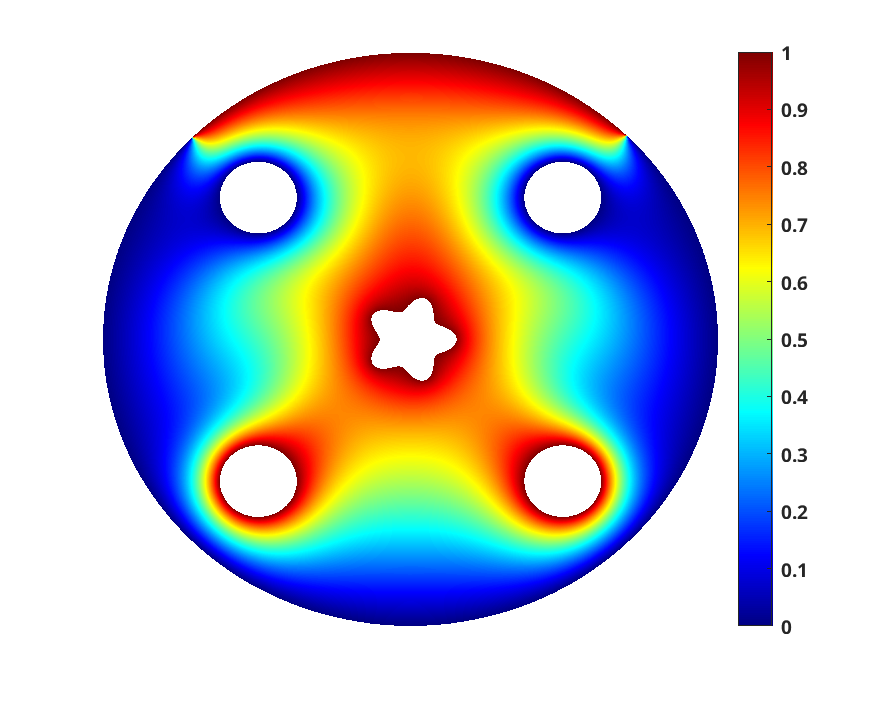}}
	{\includegraphics[height=4cm, width=5cm]{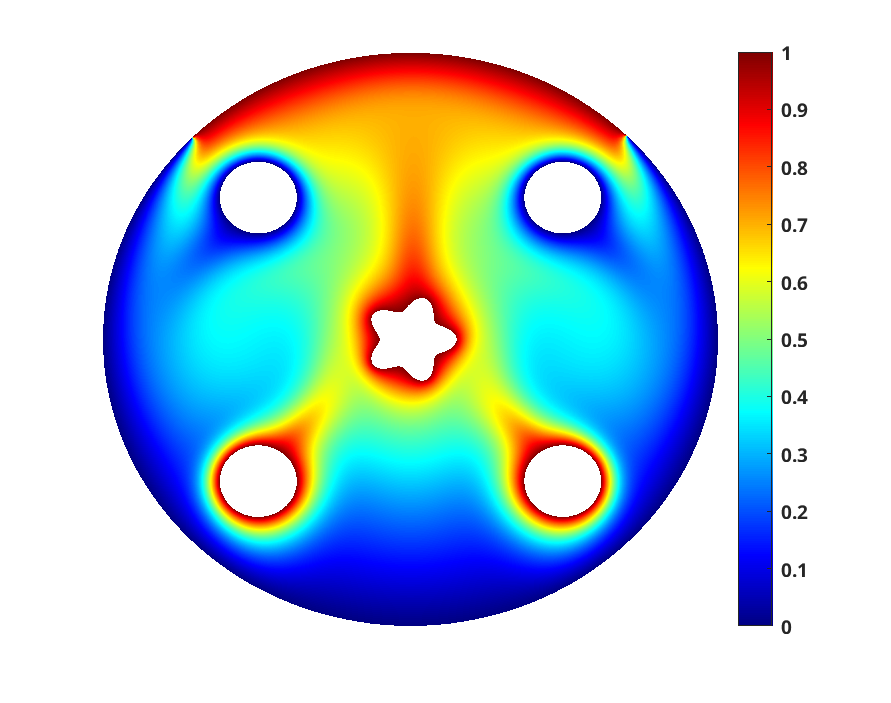}} 
	\caption{ Example 5. Snapshots of velocity and vorticity magnitudes, pressure, along with the temperature (from top to bottom, respectively) for values ${Ra} =10^2, 10^3, 10^4$ (first to third columns, respectively).}
	\label{ex5sol1} 
\end{figure}

The numerical results are computed on {Voronoi mesh shown in Figure \ref{island} with $35000$ elements.} The stabilized discrete velocity and vorticity magnitudes, pressure, as well as the stabilized temperature distributions are shown in Figures \ref{ex5sol1} and \ref{ex5sol2}, for each of the specified Rayleigh numbers. As evident from Figures \ref{ex5sol1} and \ref{ex5sol2}, the flow remains smooth and predictable at low Rayleigh numbers due to weak convective effects. However, as ${Ra}$ increases, stronger circulatory behavior emerges, resulting in more turbulent flow patterns. Correspondingly, the temperature distribution evolves from a regular, symmetric profile at ${Ra} = 10^2$ to increasingly irregular configurations for higher Rayleigh numbers, with distinct thermal plumes developing between the hot and cold regions.

\begin{figure}[h]
	\centering 
	\subfloat[Velocity magnitude.]{\includegraphics[height=4cm, width=5cm]{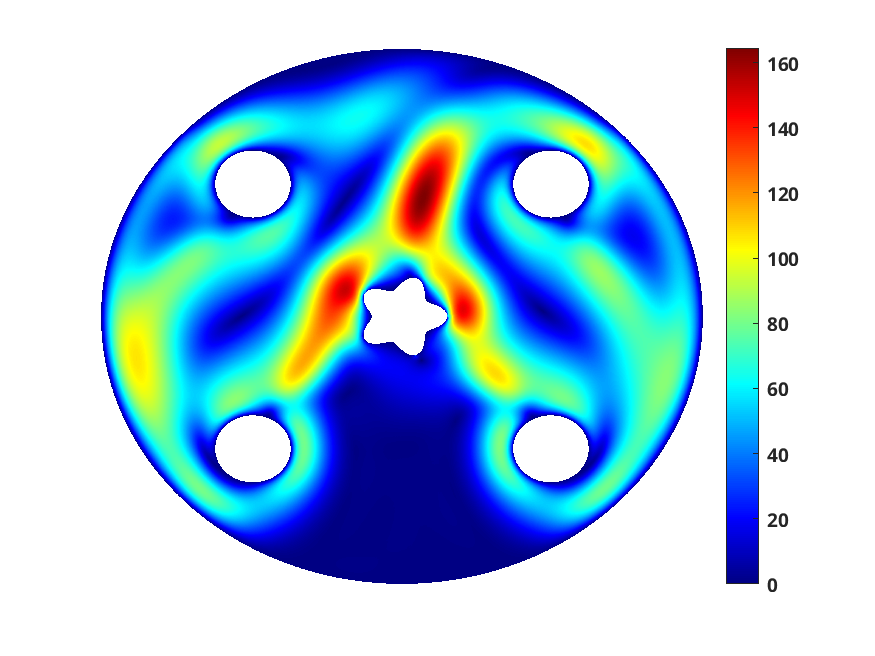}}~~~~ 
	\subfloat[Vorticity magnitude.]{\includegraphics[height=4cm, width=5cm]{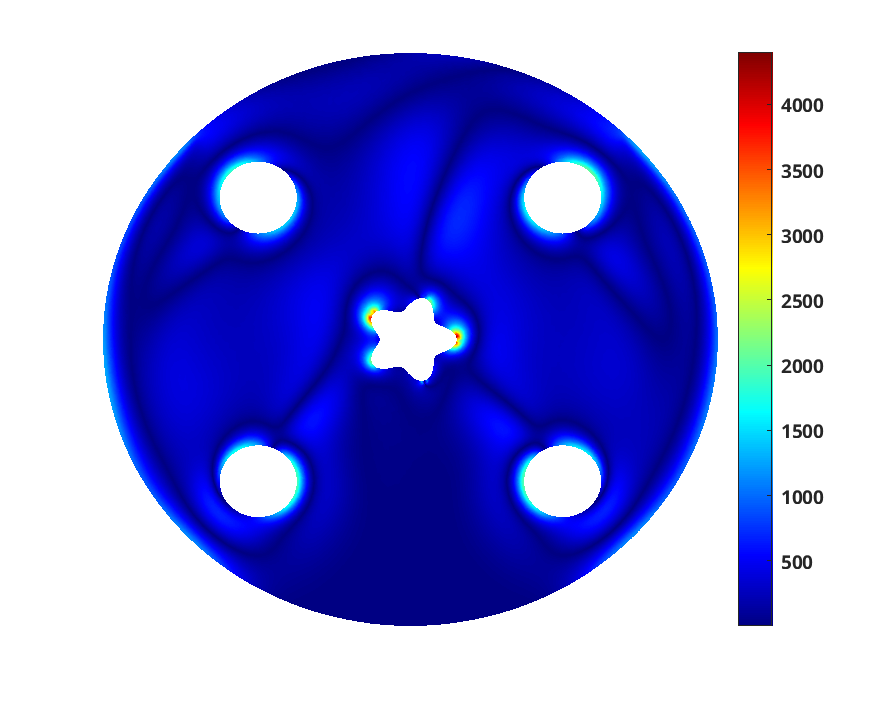}}\\
	\subfloat[Pressure.]{\includegraphics[height=4cm, width=5cm]{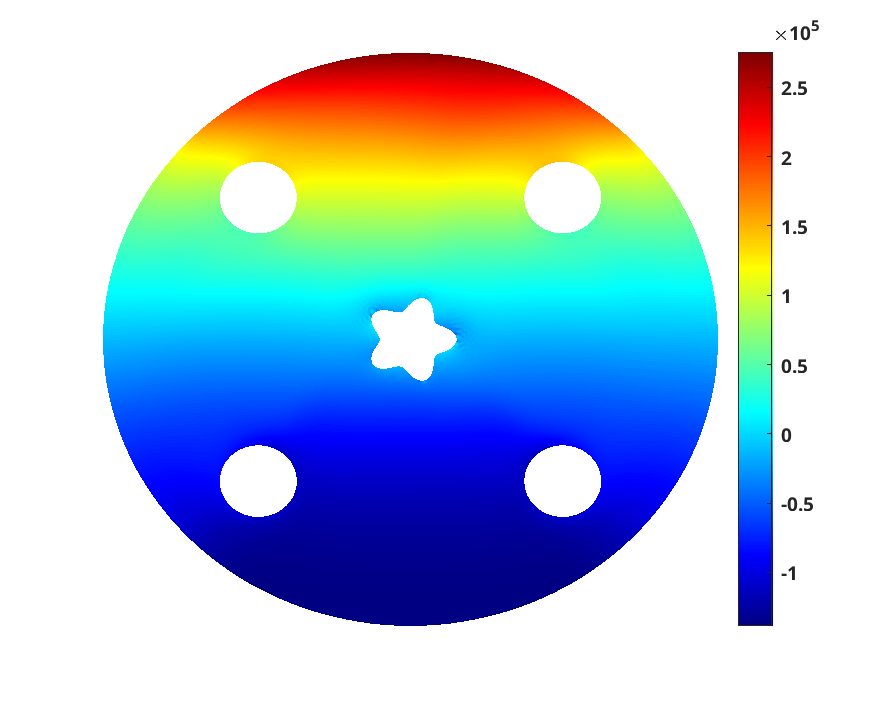}} ~~~~
	\subfloat[Temperature.]{\includegraphics[height=4cm, width=5cm]{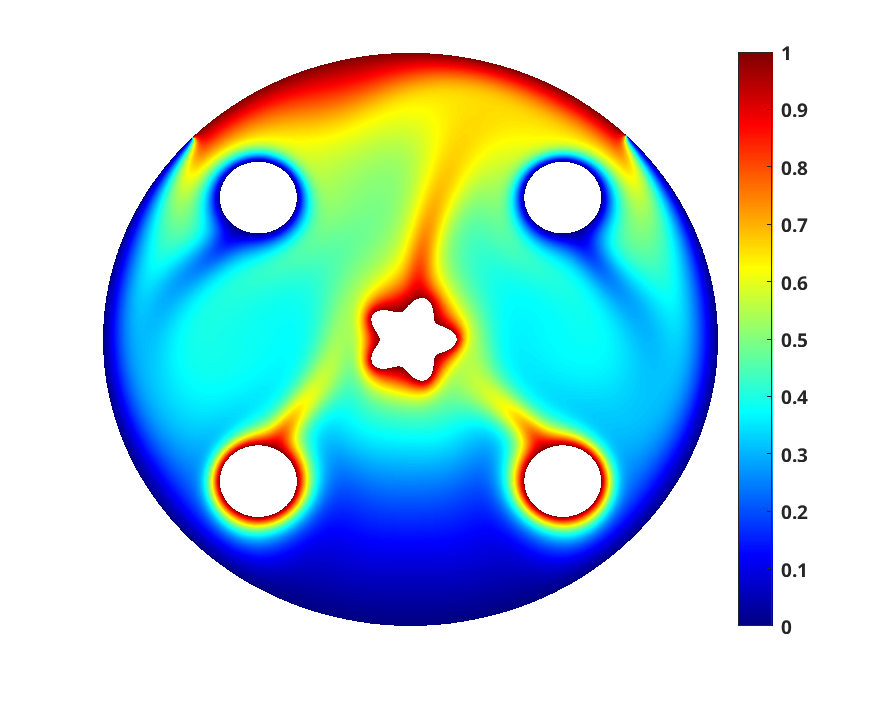}} 
	\caption{ Example 5. Snapshots of discrete solutions for ${Ra} = 4 \times 10^4$.}
	\label{ex5sol2} 
\end{figure}

\section{Conclusion} \label{sec-6}
In this work, we have developed a unified framework of the stabilized virtual element method for the generalized Boussinesq equation on polygonal meshes with temperature-dependent viscosity and thermal conductivity. To address the violation of the discrete inf-sup condition, a gradient-based local projection stabilization term is introduced in the discrete formulation. We derive the well-posedness of the continuous problem under sufficiently small datum. An equivalence relation is shown between the gradient-based and mass-based local projection stabilization methods. The well-posedness of the stabilized virtual element problem is established using the Brouwer fixed-point theorem. The error estimates are derived in the energy norm with optimal convergence rates. Numerical results confirm the theoretical findings. 

\subsection*{Data Availability}
Data sharing is not applicable to this article as no datasets were generated or analyzed.

\section*{Declarations}
\subsection*{Conflict of interest} The authors declare no conflict of interest during the current study.

\section*{Appendix}
\noindent \textbf{Proof of Lemma \ref{pressure}.}
Using the definition of $S^E_{0,k-1}(\cdot, \cdot)$, {the bound \eqref{vem-d}} and the Poincar$\acute{\text{e}}$ inequality, we infer
\begin{align}
	\sum_{E \in \Omega_h } S^E_{0,k-1}\big(q_h,q_h\big):&= \sum_{E \in \Omega_h } S^E((I-\Pi^{0,E}_{k-1})q_h, (I-\Pi^{0,E}_{k-1})q_h ) \leq {\lambda^\ast_4}\sum_{E \in \Omega_h }  \|(I-\Pi^{0,E}_{k-1})q_h\|^2_{0,E} \nonumber \\
	& \leq {\lambda^\ast_4} \sum_{E \in \Omega_h }  \|(I-\Pi^{0,E}_{k-1})(I-\Pi^{\nabla ,E}_{k-1})q_h\|^2_{0,E} \nonumber \\
	& \leq {\lambda^\ast_4} \sum_{E \in \Omega_h } C^2_p h^2_E \|\nabla (I-\Pi^{\nabla,E}_{k-1})q_h\|^2_{0,E} \nonumber \\
	& \leq { \frac{\lambda^\ast_4 C^2_p}{\lambda_{3\ast}} \sum_{E \in \Omega_h }  \tau_{2,E} S^E_p \big( q_h-\Pi^{\nabla,E}_{k-1}q_h, q_h-\Pi^{\nabla,E}_{k-1}q_h \big).}  \label{equi1}
\end{align}
Concerning the first part of $\mathcal{L}^\ast_{2,h}(\cdot, \cdot)$, we proceed as follows
\begin{align}
	\sum_{E \in \Omega_h } \|\Pi_k^{0,E} q_h - \Pi^{0,E}_{k-1} q_h\|^2_{0,E} &\leq 2 \sum_{E \in \Omega_h } \big( \| q_h - \Pi^{0,E}_k q_h\|^2_{0,E} + \| q_h - \Pi^{0,E}_{k-1} q_h\|^2_{0,E} \big) \nonumber \\
	&\leq 2 \sum_{E \in \Omega_h } \big( \| q_h - \Pi^{\nabla,E}_k q_h\|^2_{0,E} + \| q_h - \Pi^{\nabla,E}_{k-1} q_h\|^2_{0,E} \big) \nonumber \\
	&\leq 2 \sum_{E \in \Omega_h } C_p^2 h^2_E\big( \| \nabla (q_h - \Pi^{\nabla,E}_k q_h)\|^2_{0,E} + \| \nabla (q_h - \Pi^{\nabla,E}_{k-1} q_h) \|^2_{0,E} \big) \nonumber \\
	&\leq 2 \sum_{E \in \Omega_h } C_p^2 h^2_E \big( \| \nabla q_h - \boldsymbol{\Pi}^{0,E}_{k} \nabla q_h\|^2_{0,E} + \| \nabla q_h - \boldsymbol{\Pi}^{0,E}_{k-1} \nabla q_h \|^2_{0,E} \big) \nonumber \\
	&\leq 6 \sum_{E \in \Omega_h } C_p^2 h^2_E \big( \| \boldsymbol{\Pi}^{0,E}_k \nabla q_h - \boldsymbol{\Pi}^{0,E}_{k-1} \nabla q_h\|^2_{0,E} + \| \nabla q_h - \nabla \Pi^{\nabla,E}_{k-1} q_h \|^2_{0,E} \big). \label{equi2}
\end{align}
Now adding \eqref{equi1} and \eqref{equi2} and using $\tau_{2,E} \sim h^2_E$, we have the following 
\begin{align}
	\mathcal{L}^\ast_{2,h}(q_h, q_h) \leq \alpha_2 \mathcal{L}_{2,h}(q_h, q_h), \label{fpart}
\end{align}
where $\alpha_2$ depends on the Poincar$\acute{\text{e}}$ constant { $C_p$, $\lambda_{3\ast}$ and $\lambda_4^\ast$}. For the remaining part, we use the definition of {$S^E_p(\cdot, \cdot)$, the bound \eqref{vem-c} and $\tau_{2,E} \sim h^2_E$}:
\begin{align}
	\sum_{E \in \Omega_h } \tau_{2,E} S^E_p \big(q_h - \Pi^{\nabla,E}_{k-1} q_h, q_h - \Pi^{E,\nabla}_{k-1} q_h \big) &\leq {\lambda^\ast_3}\sum_{E \in \Omega_h } h^2_E \|\nabla (I- \Pi^{\nabla,E}_{k-1})q_h\|^2_{0,E} \nonumber \\
	&\leq {\lambda^\ast_3}\sum_{E \in \Omega_h } h^2_E  \| \nabla (I- \Pi^{0,E}_{k-1})q_h\|^2_{0,E}  \quad (\text{using definition of $\Pi^{\nabla,E}_{k-1}$})\nonumber \\
	&\leq {\lambda^\ast_3} \sum_{E \in \Omega_h }  C^2_{inv} \| (I- \Pi^{0,E}_{k-1})q_h\|^2_{0,E}. \qquad(\text{using Lemma \ref{inverse}}) \nonumber \\
	&\leq {\frac{\lambda^\ast_3 C_{inv}^2}{\lambda_{4\ast}} \sum_{E \in \Omega_h }  S^E_{0,k-1} \big(q_h,q_h\big).}\label{equi3}
\end{align}
We now use the triangle inequality, the definition of $L^2$-projector and Lemma \ref{inverse}:
\begin{align}
	\sum_{E \in \Omega_h } h^2_E  \| \boldsymbol{\Pi}^{0,E}_k \nabla q_h - \boldsymbol{\Pi}^{0,E}_{k-1} \nabla q_h\|^2_{0,E} &\leq 2 \sum_{E \in \Omega_h } h^2_E \big( \| \nabla q_h - \boldsymbol{\Pi}^{0,E}_{k} \nabla q_h\|^2_{0,E} + \| \nabla q_h - \boldsymbol{\Pi}^{0,E}_{k-1} \nabla q_h\|^2_{0,E} \big) \nonumber\\
	&\leq 2 \sum_{E \in \Omega_h } h^2_E \big( \| \nabla q_h - \nabla {\Pi}^{0,E}_{k}  q_h\|^2_{0,E} + \| \nabla q_h - \nabla {\Pi}^{0,E}_{k-1} q_h\|^2_{0,E} \big) \nonumber \\
	&\leq 2 \sum_{E \in \Omega_h } C^2_{inv} \big( \| q_h - \Pi^{0,E}_{k}  q_h\|^2_{0,E} + \| q_h - \Pi^{0,E}_{k-1} q_h\|^2_{0,E} \big) \nonumber \\
	&\leq 6 \sum_{E \in \Omega_h } C^2_{inv} \big( \| \Pi^{0,E}_{k} q_h - \Pi^{0,E}_{k-1}  q_h\|^2_{0,E} + \| q_h - \Pi^{0,E}_{k-1} q_h\|^2_{0,E} \big). \label{equi4}
\end{align} 
Combining \eqref{equi3} and \eqref{equi4}, we obtain
\begin{align}
	\alpha_1 \mathcal{L}_{2,h}(q_h,q_h) \leq \mathcal{L}^\ast_{2,h}(q_h,q_h), \label{spart}
\end{align}
where the constant $\alpha_1$ depends on $C_{inv}$, {$\lambda^\ast_3$ and $\lambda_{4\ast}$}. Thus, the proof of Lemma \ref{pressure} completes from \eqref{fpart} and \eqref{spart}. 



		\bibliographystyle{plain}
		\bibliography{references}

	\end{document}